\documentclass[a4paper, 12pt]{article}

\usepackage[utf8]{inputenc} %pour avoir un bon codage d'entrée (caractères accentués). utf-8 c'est le système de codage
\usepackage{lmodern} %police : famille de fonts LatinModern
\usepackage[T1]{fontenc} %codage latin
\usepackage{amssymb, amsmath, mathrsfs, mathabx, bbm} %inclusion de différentes types de fonts
\DeclareMathAlphabet{\mathbbold}{U}{bbold}{m}{n}%pour utiliser le style \mathbb{} avec les lettres greques. Il faut utiliser la commande \mathbbold{}
\usepackage[normalem]{ulem}

\usepackage{titling} %inclusion du titre, auteur et resume comme dans les articles scientifiques des revues
\usepackage{titlesec} %pour modifier les parametres des sections, chapitres, etc.
		\titlelabel{\thetitle.\quad} %ajouter un point apres le compteur des sections
		\titleformat*{\section}{\center\large} %on met les titres des section dans le centre de la page
		\titleformat*{\subsection}{\sf\large} %style de la police pour les subsections
		\titleformat*{\subsubsection}{\sf\it} %style de la police pour les subsubsections
		
\usepackage[colorlinks=true,linkcolor=black, citecolor=GreenCite,urlcolor=GreenCite]{hyperref} %inclusion de liens sur internet ou des liens dans le document lui-même. Il faut dire que l'inclusion de ce paquet implique automatiquement la hiperliens dedans le document, en particulier on peut cliquer sur les titres de la table des matières ou dans les reférences aux formules ou aux livres de la bibliographie. On écrit \url{direction web} ou \href{direction web}{NOM alternative pour le lien} Les paramètres indiqués signifient que :
\usepackage[usenames,dvipsnames]{color} %inclusion de plusieurs couleurs à part des standars (blue, yellow, black, white, magenta, cyan, green, red)
\definecolor{ToDo}{RGB}{30,144,255}
\definecolor{Strategy}{RGB}{34,139,34}
\definecolor{Provisional}{RGB}{218,165,32}
\definecolor{Question}{RGB}{220,20,60}
\definecolor{GreenCite}{RGB}{47, 79, 79}

\usepackage{verbatim} % commentaires
\usepackage[all]{xy} %inclusion de diagrammes (fleches, etc.)
\usepackage{tikz-cd}
\usepackage[usenames,dvipsnames]{color} %inclusion de plusieurs couleurs à part des standars (blue, yellow, black, white, magenta, cyan, green, red)
\usepackage{enumerate} %configuration des listes
\usepackage{leftidx}
\usepackage[makeroom]{cancel}

\numberwithin{equation}{section} %on numérote les équations selon la section où on se trouve
\usepackage{array} %inclusion de tableaux dans le mode mathématicien
\usepackage{mathtools} %differents tools when writing maths, like brace or brackets under the formulas.
\usepackage{amsthm} %inclusion des formats pour les théorèmes, remarques, etc.
\swapnumbers %on numérote les énoncés de la forme "2.0 Theorem" au lieu de "Theorem 2.0"
\theoremstyle{plain}
\theoremstyle{remark}

\theoremstyle{plain}
\newtheorem{theo}{Theorem}[subsection] %enumération à niveau de sous-section
\newtheorem{theodefi}[theo]{Theorem-Definition}
\newtheorem{lem}[theo]{Lemma}

\newtheorem{pro}[theo]{Proposition}
\newtheorem{prodefi}[theo]{Proposition-Definition}
\newtheorem{cor}[theo]{Corollary}
\newtheorem{defi}[theo]{Definition}
\newtheorem{rem}[theo]{Remark}
\newtheorem{rems}[theo]{Remarks}
\newtheorem{ex}[theo]{Example}
\newtheorem{exs}[theo]{Examples}
\theoremstyle{remark}
\newtheorem{note}[theo]{Note}

\theoremstyle{plain}
\theoremstyle{remark}

\makeatletter
\let\@fnsymbol\@alph
\makeatother
\title{\textbf{Projective representation theory for compact quantum groups and the quantum Baum-Connes assembly map}}
\author{\textsc{Kenny De Commer\thanks{Department of Mathematics and Data Science, Vrije Universiteit Brussel (Belgium). The work of K.DC. was supported by the FWO grant G032919N.}, Rubén Martos\thanks{Department of Mathematical Sciences, University of Copenhagen (Denmark). R.M. is supported by the European Union's Horizon 2020 research and innovation programme under the Marie Skłodowska-Curie grant agreement No 895141.} \& Ryszard Nest\thanks{Department of Mathematical Sciences, University of Copenhagen (Denmark).}}}
\date{}
\begin{document}
\maketitle
\renewcommand{\abstractname}{}
\vspace{-2.5cm}
\begin{abstract}
\textsc{Abstract}. We study the theory of measurable projective representations for a compact quantum group $\mathbb{G}$, i.e.\ actions of $\mathbb{G}$ on $\mathcal{B}(H)$ for some Hilbert space $H$. We show that any such measurable projective representation is inner, and is hence induced by an $\Omega$-twisted representation for some measurable $2$-cocycle $\Omega$ on $\mathbb{G}$. We show that a projective representation is \emph{continuous}, i.e.\ restricts to an action on the compact operators $\mathcal{K}(H)$, if and only if the associated $2$-cocycle is regular, and that this condition is automatically satisfied if $\mathbb{G}$ is of Kac type. This allows in particular to characterise the torsion of projective type of $\widehat{\mathbb{G}}$ in terms of the projective representation theory of $\mathbb{G}$. For a given regular $2$-cocycle $\Omega$, we then study $\Omega$-twisted actions on C$^*$-algebras. We define deformed crossed products with respect to $\Omega$, obtaining a twisted version of the Baaj-Skandalis duality and the Green-Julg isomorphism, and a quantum version of the Packer-Raeburn's trick. %As an application, we revisit a recent result of Arano and Skalski on the twisted Green-Julg isomorphism for the specific class of permutation torsion-free discrete quantum groups.

%Nous faisons une étude détaillée de la théorie des représentations projectives d'un groupe quantique compact $\mathbb{G}$. D'une part, nous définissons un produit croisé déformé par le $2$-cocycle correspondant et nous obtenons ainsi une version déformée de la dualité de Baaj-Skandalis. Comme conséquence, nous obtenons également une version quantique du connu \emph{Packer-Raeburn’s untwisting trick}. D'autre part, nous addressons des questions ouvertes autour des propriétés de cleftness pour $\mathbb{G}$ (apparaissant dans un travail antérieur du premier auteur). Ainsi, nous démontrons que tout co-object de Galois d'un groupe quantique compact de type Kac est von Neumann cleft. Ceci nous permet de donner une correspondance bijective entre les actions de torsion de type projective et les representations projectives de $\mathbb{G}$. Comme application, nous généralisons la relation d'adjunction de l'isomorphisme de Green-Julg en obtenant un morphisme d'assemblage de Baum-Connes pour les groupes quantiques discrets sans torsion de type permutation.

\end{abstract}

\newpage

\tableofcontents

\section{\textsc{Introduction}}
Let $G$ be a compact group. Let $H$ be a (possibly infinite dimensional) Hilbert space, and $\mathcal{K}(H)$ the C$^*$-algebra of compact operators on $H$. An action  $G\overset{\delta}{\curvearrowright}\mathcal{K}(H)$ is usually referred to as a \emph{continuous projective representation of $G$ on $H$}. Such a $\delta$  is always of the form $\delta_g(T)=\pi(g)T\pi(g)^*$ for all $g\in G$ and $T\in \mathcal{K}(H)$, where $\pi: G\longrightarrow \mathcal{U}(H)$ is a measurable map and $\omega: G\times G\longrightarrow S^1 \subseteq \mathbb{C}$ is a measurable $2$-cocycle such that $\pi(x)\pi(y)=\omega(x,y)\pi(xy)$, for all $x,y\in G$. Such a map $\pi$ is called a \emph{$\omega$-representation of $G$ on $H$}. One can similarly consider actions $G\overset{\delta}{\curvearrowright}\mathcal{B}(H)$ of $G$ on the von Neumann algebra $\mathcal{B}(H)$, but this does not give anything new: any such action necessarily restricts to $\mathcal{K}(H)$, and conversely any action of $G$ on $\mathcal{K}(H)$ extends to $\mathcal{B}(H)$. 

An extension of projective representation theory to compact \emph{quantum} groups was introduced in \cite{KennyProjective} by the first author. For $\mathbb{G}$ a compact quantum group, a measurable projective representation was introduced as an action $\mathbb{G}\overset{\delta}{\curvearrowright}\mathcal{B}(H)$. It was then shown that the obstruction for $\delta$ to be inner is related to the theory of \emph{Galois co-objects} for $\mathbb{G}$, which are regarded as generalized $2$-cocycle functions on $\mathbb{G}$.  Particular instances of Galois co-objects can be defined in terms of (measurable) $2$-cocycles $\Omega$ on $\mathbb{G}$, extending the classical setting described above. Such Galois co-objects are called \emph{von Neumann cleft}. It was left open in \cite{KennyProjective} whether there exist Galois co-objects which are not von Neumann cleft. We completely resolve this problem in this article:  a Galois co-object for a compact quantum group is \emph{automatically} cleft. We can hence restrict to projective representations defined through a measurable $2$-cocycle $\Omega$. See Section \ref{sec.ProjRepCQG}. For the sake of completeness of the present article, we have included an explicit development of the corresponding $\Omega$-representation theory for $\mathbb{G}$ in Section \ref{sec.omreps}. 

In spite of the above positive result, there is however an important new phenomenon arising for projective representations of compact quantum groups: Contrary to the case of classical compact groups, an action of a compact quantum group $\mathbb{G}$ on $\mathcal{B}(H)$ does \emph{not} automatically restrict to a continuous action on the C$^*$-algebra of compact operators $\mathcal{K}(H)$. If this is the case, we call the projective representation \emph{continuous}, and the associated $2$-cocycle \emph{of finite type}. In general, not all $2$-cocycles are of finite type. We show however that if $\mathbb{G}$ is of \emph{Kac type}, then all projective representations are continuous, and all $2$-cocycles are of finite type. See Section \ref{sec.contprojrep}.

In \cite{SergeyDeformation2} (see also \cite{VaesBiproduit} for the von Neumann algebra setting, and \cite{QuiggTwistDuality} for the classical setting of locally compact groups), it was shown how a $2$-cocycle $\Omega$ on a \emph{locally} compact quantum group allows to form $\Omega$-twisted actions of $\mathbb{G}$ on C$^*$-algebras, in case the $2$-cocycle satisfies a particular \emph{regularity} condition. This then allows to define also twisted crossed products for which a twisted version of the \emph{Baaj-Skandalis duality} holds. We show that in the setting of compact quantum groups, regularity of $\Omega$ is equivalent to $\Omega$ being of finite type. We then revisit some of the results of \cite{SergeyDeformation2} in the technically simpler setting of compact quantum groups, and obtain in particular a quantum version of the \emph{Packer-Raeburn's trick} \cite{PackerRaeburnTrick1}. See Section \ref{sec.TwistedBaajSkand}.

The theory of projective representations for $\mathbb{G}$ is closely related to the theory of \emph{torsion} for the dual (discrete) quantum group $\hat{\mathbb{G}}$. A general theory of torsion for discrete quantum groups was introduced first by R. Meyer and R. Nest (see \cite{MeyerNestTorsion} and \cite{MeyerNestHomological2}) in terms of ergodic actions of $\mathbb{G}$, and re-interpreted later by Y. Arano and K. De Commer in terms of fusion rings and module C$^*$-categories \cite{YukiKenny}. Basically, the torsion of $\widehat{\mathbb{G}}$ is described by the non-trivial ergodic actions of $\mathbb{G}$ on finite-dimensional C$^*$-algebras. In case the latter is simple, we obtain an ergodic action of $\mathbb{G}$ on some $\mathcal{B}(H)$ for $H$ finite-dimensional, i.e.\ a finite-dimensional irreducible projective representation of $\mathbb{G}$. 

The study of such torsion was the original motivation for this article, more precisely with respect to the \emph{Baum-Connes conjecture for discrete quantum groups}. The original Baum-Connes conjecture for (second countable) locally compact groups had been formulated in 1982 by P. Baum and A. Connes. We still do not know any counterexample to the original conjecture, but it is known that the one with coefficients is false. For this reason we refer to the Baum-Connes conjecture with coefficients as the \emph{Baum-Connes property}. The goal of the conjecture is to understand the link between two operator $K$-groups of different nature that would establish a strong connection between geometry and topology in a more abstract and general index-theory context. More precisely, if $G$ is a (second countable) locally compact group and $A$ is a (separable) $G$-C$^*$-algebra, then the Baum-Connes property for $G$ with coefficients in $A$ claims that the assembly map $\mu^G_A: K^{top}_{*}(G; A)\longrightarrow K_{*}(G\underset{r}{\ltimes}A)$ is an isomorphism, where $K^{top}_{*}(G; A)$ is the equivariant $K$-homology with compact support of $G$ with coefficients in $A$ and $K_{*}(G\underset{r}{\ltimes}A)$  is the $K$-theory of the reduced crossed product $G\underset{r}{\ltimes}A$. This property has been proved for a large class of groups; let us mention the remarkable work of N. Higson and G. Kasparov \cite{HigsonKasparovHaagerup} about groups with Haagerup property and the one of V. Lafforgue \cite{Lafforgue} about hyperbolic groups.

The equivariant $K$-homology with compact support $K^{top}_{*}(G; A)$ is the geometrical object obtained from the classifying space of proper actions of $G$, thus it is, \emph{a priori}, easier to calculate than the group $K_{*}(G\underset{r}{\ltimes}A)$, which is the one of analytical nature and less flexible in its structure. Nevertheless, sometimes the group $K^{top}_{*}(G; A)$ creates non-trivial troubles. This is why R. Meyer and R. Nest provide in 2006 a new formulation of the Baum-Connes property in a well-suited category framework, using the language of triangulated categories and derived functors \cite{MeyerNest}. More precisely, if $\mathscr{K}\mathscr{K}^{G}$ is the $G$-equivariant Kasparov category and $F(A):=K_{*}(G\underset{r}{\ltimes}A)$ is the homological functor $\mathscr{K}\mathscr{K}^{G}\longrightarrow \mathscr{A}b^{\mathbb{Z}/2}$ defining the right-hand side of the Baum-Connes assembly map, they show that the assembly map $\mu^G_A$ is equivalent to the natural transformation $\eta^G_A: \mathbb{L}F(A)\longrightarrow F(A)$, where $\mathbb{L}F$ is the localisation of the functor $F$ with respect to an appropriated complementary pair of (localizing) subcategories $(\mathscr{L}_G,\mathscr{N}_G)$. Here $\mathscr{L}_G$ is the subcategory of $\mathscr{K}\mathscr{K}^{G}$ of \emph{compactly induced $G$-C$^*$-algebras}, and $\mathscr{N}_G$ is the subcategory of $\mathscr{K}\mathscr{K}^{G}$ of \emph{compactly contractible $G$-C$^*$-algebras}. We say that $G$ satisfies the \emph{strong} Baum-Connes property if $\mathscr{L}_G=\mathscr{K}\mathscr{K}^{G}$, which corresponds, in usual terms, to the existence of a $\gamma$-element that equals $\mathbbold{1}_{\mathbb{C}}$. This approach yields as well a characterization of the Baum-Connes property \emph{only} in terms of compact subgroups, $K$-theory and crossed products.

The above reformulation allows in particular to avoid any geometrical construction, and thus seems (at least in principle) better suited to apply also when $G$ is replaced by a locally compact \emph{quantum} group $\mathbb{G}$. For instance, this approach has already been implemented by R. Meyer and R. Nest \cite{MeyerNestTorsion} by proving that duals of compact connected\footnote{In an upcoming paper, the second author (together with P. Fima) has extended this result by removing the connectedness assumption.} groups satisfy the strong Baum-Connes property. Also, for genuine \emph{discrete quantum groups} $\widehat{\mathbb{G}}$ the strong Baum-Connes property has been studied, leading to explicit $K$-theory computations of the C$^*$-algebra $C(\mathbb{G})$ in remarkable examples: \cite{VoigtBaumConnesUnitaryFree, VoigtBaumConnesAutomorphisms, RubenAmauryTorsion}.  

A major problem when studying the quantum counterpart of the Baum-Connes property in the particular setting of discrete quantum groups is the \emph{torsion structure} of a discrete quantum group $\widehat{\mathbb{G}}$. Indeed, if $G$ is an ordinary discrete group, its torsion phenomena are completely described in terms of the finite subgroups of $G$ and encoded in the localizing subcategory $\mathscr{L}_G$ using the Meyer-Nest reformulation. More precisely, induction and restriction functors provide a pair of adjoint functors allowing to apply the general Meyer-Nest machinery to define the complementary pair encoding the Baum-Connes property. Hence if we want to follow the Meyer-Nest approach, we need for discrete quantum groups $\widehat{\mathbb{G}}$ an analogous complementary pair of localizing subcategories $(\mathscr{L}_{\widehat{\mathbb{G}}},\mathscr{N}_{\widehat{\mathbb{G}}})$, where $\mathscr{L}_{\widehat{\mathbb{G}}}$ must encode the torsion phenomena of $\widehat{\mathbb{G}}$. In this case, the induction-restriction approach is no longer valid since finite discrete quantum groups do not exhaust the torsion phenomena for $\widehat{\mathbb{G}}$. 

A candidate for $\mathscr{L}_{\widehat{\mathbb{G}}}$ had been apparent for specific examples \cite[Section 1]{MeyerNestTorsion} and \cite[Section 5]{VoigtBaumConnesAutomorphisms} (see also \cite[Section 4.1.2]{RubenThesis} for a description for general discrete quantum groups), but it was not at all clear if $(\mathscr{L}_{\widehat{\mathbb{G}}},\mathscr{L}_{\widehat{\mathbb{G}}}^{\vdash})$ formed a complementary pair. %See \cite[Section 5.4]{RubenThesis} for a general strategy
This was completely resolved by Arano and Skalski \cite{YukiBCTorsion}. Their key insight concerned a direct description of $\mathscr{N} = \mathscr{L}_{\widehat{\mathbb{G}}}^{\vdash}$ in terms of a \emph{double crossed product construction}. The results of \cite{YukiBCTorsion} hence allow to define a quantum assembly map for every discrete quantum group $\widehat{\mathbb{G}}$ (torsion-free or not). 

In Section \ref{sec.QuantumAssemblyMapProj}, we will revisit this result from a different perspective in the case when there is only \emph{projective torsion}, i.e.\ any finite dimensional C$^*$-algebra carrying an ergodic action of $\mathbb{G}$ is simple. More precisely, having defined a ``twisted'' descent map $F_\delta:=j_{\mathbb{G}, T}: \mathscr{K}\mathscr{K}^{\mathbb{G}}\longrightarrow \mathscr{K}\mathscr{K}$ for each projective torsion action $(T, \delta)$ and with $\tau_T:\mathscr{K}\mathscr{K}\longrightarrow\mathscr{K}\mathscr{K}^{\mathbb{G}}$ given by making the tensor product by $T$ on the right, we re-prove the adjointness between $j_{\mathbb{G}, T}$ and $\tau_T$, formulated as a twisted Green-Julg isomorphism, by explicit use of the specific structure of cocycle crossed products. Such results can be seen as a first step towards spectra computations in the quantum Kasparov category in the realm of tensor triangular geometry \cite{BalmerSpc,AmbrogioBCSpec}.

%. We show that it is adjoint to the functor 

%let $\widehat{\mathscr{L}}_{\widehat{\mathbb{G}}}$ be the localizing subcategory of $\mathscr{K}\mathscr{K}^{\mathbb{G}}$ generated by objects of the form $C\otimes T$, where $C\in Obj(\mathscr{K}\mathscr{K})$ is a C$^*$-algebra and $(T, \delta)\in Obj(\mathscr{K}\mathscr{K}^{\mathbb{G}})$ is a projective torsion action of $\mathbb{G}$. We show that the pair $(\widehat{\mathscr{L}}_{\widehat{\mathbb{G}}}, \widehat{\mathscr{L}}_{\widehat{\mathbb{G}}}^{\dashv})$ is complementary in $\mathscr{K}\mathscr{K}^{\mathbb{G}}$. To do so we generalize the \emph{Green-Julg isomorphism} in the following way. 

%We define As byproduct we obtain that the projective torsion actions of $\mathbb{G}$ are compact objects in $\mathscr{K}\mathscr{K}^{\mathbb{G}}$.  Note that, when constructing the associated quantum assembly map, our proof of the adjointness between $j_{\mathbb{G}, T}$ and $\tau_T$ differs from the one in \cite{YukiBCTorsion}, as we make direct use of the sp 
\bigskip

%\textsc{Acknowledgments}. The authors would like to thank the anonymous referee for their valuable comments.

\section{\textsc{Preliminaries}}
	\subsection{Conventions and notations}\label{sec.Conven}
	Let us fix the notations and the conventions that we use throughout the whole article.
	
	If $E$ is a $\mathbb{C}$-vector space and $\mathcal{S}$ is a subset of vectors of $E$, then we write $span\, \mathcal{S}$ for the corresponding $\mathbb{C}$-vector subspace generated by $\mathcal{S}$. If $(E,||\cdot||)$ is a normed $\mathbb{C}$-vector space and $F\subset E$ is a vector subspace, we write $[F]:=\overline{F}^{||\cdot||}$ for the $||\cdot||$-closure of $F$ in $E$. We then also write $\overline{span}\,\mathcal{S} = [span\, \mathcal{S}]$ for $\mathcal{S}\subset E$. If $V,W$ are subspaces inside an algebra $A$, we denote $VW = V\cdot W := \mathrm{span}\{vw \mid v\in V,w\in W\}$. 

Let $H$ be a Hilbert space. We denote by $\mathcal{B}(H)$ (resp.\ $\mathcal{K}(H)$) the space of all linear bounded (resp.\ compact) operators on $H$. If $\mathcal{S}$ is a subset of $\mathcal{B}(H)$, then we write $\overline{span}^{\sigma-weak}\mathcal{S}$ for the closure of the linear subspace generated by $\mathcal{S}$ with respect to the $\sigma$-weak topology. We denote by $\mathcal{B}(H)_*$ the space of normal functionals on $\mathcal{B}(H)$, and for $\xi, \eta\in H$ we denote by $\omega_{\xi, \eta}\in\mathcal{B}(H)_*$ the linear form defined by $\omega_{\xi, \eta}(T):=\langle \xi, T(\eta)\rangle$, for all $T\in \mathcal{B}(H)$. If $V\in \mathcal{B}(H\otimes H)$ is a unitary operator, we put $\mathscr{C}(V):=[\{(id\otimes \eta)(\Sigma V)\ |\ \eta\in\mathcal{B}(H)_*\}]$. Observe that $(id\otimes \eta)(\Sigma V)=((\eta\otimes id)\circ Ad_{\Sigma})(\Sigma V)=(\eta\otimes id)(V\Sigma)$, for all $\eta\in\mathcal{B}(H)_*$. Also, we clearly have $\mathscr{C}(V)  = \overline{span}\{(id\otimes \omega_{\xi, \eta})(\Sigma V)\ |\ \xi, \eta\in H\}$.

If $A$ is a C$^*$-algebra and $H$ a  Hilbert $A$-module, we denote by $\mathcal{L}_{A}(H)$ (resp.\ $\mathcal{K}_A(H)$) the space of all (resp.\ compact) adjointable operators on $H$. Hilbert $A$-modules are considered to be \emph{right $A$-modules}, so that the corresponding inner products are considered to be conjugate-linear on the left and linear on the right. Given a Hilbert $A$-module $H$ and $\xi, \eta\in H$ we denote by $\theta_{\xi, \eta}\in\mathcal{L}_A(H)$ the rank one operator defined by $\theta_{\xi, \eta}(\zeta):=\xi \langle \eta, \zeta\rangle$, for all $\zeta\in H$. Then $\mathcal{K}_A(H) = \overline{span}\{\theta_{\xi, \eta}\mid \xi, \eta\in H\}$. 

All our C$^*$-algebras (except for obvious exceptions such as multiplier C$^*$-algebras and von Neumann algebras) are supposed to be \emph{separable} and all our Hilbert modules are supposed to be \emph{countably generated}. If $A$ is a C$^*$-algebra and $\mathcal{S}$ is a subset of elements in $A$, we write $C^*\langle \mathcal{S}\rangle := C^*\langle \mathcal{S}\cup \mathcal{S}^*\rangle $ for the corresponding C$^*$-subalgebra of $A$ generated by $\mathcal{S}$, that is, the intersection of all C$^*$-subalgebras of $A$ containing $\mathcal{S}$. In this case, the elements of $\mathcal{S}$ are called \emph{generators} of $C^*\langle \mathcal{S}\rangle$. 

	The symbol $\otimes$ stands for the minimal tensor product of C$^*$-algebras and the exterior tensor product of Hilbert modules depending on the context. The symbol $\underset{\max}{\otimes}$ stands for the maximal tensor product of C$^*$-algebras. The symbol $\overline{\otimes}$ stands for the tensor product of von Neumann algebras. 	In any of the previous cases, the \emph{elementary tensors} in the corresponding tensor product are denoted simply by $\otimes$ and the context will distinguish the specific situation. If $H$ is a Hilbert $A$-module and $(K, \pi)$ is a Hilbert $(A, B)$-bimodule, the interior tensor product of $H$ and $K$ with respect to $\pi$ is denoted by $H\underset{\pi}{\otimes} K$ or $H\underset{A}{\otimes} K$. If $A$ and $B$ are two C$^*$-algebras, $\Sigma:A\otimes B\longrightarrow B\otimes A$ denotes the flip map. The symbol $\Sigma$ is used as well for the suspension functor in the framework of triangulated categories. The context will distinguish the specific situation. We use systematically the leg numbering, so if $H$ is a Hilbert space then $X_{12} = X\otimes 1 \in \mathcal{B}(H^{\otimes 3})$ for $X\in \mathcal{B}(H^{\otimes 2})$, etc.

	If $S, A$ are C$^*$-algebras, we denote by $M(A)=\mathcal{L}_A(A)$ the multiplier algebra of $A$ and we put $\widetilde{M}(A\otimes S):=\{x\in M(A\otimes S)\ |\ x(id_A\otimes S)\subset A\otimes S\mbox{ and } (id_A\otimes S)x\subset A\otimes S\}$, which contains $A\otimes M(S)$. If $H$ is a Hilbert $A$-module, we put $M(H):=\mathcal{L}_A(A, H)$, which contains canonically $H\cong \mathcal{K}_A(A, H)$. We put $\widetilde{M}(H\otimes S):=\{X\in M(H\otimes S)\ |\ X(id_A\otimes S)\subset H\otimes S\mbox{ and } (id_H\otimes S)X\subset H\otimes S\}$, which contains $H\otimes M(S)$.

If $T:= \mathcal{B}(H)$ is a type $I$-factor, we denote by $\text{Tr}$ the usual trace on $T$. If $\varphi$ is any state on $T$, we denote by $\varrho\in T$ the density matrix (i.e.\ the positive matrix with trace $1$) such that $\varphi=\text{Tr}(\varrho\ \cdot)$.

Given a state $\varphi$ on $T$, we denote by $(L^2(T), \lambda_T, \Lambda_T, \xi_T)$ the corresponding GNS construction, but we drop the notation $\lambda_T$ when it is clear from the context. If $T^{op}$ denotes the opposite von Neumann algebra of $T$, then the modular properties for $\varphi$ yield a $*$-representation $\rho_T$ of $T^{op}$ on $L^2(T)$ determined by the formula $\rho_T(s^{op})(t\xi_T):=t \varrho^{1/2}s\varrho^{-1/2}\xi_T$ for all $s, t\in T$ of finite rank with respect to an eigenbasis of $\rho$. We consider the anti $*$-homomorphism $(\cdot)^{\circ}: T\longrightarrow \mathcal{B}(L^2(T))$ defined by $s^{\circ}:=\rho_T (s^{op})=J_Ts^*J_T$ for all $s\in T$, where $J_Tt\xi_T := \rho^{1/2}t^*\rho^{-1/2}\xi_T$ is the modular conjugation on $L^2(T)$. We then have $T^{\circ}=T'$. In the following, we will also identify $j: \mathcal{B}(H)^{op} \cong \mathcal{B}(\overline{H})$ through the $*$-isomorphism $T \mapsto \overline{T^*}$.

	Whenever $\mathscr{C}$ denotes a category, we shall assume that $\mathscr{C}$ is essentially small, so morphisms $Hom_{\mathscr{C}}(\cdot, \cdot)$ form sets. Given a category $\mathscr{C}$, we denote by $\mathscr{C}^{op}$ its opposite category. 	We say that $\mathscr{C}$ is \emph{countable additive} if it is additive and it admits \emph{countable direct sums}. If $F$ is an \emph{additive} functor on an additive category, it is, by definition, compatible with \emph{finite} direct sums. The categories considered in the present paper are assumed to be countable additive. Whenever we require an additive functor to be compatible with \emph{infinite (countable)} direct sums, it will be explicitly indicated.	

We denote by $\mathscr{A}b$ the abelian category of abelian groups and by $\mathscr{A}b^{\mathbb{Z}/2}$ the abelian category of $\mathbb{Z}/2$-graded abelian groups.

	\subsection{Compact/Discrete Quantum Groups}\label{sec.CQG}

	In this section we recall elementary and fundamental facts concerning compact quantum groups and their corresponding duality theory. We refer to the books \cite{Sergey,Timmermann} or to the original papers \cite{Woronowicz,SkandalisUnitaries} for more details.

	\begin{defi}\label{defi.CQG}
			A compact quantum group $\mathbb{G}$ is the data $(C(\mathbb{G}),\Delta)$ where $C(\mathbb{G})$ is a unital C$^*$-algebra and $\Delta: C(\mathbb{G})\longrightarrow C(\mathbb{G})\otimes C(\mathbb{G})$ is a unital $*$-homomorphism such that: 
\begin{enumerate}[i)]
\item $\Delta$ is co-associative meaning that $(id\otimes \Delta)\Delta=(\Delta\otimes id)\Delta$ and 
\item  $\Delta$ satisfies the cancellation property meaning that $[\Delta(C(\mathbb{G}))(C(\mathbb{G})\otimes 1)]=C(\mathbb{G})\otimes C(\mathbb{G})= [\Delta(C(\mathbb{G}))(1\otimes C(\mathbb{G}))]$.
\end{enumerate}
	\end{defi}

A compact quantum group has a unique \emph{Haar state} $h_{\mathbb{G}}$ such that $(h_{\mathbb{G}}\otimes id)\Delta(x) = h_{\mathbb{G}}(x)1_{C(\mathbb{G})} = (id\otimes h_{\mathbb{G}})\Delta(x)$ for all $x\in C(\mathbb{G})$. \emph{We will make the standing assumption that $h_{\mathbb{G}}$ is faithful, so we only work with the reduced form $C(\mathbb{G})$ of a compact quantum group.} 

The GNS construction corresponding to $h_{\mathbb{G}}$ is denoted by $(L^2(\mathbb{G}), \lambda, \xi_{\mathbb{G}})$. We also write $\Lambda(x) = \lambda(x)\xi_{\mathbb{G}}$ for $x\in C(\mathbb{G})$. We adopt the standard convention for the inner product on $L^2(\mathbb{G})$, which means that $\langle \Lambda(x), \Lambda(y)\rangle:=h_{\mathbb{G}}(x^*y)$ for all $x,y\in C(\mathbb{G})$. We suppress the notation $\lambda$ in computations so that we simply write $x\Lambda(y)=\Lambda(xy)$ for all $x,y\in C(\mathbb{G})$. 

	\begin{theodefi}[Regular representation]\label{theo.RRegularRepresentation}
		Let $\mathbb{G}=(C(\mathbb{G}),\Delta)$ be a compact quantum group.
		\begin{enumerate}[i)]
			\item There exists a unique unitary operator $V_{\mathbb{G}}\in M(\mathcal{K}(L^2(\mathbb{G}))\otimes C(\mathbb{G}))$ such that $V_{\mathbb{G}}(\Lambda(x)\otimes \xi)=\Delta(x)(\xi_{\mathbb{G}}\otimes \xi)\mbox{,}$
				for all $x\in C(\mathbb{G})$ and $\xi\in L^2(\mathbb{G})$.
			\item\label{ItCopr} For all $x\in C(\mathbb{G})$ we have $\Delta = \Delta_{V_{\mathbb{G}}}$ where $\Delta_{V_{\mathbb{G}}}(x)=V_{\mathbb{G}}(x\otimes 1)V_{\mathbb{G}}^*$.
			\item The following identity holds: $(id\otimes \Delta)(V_{\mathbb{G}})=(V_{\mathbb{G}})_{12}(V_{\mathbb{G}})_{13}$.
			\item The following pentagonal equation holds: $(V_{\mathbb{G}})_{12}(V_{\mathbb{G}})_{13}(V_{\mathbb{G}})_{23}=(V_{\mathbb{G}})_{23}(V_{\mathbb{G}})_{12}$. So $V_{\mathbb{G}}$ is a multiplicative unitary on $L^2(\mathbb{G})$ in the sense of Baaj-Skandalis \cite{SkandalisUnitaries}.
			\item We have that $C(\mathbb{G}) = S_{V_{\mathbb{G}}} := \overline{span}\{(\eta\otimes id)(V_{\mathbb{G}})\ |\ \eta\in \mathcal{B}(L^2(\mathbb{G}))_*\}$.
		\end{enumerate}
		
		The unitary $V_{\mathbb{G}}$ is called \emph{right regular representation of $\mathbb{G}$ on $L^2(\mathbb{G})$} or \emph{fundamental unitary of $\mathbb{G}$}. 
	\end{theodefi}
	\begin{rem}
		In a similar way, we can define the \emph{left regular representation of $\mathbb{G}$}: there exists a unique multiplicative unitary $W_{\mathbb{G}}\in M(C(\mathbb{G})\otimes \mathcal{K}(L^2(\mathbb{G})))$ such that $(W_{\mathbb{G}})^*(\xi\otimes \Lambda(x))=\Delta(x)(\xi\otimes \xi_{\mathbb{G}})\mbox{,}$
		for all $x\in C(\mathbb{G})$ and $\xi\in L^2(\mathbb{G})$. For all $x\in C(\mathbb{G})$ we have $\Delta(x)=W_{\mathbb{G}}^*(1\otimes x)W_{\mathbb{G}}$ and the following identity holds: $(\Delta\otimes id)(W_{\mathbb{G}})=(W_{\mathbb{G}})_{13}(W_{\mathbb{G}})_{23}$. 
	\end{rem}

The coproduct on $C(\mathbb{G})$ can be extended to $L^{\infty}(\mathbb{G}) = C(\mathbb{G})''$ using Theorem-Definition \ref{theo.RRegularRepresentation}.\eqref{ItCopr}, obtaining the normal map $\Delta: L^{\infty}(\mathbb{G}) \rightarrow L^{\infty}(\mathbb{G}) \overline{\otimes} L^{\infty}(\mathbb{G})$. The Haar state extends uniquely to a normal faithful state on $L^{\infty}(\mathbb{G})$, and we denote by $J_{\mathbb{G}}$ the associated modular conjugation on $L^2(\mathbb{G})$.

Conversely, if $L^{\infty}(\mathbb{G})$ is a von Neumann algebra with a coassociative normal $*$-homomorphism $\Delta: L^{\infty}(\mathbb{G}) \rightarrow L^{\infty}(\mathbb{G}) \overline{\otimes} L^{\infty}(\mathbb{G})$ and admitting an invariant normal faithful state $h_{\mathbb{G}}$, then $(L^{\infty}(\mathbb{G}),\Delta)$ arises from a (reduced) compact quantum group $\mathbb{G}$ in a unique way. 

\begin{defi}
A unitary representation of $\mathbb{G}$ on a Hilbert space $H =H_u$ is a unitary element $u$ in $M(\mathcal{K}(H)\otimes C(\mathbb{G}))$ with $(id\otimes \Delta)(u) = u_{12}u_{13}$. For $u,v$ unitary representations, we denote $Hom_{\mathbb{G}}(u,v) = \{T: H_u \rightarrow H_v\mid T \textrm{ bounded and }(T\otimes 1)u =v(T\otimes 1)\}$. One calls $u$ \emph{irreducible} if $End_{\mathbb{G}}(u) = Hom_{\mathbb{G}}(u,u) = \mathbb{C}id_{H_u}$.   
\end{defi}
In the following, all representations will be assumed unitary.

Any irreducible representation $u$ has finite dimensional  $H$, so then $u \in \mathcal{B}(H)\otimes C(\mathbb{G})$. The set of all equivalence classes of irreducible representations of $\mathbb{G}$ is denoted by $\text{Irr}(\mathbb{G})$. If $x\in \text{Irr}(\mathbb{G})$ is such a class, we write $u^x\in\mathcal{B}(H_x)\otimes C(\mathbb{G})$ for a representative of $x$ and $H_x$ for the finite dimensional Hilbert space on which $u^x$ acts. We write $dim(x):= n_x$ for the dimension of $H_x$. The trivial representation of $\mathbb{G}$ is denoted by $\epsilon$, and we put $u^{\epsilon} = 1_{C(\mathbb{G})}$. Given $x,y\in \text{Irr}(\mathbb{G})$, the tensor product of $x$ and $y$ is denoted by $x\otimes y$. Given $x\in \text{Irr}(\mathbb{G})$, there exists a unique class $\overline{x}$ of irreducible representations of $\mathbb{G}$  such that $Hom_{\mathbb{G}}(\epsilon, u^x\otimes u^{\overline{x}})\neq 0\neq Hom_{\mathbb{G}}(\epsilon, u^{\overline{x}}\otimes u^x)$. It is called the \emph{contragredient} or \emph{conjugate representation} of $x$. 

The linear span of matrix coefficients of all finite dimensional representations of $\mathbb{G}$ is denoted by $\text{Pol}(\mathbb{G})$. It is a $*$-Hopf algebra by restriction of the co-multiplication $\Delta$, and we denote its co-unit by $\varepsilon$ and its antipode by $S$. Let $I_0$ be the anti-linear involutive map $\Lambda(\text{Pol}(\mathbb{G})) \rightarrow L^2(\mathbb{G})$ defined by $\Lambda(x) \mapsto \Lambda(S(x)^*)$ for $x\in \text{Pol}(\mathbb{G})$. Then $I_0$ is closeable, and we denote $I = \widehat{J}_{\mathbb{G}}|I|$ for the polar decomposition of its closure. The map $R(x)=\widehat{J}_{\mathbb{G}}x^* \widehat{J}_{\mathbb{G}}$, for all $x\in C(\mathbb{G})$, is a well-defined anti-multiplicative and anti-co-multiplicative map on $C(\mathbb{G})$ preserving $\text{Pol}(\mathbb{G})$, called \emph{unitary antipode}.
	
	\begin{theodefi}[Discrete quantum group]
		Let $\mathbb{G}=(C(\mathbb{G}), \Delta)$ be a compact quantum group. We switch between the following notations for the same space
		$c_0(\widehat{\mathbb{G}})= C^*_r(\mathbb{G})= \widehat{S}_{V_{\mathbb{G}}} := [\{(id\otimes \eta)(V_{\mathbb{G}})\ |\ \eta\in \mathcal{B}(L^2(\mathbb{G}))_*\}]\subset \mathcal{B}(L^2(\mathbb{G}))$. Then $c_0(\widehat{\mathbb{G}})$ is a C$^*$-algebra, and we denote also the identity map by:
\[
\widehat{\lambda}: c_0(\widehat{\mathbb{G}}) \rightarrow \mathcal{B}(L^2(\mathbb{G})).
\]
Furthermore, we have the following:
		\begin{enumerate}[i)]
			\item\label{EqCoprDual} The formula $\widehat{\Delta}(x) = \widehat{\Delta}_{V_{\mathbb{G}}}^{cop}(x):=\Sigma V_{\mathbb{G}}^*(1\otimes x)V_{\mathbb{G}}\Sigma $ defines a non-degenerate $*$-homomorphism $c_0(\widehat{\mathbb{G}})\longrightarrow \widetilde{M}(c_0(\widehat{\mathbb{G}})\otimes c_0(\widehat{\mathbb{G}}))$ such that the pair $\widehat{\mathbb{G}} = (c_0(\widehat{\mathbb{G}}), \widehat{\Delta})$ is a locally compact quantum group. One calls $\widehat{\mathbb{G}}$ the \emph{(Pontryagin) dual discrete quantum group} of $\mathbb{G}$.
			\item There exists a natural isomorphism $c_0(\widehat{\mathbb{G}})\cong \underset{x\in \text{Irr}(\mathbb{G})}{\bigoplus^{c_0}} \mathcal{B}(H_x)$.
\item We have $V_{\mathbb{G}} \in M(c_0(\widehat{\mathbb{G}})\otimes C(\mathbb{G}))$.
\item We have $[ (\eta\otimes id)(W^*_{\mathbb{G}})\ |\ \eta\in \mathcal{B}(L^2(\mathbb{G}))_*]=\widehat{J}_{\mathbb{G}}c_0(\widehat{\mathbb{G}})\widehat{J}_{\mathbb{G}}\subset \mathcal{B}(L^2(\mathbb{G}))$.
		\end{enumerate}
	\end{theodefi}

We denote $l^{\infty}(\widehat{\mathbb{G}})$ for the $\sigma$-weak closure of $c_0(\widehat{\mathbb{G}})$. It is a von Neumann algebra with coproduct $\widehat{\Delta}$ given by extending the formula in item \ref{EqCoprDual}) above. It has a left, resp.\ right invariant normal, semifinite faithful weight $\widehat{h}_L$, resp.\ $\widehat{h}_R$. We can identify $L^2(\mathbb{G})$ with the standard space of $l^{\infty}(\widehat{\mathbb{G}})$ in such a way that $\widehat{J}_{\mathbb{G}}$ becomes the associated modular conjugation. We further have inside $c_0(\widehat{\mathbb{G}})$ the dense $2$-sided ideal: 
\[
c_{00}(\widehat{\mathbb{G}}) \cong  \underset{x\in \text{Irr}(\mathbb{G})}{\bigoplus^{\mathrm{alg}}} \mathcal{B}(H_x),
\]
contained in the set of integrable elements for $\widehat{h}_L$ and $\widehat{h}_R$.

	\begin{theodefi}[Kac system associated to $\mathbb{G}$]\label{theo.KacSystemG}
		Let $\mathbb{G}=(C(\mathbb{G}), \Delta)$ be a compact quantum group. Then $U_{\mathbb{G}} = J_{\mathbb{G}}\widehat{J}_{\mathbb{G}} = \widehat{J}_{\mathbb{G}}J_{\mathbb{G}}\in \mathcal{B}(L^2(\mathbb{G}))$ is a symmetry, and we call the pair $(V_{\mathbb{G}}, U_{\mathbb{G}})$ the \emph{standard Kac system associated to $\mathbb{G}$}. We then denote 
\[
\rho(a) = U_{\mathbb{G}}\lambda(a)U_{\mathbb{G}},\qquad \widehat{\rho}(x) = U_{\mathbb{G}}\widehat{\lambda}(x)U_{\mathbb{G}},\qquad  a \in C(\mathbb{G}),x\in c_0(\widehat{\mathbb{G}}). 
\]

Moreover, we have $W_{\mathbb{G}} = \check{V}_{\mathbb{G}}$ where  
		$$\check{V}_{\mathbb{G}}:=\Sigma(U_{\mathbb{G}}\otimes 1)V_{\mathbb{G}}(U_{\mathbb{G}}\otimes 1)\Sigma\equiv (U_{\mathbb{G}})_{2}(V_{\mathbb{G}})_{21}(U_{\mathbb{G}})_{2} \in M(C(\mathbb{G}) \otimes \widehat{\rho}(c_0(\mathbb{G}))),$$
and $V_{\mathbb{G}}$, $\check{V}_{\mathbb{G}}$ together with
		$$\widetilde{V}_{\mathbb{G}}:=\Sigma(1\otimes U_{\mathbb{G}})V_{\mathbb{G}}(1\otimes U_{\mathbb{G}})\Sigma\equiv (U_{\mathbb{G}})_{1}(V_{\mathbb{G}})_{21}(U_{\mathbb{G}})_{1} \in M(\rho(C(\mathbb{G}))\otimes c_0(\mathbb{G})),$$
$$\widetilde{\widetilde{V}}_{\mathbb{G}}=(U_{\mathbb{G}}\otimes U_{\mathbb{G}})V_{\mathbb{G}}(U_{\mathbb{G}}\otimes U_{\mathbb{G}}) \in M(\rho(C(\mathbb{G}))\otimes \widehat{\rho}(c_0(\mathbb{G})))$$
		\begin{enumerate}[i)]
			\item are multiplicative on $L^2(\mathbb{G})$ in the sense of Baaj-Skandalis,
			\item are regular, meaning that $\mathcal{K}(L^2(\mathbb{G})) = \mathscr{C}(V_{\mathbb{G}}) = \mathscr{C}(\widetilde{V}_{\mathbb{G}})=\mathscr{C}(\check{V}_{\mathbb{G}}) = \mathscr{C}(\widetilde{\widetilde{V}}_{\mathbb{G}})$, 
			\item satisfy the following identity in $V$: $(\Sigma(1\otimes U_{\mathbb{G}})V)^3=id$.
		\end{enumerate}
		
%KDC: Changed \overline{span} into [-] in viii).
		Moreover, the following properties hold:
		\begin{enumerate}[i)]
			\item $(V_{\mathbb{G}})_{13}(V_{\mathbb{G}})_{23}(\widetilde{V}_{\mathbb{G}})_{12}=(\widetilde{V}_{\mathbb{G}})_{12}(V_{\mathbb{G}})_{13}$ and $(\check{V}_{\mathbb{G}})_{23}(V_{\mathbb{G}})_{12}(V_{\mathbb{G}})_{13}=(V_{\mathbb{G}})_{13}(\check{V}_{\mathbb{G}})_{23}$.
			\item $(c_0(\widehat{\mathbb{G}}), \widehat{\Delta}^{cop})= (S_{\widetilde{V}_{\mathbb{G}}}, \Delta_{\widetilde{V}_{\mathbb{G}}})$; in particular, $\widehat{\Delta}^{cop}(x)=\widetilde{V}_{\mathbb{G}}(x\otimes 1)\widetilde{V}_{\mathbb{G}}^*$, for all $x\in c_0(\widehat{\mathbb{G}})$.
			\item $(C(\mathbb{G}), \Delta)= (\widehat{S}_{\check{V}_{\mathbb{G}}}, \widehat{\Delta}_{\check{V}_{\mathbb{G}}})$; in particular, $\Delta(a)=\check{V}_{\mathbb{G}}^*(1\otimes a)\check{V}_{\mathbb{G}}$, for all $a\in C(\mathbb{G})$.
			\item $(V_{\mathbb{G}})_{12}(\widetilde{V}_{\mathbb{G}})_{23}=(\widetilde{V}_{\mathbb{G}})_{23}(V_{\mathbb{G}})_{12}$; in particular, $\widetilde{V}_{\mathbb{G}}(a\otimes 1)\widetilde{V}_{\mathbb{G}}^*=a\otimes 1$, for all $a\in C(\mathbb{G})$.
			\item $(V_{\mathbb{G}})_{23}(\check{V}_{\mathbb{G}})_{12}=(\check{V}_{\mathbb{G}})_{12}(V_{\mathbb{G}})_{23}$; in particular, $\check{V}_{\mathbb{G}}(1\otimes x)\check{V}^*_{\mathbb{G}}=1\otimes x$, for all $x\in c_0(\widehat{\mathbb{G}})$.
\item $(\widetilde{V}_{\mathbb{G}})_{12}(\widetilde{\widetilde{V}}_{\mathbb{G}})_{23}=(\widetilde{\widetilde{V}}_{\mathbb{G}})_{23}(\widetilde{V}_{\mathbb{G}})_{12}$; in particular, $V_{\mathbb{G}}(x\otimes 1)V_{\mathbb{G}}^*=x\otimes 1$, for all $x\in U_{\mathbb{G}}c_0(\widehat{\mathbb{G}})U_{\mathbb{G}}$.
			\item $(U_{\mathbb{G}}C(\mathbb{G})U_{\mathbb{G}}), \Delta_{U_{\mathbb{G}}})=(\widehat{S}_{\widetilde{V}_{\mathbb{G}}}, \widehat{\Delta}_{\widetilde{V}_{\mathbb{G}}})$,  where $\Delta_{U_{\mathbb{G}}}(U_{\mathbb{G}}aU_{\mathbb{G}}):=Ad_{U_{\mathbb{G}}\otimes U_{\mathbb{G}}}(\Delta(a))$, for all $a\in C(\mathbb{G})$; in particular, $V_{\mathbb{G}}(1\otimes a)V^*_{\mathbb{G}}=1\otimes a$, for all $a\in U_{\mathbb{G}}C(\mathbb{G})U_{\mathbb{G}}$. 			
			\item $[C(\mathbb{G})\cdot c_0(\widehat{\mathbb{G}})] =\mathcal{K}(L^2(\mathbb{G}))$.
		\end{enumerate}
	\end{theodefi}

We refer to \cite{Timmermann} or \cite{SkandalisUnitaries} for more 
details about these computations.

\subsection{Actions of compact and discrete quantum groups}

In this section we recall elementary notions and results concerning actions of quantum groups. 

%KDC: I changed the terminology after all, I agree it makes more sense.
		\begin{defi}\label{defi.CompactQuantumAction}
			Let $\mathbb{G}=(C(\mathbb{G}),\Delta)$ be a compact quantum group and $A$ a C$^*$-algebra. A \emph{left (continuous) action} of $\mathbb{G}$ on $A$ (or a \emph{right co-action} of $C(\mathbb{G})$ on $A$) is a non-degenerate $*$-homomorphism $\delta: A\longrightarrow A\otimes C(\mathbb{G})$ such that: 
\begin{enumerate}[i)]
\item $\delta$ intertwines the co-multiplication, meaning that $(\delta\otimes id_{C(\mathbb{G})})\circ\delta = (id_A\otimes \Delta)\circ\delta$ and 
\item $\delta$ satisfies the density condition $[\delta(A)(1\otimes C(\mathbb{G}))] =A\otimes C(\mathbb{G})$.
\end{enumerate}
 We write $\mathbb{G}\overset{\delta}{\curvearrowright} A$. We say that $(A,\delta)$ is a left $\mathbb{G}$-C$^*$-algebra if moreover $\delta$ is injective.

If $M$ is a von Neumann algebra, then a left (measurable) action of $\mathbb{G}$ on $M$ is a normal unital $*$-homomorphism $\delta: M\longrightarrow M\overline{\otimes} L^{\infty}(\mathbb{G})$ intertwining the co-multiplication (the density condition being superfluous in this case).
\end{defi}

\begin{ex}
The co-multiplication of any compact quantum group $\mathbb{G}$ defines an action of $\mathbb{G}$ on its defining C$^*$-algebra. This action is called the \emph{regular action of $\mathbb{G}$}.
\end{ex}

		Similarly, we can define a \emph{right} action of $\mathbb{G}$ on $A$ (or a \emph{left} co-action of $C(\mathbb{G})$ on $A$) as a non-degenerate $*$-homomorphism $\delta: A\longrightarrow C(\mathbb{G})\otimes A$ satisfying the analogous properties of the preceding definition. In the present article, an action of a compact quantum group $\mathbb{G}$ is supposed to be a \emph{left} one unless the contrary is explicitly indicated. Hence, we refer to such actions simply as \emph{actions of $\mathbb{G}$}.	Observe however that if $(A, \delta)$ is a \emph{left} $\mathbb{G}$-C$^*$-algebra, then $(A^{op}, \overline{\delta})$ is a \emph{right} $\mathbb{G}$-C$^*$-algebra where $A^{op}$ denotes the opposite C$^*$-algebra of $A$ and 
$\overline{\delta}: A^{op}\longrightarrow C(\mathbb{G})\otimes A^{op}$ is defined by:
\begin{equation}\label{eq.oppcoact}
\overline{\delta}:=(R\otimes id)\circ\Sigma\circ\delta,
\end{equation}
 where $R$ denotes the unitary antipode of $\mathbb{G}$.
 
\begin{rem}
Our conventions for the left/right terminology are motivated by the following: when $\mathbb{G} = G$ for an ordinary compact group $G$, we obtain an honest left action of $G$ on $A$ from a right coaction $\delta: A \rightarrow A\otimes C(G)$ by evaluating the second leg in $g\in G$: 
\[
G\times A \rightarrow A,\qquad (g,a)\mapsto \delta_g(a) :=  (\mathrm{id}\otimes \mathrm{ev}_g)\delta(a). 
\]  
\end{rem} 

\begin{defi}\label{defi.condExp}
 If $(A, \delta)$ is a $\mathbb{G}$-C$^*$-algebra, we denote $A^\delta = \{a\in A\mid \delta(a) = a\otimes 1\}$. We call $(A,\delta)$ \emph{ergodic} if $A$ is unital and $A^{\delta}  = \mathbb{C}1_A$. 

In general, we denote  $E_\delta: A\longrightarrow A^\delta$ for the $\delta$-invariant conditional expectation given by $E_{\delta}(a)=(id_A\otimes h_{\mathbb{G}})\delta(a)$, for all $a\in A$.
\end{defi} 

%KDC: Added that approximate unit should be bounded, and removed the index \N as we want nets in general. Also changed a type v_i -> u_i.
\begin{rem}\label{rem.nondegact} 
Recall that we assume $\mathbb{G}$ to be a reduced compact quantum group, so $E_{\delta}$ is automatically faithful. Moreover, if $(u_i)_{i}$ is a (bounded) approximate unit for $A$, then by non-degeneracy of $\delta$ we have that $\big(\delta(u_i)\big)_{i}$ is an approximate unit for $A\otimes C(\mathbb{G})$. Thanks to the continuity of $id\otimes h_{\mathbb{G}}$, the operators $v_i:=(id\otimes h_{\mathbb{G}})\big(\delta(u_i)\big)$ form an approximate unit for $A$ inside $A^{\delta}$, and we have $[AA^{\delta}] = [AA^{\delta}] = A$. In particular, if $A$ acts non-degenerately on a Hilbert space $H$, then also $A^{\delta}$ acts nondegenerately on $H$. 
\end{rem} 

Given a $\mathbb{G}$-C$^*$-algebra $A$, we can equip $A$ with the pre-Hilbert $A^\delta$-module structure given by $\langle a, b\rangle_{E_\delta}:=E_\delta(a^*b)$, for all $a,b\in A$. We denote by $L^2(A, E_\delta)$ the completion of $A$ with respect to the inner product $\langle\cdot,\cdot\rangle_{E_\delta}$. When $\delta$ is ergodic, we have $E_\delta(a)=\varphi_{\delta}(a)1_A$ for $\varphi_{\delta}$ a (unique) $\delta$-invariant state on $A$. We then write $L^2(A)=L^2(A, \varphi)$ for the Hilbert space completion of $A$ with respect to the inner product $\langle a, b\rangle_{\varphi}:=\varphi(a^*b)$, for all $a,b\in A$.
	
	The notion of action of $\mathbb{G}$ can be defined also for Hilbert modules.
	\begin{defi}
		Let $\mathbb{G}=(C(\mathbb{G}),\Delta)$ be a compact quantum group and $(A, \delta)$ a $\mathbb{G}$-C$^*$-algebra. Let $E$ be a Hilbert $A$-module. A left action of $\mathbb{G}$ on $E$ (or a right co-action of $C(\mathbb{G})$ on $E$) is a linear map $\delta_E: E\longrightarrow E\otimes C(\mathbb{G})$ such that: 
\begin{enumerate}[i)]
\item $\delta_E(\xi\cdot a)=\delta_E(\xi)\delta(a)$ for all $\xi\in E$, $a\in A$; 
\item $\delta(\langle \xi, \eta\rangle)=\langle \delta_E(\xi), \delta_E(\eta)\rangle$ for all $\xi,\eta\in E$;
\item $\delta_E$ intertwines the co-multiplication meaning that $(\delta_E\otimes id_{C(\mathbb{G})})\circ\delta_E = (id_A\otimes \Delta)\circ\delta_E$;
\item The density conditions $[\delta_E(E)(1\otimes C(\mathbb{G}))]=[(1\otimes C(\mathbb{G}))\delta_E(E)] = E\otimes C(\mathbb{G})$ are satisfied.
\end{enumerate}
We write $ \mathbb{G}\overset{\delta_E}{\curvearrowright}E$. We say that $(E,\delta_E)$ is a left $\mathbb{G}$-equivariant Hilbert $A$-module if moreover $\delta_E$ is injective.
	\end{defi}
% I wrote previously that the density condition was automatically satisfied, but that must have been maybe due to faulty memory, I can't recall why this would be automatic...

		If $(E, \delta_E)$ is a $\mathbb{G}$-equivariant Hilbert $A$-module as above, then $\mathcal{K}_A(E)$ is a $\mathbb{G}$-C$^*$-algebra with action $\delta_{\mathcal{K}_A(E)}$ defined by $\delta_{\mathcal{K}_A(E)}(\theta_{\xi, \eta})=\delta_E(\xi)\delta_E(\eta)^*\in \mathcal{K}_A(E)\otimes C(\mathbb{G})$, for all $\xi, \eta\in E$ where $\theta_{\xi, \eta}$ denotes the corresponding rank one operator in $E$. By abuse of notation, we still denote by $\delta_{\mathcal{K}_A(E)}$ the extension of this homomorphism to $\mathcal{L}_A(E) = M(\mathcal{K}_A(E)) \rightarrow M(\mathcal{K}_{A}(E) \otimes C(\mathbb{G}))$. The latter is however not in general an action of $\mathbb{G}$ on $\mathcal{L}_A(E)$. 

Recall further that giving an action $\delta_E$ is equivalent to giving a unitary operator $V_E\in\mathcal{L}_{A\otimes C(\mathbb{G})}\big(E\underset{\delta}{\otimes}(A\otimes C(\mathbb{G})), E\otimes C(\mathbb{G})\big)$ such that $\delta_E(\xi)=V_E\circ T_\xi$ for all $\xi\in E$ where $T_\xi\in\mathcal{L}_{A\otimes C(\mathbb{G})}(A\otimes C(\mathbb{G}), E\underset{\delta}{\otimes}(A\otimes C(\mathbb{G})))$ is such that $T_\xi(x)=\xi\underset{\delta}{\otimes} x$, for all $x\in A\otimes C(\mathbb{G})$. One calls $V_E$ the \emph{admissible operator for $(E, \delta_E)$}. Moreover, we have $\delta_{\mathcal{K}_A(E)}=Ad_{V_E}$. We refer to \cite{BaajSkandalisQuantumKK} for more details.
	
	Next, we recall the following useful result (recall the notations from Definition \ref{defi.condExp}).

	\begin{pro}\label{theo.UnitaryImplementation}
		Let $\mathbb{G}$ be a compact quantum group. Let $(A, \delta)$ be a unital $\mathbb{G}$-C$^*$-algebra. If $\delta$ is ergodic, then there exists a representation $V_A\in M(\mathcal{K}(L^2(A))\otimes C(\mathbb{G}))$ of $\mathbb{G}$ such that $\delta(a)=V_A(a\otimes 1)V^*_A$, for all $a\in A$. 
	\end{pro}
	\begin{proof}
		Consider the map $A\otimes C(\mathbb{G})\overset{V_A}{\longrightarrow} A\otimes C(\mathbb{G})$ such that $a\otimes x\mapsto \delta(a)(1_A\otimes x)$. By $\delta$-invariance of $\varphi_{\delta}$, this map is isometric with respect to the natural pre-Hilbert $C(\mathbb{G})$-module structure on $A\otimes C(\mathbb{G})$. Moreover, since $\delta$ is an action of $\mathbb{G}$ on $A$, we know that $[\delta(A)(1\otimes C(\mathbb{G}))]=A\otimes C(\mathbb{G})$, that is, $V_A$ has dense range. Accordingly, $V_A$ extends to a unitary operator in $M(\mathcal{K}(L^2(A))\otimes C(\mathbb{G}))$, which we still denote by $V_A$. 

The relation $\delta(a)V_A=V_A(a\otimes 1)$, for all $a\in A$ is obvious. The coaction property for $\delta$ straightforwardly leads to $(V_A)_{12}(V_A)_{13}(V_{\mathbb{G}})_{23}= (V_{\mathbb{G}})_{23}(V_A)_{12}$, so $(id\otimes \Delta)V_A=(V_A)_{12}(V_A)_{13}$ and $V_A$ is a representation of $\mathbb{G}$ on $L^2(A)$ (see \cite{Boca} for more details).
	\end{proof}
	\begin{rem}\label{rem.ActExtvN}
		A similar result can be obtained when $\delta$ is not ergodic by considering instead the Hilbert $A^\delta$-module $L^2(A, E_\delta)$. One also has a corresponding result in the von Neumann algebraic setting: if $M \rightarrow M\overline{\otimes} L^{\infty}(\mathbb{G})$ on a von Neumann algebra (say with separable predual), we can find a $\mathbb{G}$-invariant state on $M$ leading to a unitary $V_M: L^2(M) \otimes L^2(\mathbb{G}) \rightarrow L^2(M)\otimes L^2(\mathbb{G})$ as above. This map is independent of the chosen state \cite{VaesUnitImpl}.
\end{rem}

We also recall the notion of action for discrete quantum groups.

	\begin{defi}\label{defi.QuantumAction}
		Let $\mathbb{G}$ be a compact quantum group and $A$ a C$^*$-algebra. A \emph{left action} of $\widehat{\mathbb{G}}$ on $A$ (or a right co-action of $c_0(\widehat{\mathbb{G}})$ on $A$) is a non-degenerate $*$-homomorphism $\delta: A\longrightarrow \widetilde{M}(A\otimes c_0(\widehat{\mathbb{G}}))$ such that:
\begin{enumerate}[i)]
\item $\delta$ intertwines the co-multiplication meaning that $(\delta\otimes id)\delta=(id\otimes \widehat{\Delta})\delta$ and
\item $\delta$ satisfies the cancellation property meaning that $[\delta(A)(1\otimes c_0(\widehat{\mathbb{G}}))]=A\otimes c_0(\widehat{\mathbb{G}})$.
\end{enumerate} 
We say that $(A,\delta)$ is a left $\widehat{\mathbb{G}}$-C$^*$-algebra if moreover $\delta$ is injective.
		\end{defi}
Again, one has the analogous notion of a right action of $\widehat{\mathbb{G}}$.  In the following, an action of a discrete quantum group $\widehat{\mathbb{G}}$ is supposed to be a \emph{left} one unless the contrary is explicitly indicated.

	\subsection{Torsion phenomena for discrete quantum groups}
	
	In this section we recall elementary notions and results concerning the notion of torsion-freeness for discrete quantum groups. It was initially introduced by R. Meyer and R. Nest and it can be characterized as in Theorem-Definition \ref{theo.TorsionFreeMeyerNest} below (see \cite{MeyerNestTorsion,MeyerNestHomological2} and \cite{VoigtBaumConnesAutomorphisms} for more details). 
		\begin{defi}\label{defi.CompactQuantumActionTor}
			Let $\mathbb{G}=(C(\mathbb{G}),\Delta)$ be a compact quantum group. A torsion action of $\mathbb{G}$ or a torsion for $\widehat{\mathbb{G}}$ is a left $\mathbb{G}$-C$^*$-algebra $(A, \delta)$ with $A$ finite dimensional and $\delta$ ergodic. We say that $(A, \delta)$ is a torsion action of \emph{permutation type} if $A$ is not simple. We say that $(A, \delta)$ is a torsion action of \emph{projective type} if $A$ is simple. The set of all equivariant Morita equivalence classes of torsion actions of $\mathbb{G}$ is denoted by $\text{Tor}(\widehat{\mathbb{G}})$.
	\end{defi}

\begin{rem}
If $\widehat{\mathbb{G}}$ is a classical discrete group $\Gamma$, then $\text{Tor}(\Gamma)$ detects the torsion in $\Gamma$, hence the notational use in general of the dual discrete quantum group. 
\end{rem}

%KDC: Added some more words for last example.
	\begin{exs}\label{ex.ActionsCQG}
		\begin{enumerate}
			\item The trivial action $(\mathbb{C}, trv.)$ is of course a torsion action of any compact quantum group $\mathbb{G}$. 
			\item If $\widehat{\mathbb{H}}<\widehat{\mathbb{G}}$ is a discrete quantum subgroup of $\widehat{\mathbb{G}}$, we have by definition an inclusion of C$^*$-algebras $C(\mathbb{H})\overset{\iota}{\subset} C(\mathbb{G})$ intertwining the corresponding co-multiplications. Therefore, if $(B,\beta)$ is a $\mathbb{H}$-C$^*$-algebra, we can obviously extend $\beta$ (by composing with $\iota$) into an action of $\mathbb{G}$ on $B$, which is denoted by $\widetilde{\beta}$. We denote by $Ind^{\mathbb{G}}_{\mathbb{H}}(B, \beta)$ the same C$^*$-algebra $B$ but equipped with the composition $\widetilde{\beta}:=(id_B\otimes\iota)\circ\beta$ as an action of $\mathbb{G}$. Observe that if $(B,\beta)$ is a torsion action of $\mathbb{H}$, then $Ind^{\mathbb{G}}_{\mathbb{H}}(B, \beta)$ is a torsion action of $\mathbb{G}$.
			\item Let $\widehat{\mathbb{G}}$ be a discrete quantum group that has a non-trivial finite discrete quantum subgroup, say $\widehat{\mathbb{H}}<\widehat{\mathbb{G}}$. Then $(C(\mathbb{H}), \Delta_{\mathbb{H}})$ defines a non-trivial torsion action of $\mathbb{G}$.
			\item If $u\in \mathcal{B}(H_u)\otimes C(\mathbb{G})$ is a representation of $\mathbb{G}$ on a finite dimensional Hilbert space $H_u$, then it defines an action of $\mathbb{G}$ on $\mathcal{B}(H_u)$ given by: 
			$$
			\begin{array}{rcclccl}
				Ad_u:&\mathcal{B}(H_u)& \longrightarrow &\mathcal{B}(H_u)\otimes C(\mathbb{G}),&T & \longmapsto &Ad_u(T):=u(T\otimes 1_{C(\mathbb{G})})u^*.
			\end{array}
			$$
			
			It is clear that $\mathcal{B}(H)^{Ad_u}=End_{\mathbb{G}}(u)$. Hence, the pair $(\mathcal{B}(H_u), Ad_u)$ is a torsion action of $\mathbb{G}$ if and only if $u$ is irreducible.
			\item Consider the rotation group $SO(3)$. Recall that $SO(3)\cong SU(2)/\mathbb{Z}_2$, where $\mathbb{Z}_2\cong Z(SU(2))$ is the center of $SU(2)$. Then the conjugation action of $SU(2)$ on $\mathcal{M}_2(\mathbb{C})$ descends to a torsion action of projective type $\delta$ of $SO(3)$ on $\mathcal{M}_2(\mathbb{C})$. Similar considerations can be made for $SO_q(3)$ with $q\in (-1,1)\backslash \{0\}$ (see \cite{SoltanSO3} for more details).					
		\end{enumerate}
	\end{exs}
	
	The following characterisation of torsion-freeness for discrete quantum groups is well-known. A full proof can be found in \cite[Theorem 1.6.1.4]{RubenThesis}.
	\begin{theodefi}\label{theo.TorsionFreeMeyerNest}
			Let $\mathbb{G}$ be a compact quantum group. We say that $\widehat{\mathbb{G}}$ is torsion-free if one of the following equivalent conditions hold:
				\begin{enumerate}[i)]
					\item Any torsion action of $\mathbb{G}$ is $\mathbb{G}$-equivariantly Morita equivalent to the trivial $\mathbb{G}$-C$^*$-algebra $\mathbb{C}$.
					\item Every finite dimensional $\mathbb{G}$-C$^*$-algebra is $\mathbb{G}$-equivariantly isomorphic to a direct sum of $\mathbb{G}$-C$^*$-algebras which are $\mathbb{G}$-equivariantly Morita equivalent to the trivial $\mathbb{G}$-C$^*$-algebra $\mathbb{C}$.
					\item Every torsion action of $\mathbb{G}$ of permutation type and every torsion action of $\mathbb{G}$ of projective type is $\mathbb{G}$-equivariantly Morita equivalent to the trivial $\mathbb{G}$-C$^*$-algebra $\mathbb{C}$.
				\end{enumerate}
		\end{theodefi}
	In view of characterisation $(iii)$ of the previous theorem, we give the following definition. %KDC: added missing verb
	\begin{defi}
		Let $\mathbb{G}$ be a compact quantum group. We say that $\widehat{\mathbb{G}}$ is \emph{permutation torsion-free} if every torsion action of $\mathbb{G}$ of permutation type is $\mathbb{G}$-equivariantly Morita equivalent to the trivial $\mathbb{G}$-C$^*$-algebra $\mathbb{C}$. We say that $\widehat{\mathbb{G}}$ is \emph{projective torsion-free} if every torsion action of $\mathbb{G}$ of projective type is $\mathbb{G}$-equivariantly Morita equivalent to the trivial $\mathbb{G}$-C$^*$-algebra $\mathbb{C}$.
	\end{defi}
%KDC: removed the word `Indeed', since it was unnecessary
	\begin{ex}\label{ex.ProjectiveTorsionSO(3)}
		The $SO(3)$-C$^*$-algebra $(\mathcal{M}_2(\mathbb{C}), \delta)$ introduced in Examples \ref{ex.ActionsCQG} is a torsion action of $SO(3)$ of projective type. Notice that $(\mathcal{M}_2(\mathbb{C}), \delta)$ is not $SO(3)$-equivariantly Morita equivalent to $\mathbb{C}$, as there are no irreducible $2$-dimensional $SO(3)$-representations to implement this equivalence. Hence, $\widehat{SO(3)}$ is not torsion-free. Moreover, $(\mathcal{M}_2(\mathbb{C}), \delta)$ is the only, up to equivariant Morita equivalence, non-trivial torsion action of $SO(3)$. Similar considerations can be made for $SO_q(3)$ with $q\in (-1,1)\backslash \{0\}$; namely $\widehat{SO_q(3)}$ is not torsion-free and $SO_q(3)$ has only one, up to equivariant Morita equivalence, non-trivial torsion action, which is of projective type (see for instance \cite{VoigtBaumConnesAutomorphisms}). 
	\end{ex}

\section{\textsc{Projective representation theory for compact quantum groups}}	

In this section, we develop the theory of projective representations for compact quantum groups based on the notion of (measurable) $2$-cocycle. We obtain a projective representation theory analogous to the one for classical compact groups. Namely, given a  $2$-cocycle, we construct the associated projective regular representation containing all irreducible $\Omega$-twisted representations and reaching thus a twisted version of the Peter-Weyl theorem. The content of this section concerns a particular case of the more general framework developed in \cite{KennyProjective} by the first author, but we give more attention here to the associated C$^*$-algebraic theory.
	
	\subsection{Projective representations of compact quantum groups}\label{sec.ProjRepCQG}

\begin{defi}
Let $\mathbb{G}$ be a compact quantum group. A \emph{measurable left projective representation} of $\mathbb{G}$  consists of a Hilbert space $H$ and a (measurable) right coaction $\delta: \mathcal{B}(H) \rightarrow \mathcal{B}(H)\overline{\otimes} L^{\infty}(\mathbb{G})$. A \emph{continuous left projective representation} of $\mathbb{G}$  consists of a Hilbert space $H$ and a (continuous) right coaction $\delta: \mathcal{K}(H) \rightarrow \mathcal{K}(H)\otimes C(\mathbb{G})$.
\end{defi}

%KDC: removed the wrong statement that $C(\mathbb{G})^{**}=\lambda(C(\mathbb{G}))''$, and slightly adapted the argument.
Note that any continuous coaction  $\delta: \mathcal{K}(H) \rightarrow \mathcal{K}(H)\otimes C(\mathbb{G})$ extends uniquely to a (normal) coaction $\delta: \mathcal{B}(H) = M(\mathcal{K}(H)) \rightarrow \mathcal{B}(H)\overline{\otimes} L^{\infty}(\mathbb{G})$. Indeed, since $\delta$ is by definition non-degenerate, and since $\mathcal{K}(H)^{**}=\mathcal{B}(H)$, we have (\cite{DixmierCNew}) a unique normal unital $*$-homomorphism:
\[
\mathcal{B}(H) = \mathcal{K}(H)^{**} \rightarrow (\mathcal{K}(H) \otimes C(\mathbb{G}))^{**} \rightarrow \mathcal{B}(H) \overline{\otimes} L^{\infty}(\mathbb{G}),
\]
which extends $\delta$. Hence a continuous projective representation can be seen as a special type of measurable projective representation. On the other hand, any measurable left projective representation $\delta$ on a finite dimensional Hilbert space $H$ is automatically continuous: We can endow $\mathcal{B}(H)$ with a Hilbert space structure for which $\delta$ becomes a finite dimensional  representation, hence its matrix coefficients lie in $\text{Pol}(\mathbb{G}) \subset C(\mathbb{G})$. On the other hand, it is not true that a general measurable left projective representation is automatically continuous, as we will comment on later. 

\begin{rem}
There is also an obvious notion of right projective representation.  Identifying $j: \mathcal{B}(H)^{op} \cong \mathcal{B}(\overline{H})$ via $j(x) = \overline{x^*}$, there is a natural correspondence between left and right measurable/continuous projective representations by $\delta  \leftrightarrow \overline{\delta}:= \Sigma(j\otimes R)\delta j^{-1} $, so we consider $\overline{\delta}$ as a right continuous projective representation of $\mathbb{G}$ on $\overline{H}$. More directly, one can also view a left projective representation of $\mathbb{G}$ as a right projective representation of $\mathbb{G}^{cop}$. 
\end{rem}

Recall from the introduction that any continuous action of a classical compact group $G$ on $\mathcal{K}(H)$, for some Hilbert space $H$, is implemented by an $\omega$-representation of $G$ on $H$, where $\omega$ is a measurable $2$-cocycle on $G$. The same in fact holds for measurable actions of $G$ on $\mathcal{B}(H)$. The main goal of this section will be to show that these statements are still true for compact quantum groups. This will in particular justify the terminology \emph{projective torsion action} (recall Theorem \ref{theo.TorsionFreeMeyerNest}). 

\begin{defi}
Let $\delta$ be a measurable left projective representation. We say that $\delta$ is \emph{cleft} if there exists a unitary $u \in \mathcal{B}(H)\overline{\otimes} L^{\infty}(\mathbb{G})$ such that  $\delta(a) = u(a\otimes 1)u^*$.
\end{defi}

Similarly, we say that a measurable right projective representation $\delta$ is cleft if there exists a unitary $u \in \mathcal{B}(H)\overline{\otimes} L^{\infty}(\mathbb{G})$ such that  $\delta(a) = \Sigma(u^*(a\otimes 1)u)$. Clearly $\delta$ is cleft if and only if $\overline{\delta}$ is cleft. We call $u$ an \emph{implementing unitary}. 

We will show in Theorem \ref{theo.vNCleft} that all  measurable projective representations are cleft. We start with the following elementary well-known result, which is a verison of the Skolem-Noether theorem in the setting of von Neumann algebras.

%KDC: specified that the copies have to be unital. Added some precisions.
\begin{lem}\label{lem.UnitaryImplementationAction}
	Let $M$ be a von Neumann algebra and $k\in\mathbb{N}$. If $\{e_{i j}\}_{i,j=1,\ldots, k}$ and $\{f_{i j}\}_{i,j=1,\ldots, k}$ are the matrix units of two unital copies of $\mathcal{M}_k(\mathbb{C})$ inside $M$, then there exists a unitary $U$ in $M$ such that $Ue_{i j}U^*=f_{i j}$, for all $i,j=1,\ldots, k$. Moreover, $U$ is unique up to a unitary in $\{e_{i j}\}'_{i,j=1,\ldots, k}\cap M$.
\end{lem}
	\begin{proof}
		Let $p$ be the maximal properly infinite projection of $M$, and put $q:=1-p$. We split $M$ into its finite and its properly infinite part, $M=qM\oplus pM$. Next, note that $e_{11}$ and $f_{11}$ have central support $1$ and $k[e_{11}]=[1]=k[f_{11}]$ in $K_0(M)$. Then, by taking the center valued trace on $qM$, we deduce that $qe_{11}$ and $qf_{11}$ are Murray-von Neumann equivalent in $qM$ (cf. \cite[Corollary 2.8]{Takesaki1}, for instance). It is obvious that $pe_{11}$ and $pf_{11}$ are equivalent in $pM$. Hence $e_{11}$ and $f_{11}$ are Murray-von Neumann equivalent by a partial isometry $u \in M$. Then $U = \sum_s f_{s1}ue_{1s}$ is the sought-after unitary. The stated uniqueness of $U$ is clear.
	\end{proof}
		
	\begin{theo}\label{theo.vNCleftProjectiveTorsion}
		Let $\mathbb{G}$ be a compact quantum group. Then any finite dimensional projective representation is cleft. 
	\end{theo}
	\begin{proof}
Let $\delta: \mathcal{B}(H) \rightarrow \mathcal{B}(H) \otimes L^{\infty}(\mathbb{G})$ be a right coaction with $H$ finite dimensional. Applying Lemma \ref{lem.UnitaryImplementationAction} with respect to the matrix units $\delta(e_{ij})$ and $e_{ij}\otimes 1$ provides a unitary $u\in \mathcal{B}(H)\otimes L^\infty(\mathbb{G})$ implementing the coaction $\delta$, hence $\delta$ is cleft. 
\end{proof}

To extend this result to arbitrary  projective representations, we first take a further look at implementing unitaries. Note that if $\delta: \mathcal{B}(H) \rightarrow \mathcal{B}(H)\overline{\otimes} L^{\infty}(\mathbb{G})$ is a measurable projective representation implemented by $u \in  \mathcal{B}(H)\overline{\otimes} L^{\infty}(\mathbb{G})$, then \begin{equation*}\begin{split}
			(\delta\otimes id)\delta(a)&=(\delta\otimes id)\Big(u(a\otimes 1)u^*\Big)=u_{12}\Big(u(a\otimes 1)u^*\Big)_{13}u^*_{12}=u_{12}u_{13}(a\otimes 1\otimes 1)u^*_{13}u^*_{12},
		\end{split}
		\end{equation*}
while
		\begin{equation*}
		\begin{split}
			(id\otimes \Delta)\delta(a)&=(id\otimes \Delta)\Big(u(a\otimes 1)u^*\Big)=(V_{\mathbb{G}})_{23}\Big(u(a\otimes 1)u^*\Big)_{12}(V_{\mathbb{G}})^*_{23}\\
			&=(V_{\mathbb{G}})_{23}u_{12}(a\otimes 1\otimes 1)u^*_{12}(V_{\mathbb{G}})^*_{23}.
		\end{split}
		\end{equation*}
Hence by the coaction property $(\delta\otimes id)\delta=(id\otimes \Delta)\delta$, we obtain $u_{12}u_{13}(a\otimes 1\otimes 1)u^*_{13}u^*_{12}=(V_{\mathbb{G}})_{23}u_{12}(a\otimes 1\otimes 1)u^*_{12}(V_{\mathbb{G}})^*_{23}$, for all $a\in \mathcal{B}(H)$. In other words, \\$u^*_{13}u^*_{12}(V_{\mathbb{G}})_{23}u_{12}\in \big(\mathcal{B}(H)\overline{\otimes} \mathbb{C}\overline{\otimes} L^\infty(\mathbb{G})'\big)'=\mathbb{C}\overline{\otimes} \mathcal{B}(L^2(\mathbb{G}))\overline{\otimes} L^\infty(\mathbb{G})$, so there exists a unitary $X\in \mathcal{B}(L^2(\mathbb{G}))\overline{\otimes} L^\infty(\mathbb{G})$ such that $(V_{\mathbb{G}})_{23}u_{12}=u_{12}u_{13}X_{23}$. This relation allows to write the following: $(id\otimes \Delta)(u)=(V_{\mathbb{G}})_{23}u_{12}(V_{\mathbb{G}})_{23}^*=u_{12}u_{13}X_{23}(V_{\mathbb{G}})_{23}^*$.

%KDC: added overlines
		Since the left hand side is an element of $\mathcal{B}(H)\overline{\otimes} L^\infty(\mathbb{G})\overline{\otimes} L^\infty(\mathbb{G})$ and $u$ is a unitary, then $\widetilde{\Omega}:=u_{13}^*u_{12}^*\Delta_{23}(u_{12})\in \mathbb{C} \otimes L^\infty(\mathbb{G})\overline{\otimes} L^\infty(\mathbb{G})$ and $\widetilde{\Omega}=1\otimes \Omega$ with $\Omega\in L^\infty(\mathbb{G})\overline{\otimes} L^\infty(\mathbb{G})$ unitary. Moreover, we have $\Delta_{23}(u_{12})\cdot \Omega^*_{23}=u_{12}u_{13}$. Applying this equation to the identity $(u_{12}u_{13})u_{14}=u_{12}(u_{13}u_{14})$, we obtain that $\Omega^*$ satisfies the \emph{$2$-cocycle relation}: $$(\Delta\otimes id)(\Omega^*)\Omega_{12}^*=(id\otimes \Delta)(\Omega^*)\Omega_{23}^*.$$

Let us formalize this in the following definition. 

\begin{defi}
		Let $\mathbb{G}$ be a compact quantum group. A (measurable, unitary) $2$-cocycle on $\mathbb{G}$ is a unitary element $\Omega\in L^{\infty}(\mathbb{G})\overline{\otimes} L^{\infty}(\mathbb{G})$ such that:
			$$(\Omega \otimes 1)(\Delta\otimes id)(\Omega)=(1\otimes \Omega)(id\otimes \Delta)(\Omega).$$
			
			Two $2$-cocycles $\Omega$ and $\Omega'$ on $\mathbb{G}$ are said to be coboundary equivalent if there exists a unitary $X\in L^{\infty}(\mathbb{G})$ such that $\Omega'=(X^*\otimes X^*)\Omega\Delta(X)$.
	\end{defi}
	\begin{note}
		If one replaces $L^{\infty}(\mathbb{G})$ by $C(\mathbb{G})$ or $\text{Pol}(\mathbb{G})$, then we define analogously a \emph{continuous} or \emph{algebraic} $2$-cocycle on $\mathbb{G}$, respectively. 
	\end{note}

		Given a $2$-cocycle $\Omega$ on $\mathbb{G}$ we can define the following linear maps
		$$
			\begin{array}{rcclcrccl}
				{}_{\Omega}\Delta:&L^{\infty}(\mathbb{G})& \longrightarrow & L^{\infty}(\mathbb{G})\overline{\otimes} L^{\infty}(\mathbb{G}) &,\ & \Delta_{\Omega^*}:&L^{\infty}(\mathbb{G})& \longrightarrow & L^{\infty}(\mathbb{G})\overline{\otimes} L^{\infty}(\mathbb{G})\\
				&x& \longmapsto & {}_{\Omega}\Delta(x):=\Omega\cdot \Delta(x) & & &x& \longmapsto & \Delta_{\Omega^*}(x):=\Delta(x)\cdot \Omega^*.
			\end{array}
		$$
		We call ${}_{\Omega}\Delta$ and $\Delta_{\Omega^*}$ the \emph{right/left twisted pseudo co-multiplication on $\mathbb{G}$ with respect to $\Omega$} or the \emph{right/left $\Omega$-pseudo co-multiplication on $\mathbb{G}$}; respectively. Observe that both ${}_{\Omega}\Delta$ and $\Delta_{\Omega^*}$ are linear maps satisfying the following identities:
		\begin{enumerate}[i)]
			\item ${}_{\Omega}\Delta(xy)={}_{\Omega}\Delta(x)\Delta(y)$ and $\Delta_{\Omega^*}(xy)=\Delta(x)\Delta_{\Omega^*}(y)$, for all $x,y\in L^{\infty}(\mathbb{G})$,
			\item ${}_{\Omega}\Delta(x)^*{}_{\Omega}\Delta(y)=\Delta(x^*y)$ and $\Delta_{\Omega^*}(x)\Delta_{\Omega^*}(y)^*=\Delta(xy^*)$, for all $x,y\in L^{\infty}(\mathbb{G})$,
			\item $({}_{\Omega}\Delta\otimes id){}_{\Omega}\Delta=(id\otimes {}_{\Omega}\Delta){}_{\Omega}\Delta$ and $(\Delta_{\Omega^*}\otimes id)\Delta_{\Omega^*}=(id\otimes\Delta_{\Omega^*})\Delta_{\Omega^*}$; and 
			\item $\overline{span}^{\sigma-\textrm{weak}}\{{}_{\Omega}\Delta(x)(y\otimes z)\ |\ x, y, z\in L^{\infty}(\mathbb{G})\}=L^{\infty}(\mathbb{G})\overline{\otimes} L^{\infty}(\mathbb{G})$ and $\overline{span}^{\sigma-\textrm{weak}}\{(y\otimes z)\Delta_{\Omega^*}(x)\ |\ x, y, z\in L^{\infty}(\mathbb{G})\}=L^{\infty}(\mathbb{G})\overline{\otimes} L^{\infty}(\mathbb{G})$.
		\end{enumerate}
These identities are obtained after a straightforward computation. Hence, ${}_{\Omega}\Delta$ and $\Delta_{\Omega^*}$ are in particular von Neumann algebraic analogues of module coalgebras, and form  concrete instances of the notion of \emph{Galois co-object} introduced in \cite{KennyComonoidalMorita}.

	\begin{defi}\label{defi.TwistedRep}
		Let $\mathbb{G}$ be a compact quantum group and $\Omega$ a $2$-cocycle on $\mathbb{G}$. A (measurable) $\Omega$-representation of $\mathbb{G}$ on a Hilbert space $H$ is a unitary element $u\in \mathcal{B}(H)\overline{\otimes} L^{\infty}(\mathbb{G})$  such that $(id\otimes {}_{\Omega}\Delta)(u)=u_{12}u_{13}$. A (measurable)  $\Omega^*$-representation on $H$  is a unitary element $u\in \mathcal{B}(H) \overline{\otimes}
L^{\infty}(\mathbb{G})$ satisfying $(id\otimes\Delta_{\Omega^*})(u)=u_{12}u_{13}$.
\end{defi}

The elements of the form $u_{\xi,\zeta}:=(\omega_{\xi,\zeta}\otimes 1)(u)\in L^{\infty}(\mathbb{G})$ with $\xi,\zeta\in H$ are called matrix coefficients of $u$. In particular, if we fix an orthonormal basis $\{\xi_i\}_{i}$ in $H$, we write $u_{i j}:=u_{\xi_i,\xi_j}$. Then for an $\Omega$-representation $u$ we obtain the usual corepresentation identities ${}_{\Omega}\Delta(u_{ij})= \underset{k}{\sum} u_{ik}\otimes u_{kj}$, for all $i,j$, where the sum converges in (say) the strong operator topology. The same conclusion holds for $\Omega^*$-representations.

Note that if the $u_i$ are measurable  $\Omega$-representations on Hilbert spaces $H_i$, then clearly $u = \oplus_i u_i$ is a measurable  $\Omega$-representation on $H = \oplus_iH_i$, called the direct sum  $\Omega$-representation.

Summarizing the discussion following Theorem \ref{theo.vNCleftProjectiveTorsion}, we obtain the following result. 

\begin{pro}
Let $\delta: \mathcal{B}(H) \rightarrow \mathcal{B}(H) \overline{\otimes} L^{\infty}(\mathbb{G})$ be a cleft measurable projective representation. Then  there exists a $2$-cocycle $\Omega \in L^{\infty}(\mathbb{G})\overline{\otimes}L^{\infty}(\mathbb{G})$ and a  $\Omega^*$-representation  $u\in \mathcal{B}(H) \overline{\otimes} L^{\infty}(\mathbb{G})$ such that $\delta(a) = u(a\otimes 1)u^*$. 
\end{pro}

Conversely, if $u$ is a  $\Omega^*$-representation, we obtain a measurable right coaction: 
\begin{equation}\label{eq.deltau2}
\delta_u: \mathcal{B}(H) \rightarrow  \mathcal{B}(H)\overline{\otimes} L^{\infty}(\mathbb{G}), \qquad \delta_u(a) = u(a\otimes 1)u^*,\qquad a\in \mathcal{B}(H),
\end{equation}
where the coaction property follows immediately from the $\Omega^*$-representation property of $u$. Similarly, any $\Omega$-representation $u$ provides a measurable left coaction:
\begin{equation}\label{eq.deltau}
\delta_u: \mathcal{B}(H) \rightarrow L^{\infty}(\mathbb{G}) \overline{\otimes} \mathcal{B}(H), \qquad \delta_u(a) = \Sigma (u^*(a\otimes 1)u),\qquad a\in \mathcal{B}(H).
\end{equation}

\begin{note}\label{note.UnicityImplementation}
If $v$ is another implementing unitary for $\delta$, we see that $v^*u(a\otimes 1)u^*v = a\otimes 1$ for all $a\in \mathcal{B}(H)$, hence $v = u(1\otimes X)$ for some unitary $X \in L^{\infty}(\mathbb{G})$. If $v$ has associated $2$-cocycle $\Omega'$, it then follows from the projective representation identities for $v$ and $u$ that $\Omega'= (X^*\otimes X^*)\Omega\Delta(X)$, so $\Omega$ and $\Omega'$ are cohomologous. 
\end{note}

\begin{lem}\label{lem.indsyst}
Let $\delta: \mathcal{B}(H)\rightarrow \mathcal{B}(H)\overline{\otimes} L^{\infty}(\mathbb{G})$ be a measurable projective representation on $H$. Assume that $p \in \mathcal{B}(H)$ is a non-zero coinvariant projection, and let $\delta_p: \mathcal{B}(pH)\rightarrow \mathcal{B}(pH)\overline{\otimes} L^{\infty}(\mathbb{G})$ be the restriction of $\delta$. Then $\delta$ is cleft if and only if $\delta_p$ is cleft. Moreover, if $\delta_p$ is implemented by an $\Omega^*$-representation $v$ for some $2$-cocycle $\Omega$, then $\delta$ can be implemented by some $\Omega^*$-representation $u$ such that $u(p\otimes 1) = (p\otimes 1) u = v$. 
\end{lem}
\begin{proof}
Assume that $\delta$ is cleft, with implementing unitary $u$. Then as $\delta(p) = p\otimes 1$, it follows immediately that $u$ commutes with $p$, hence $v = (p\otimes 1)u =u(p\otimes 1)$ is a unitary in $\mathcal{B}(pH)\overline{\otimes} L^{\infty}(\mathbb{G})$ implementing $\delta_p$. This shows that $\delta_p$ is cleft. 

Conversely, assume that $\delta_p$ is cleft. Then $\mathcal{B}(H)$ contains a minimal projection $e$ such that $\delta(e)$ and $e\otimes 1$ are unitarily equivalent in $\mathcal{B}(H)\otimes M$. The same reasoning as in Theorem \ref{theo.vNCleftProjectiveTorsion} then shows that $\delta$ is cleft. 

Assume now that $v$ is a unitary in $\mathcal{B}(pH)\overline{\otimes}L^{\infty}(\mathbb{G})$ implementing $\delta_p$ and $\widetilde{u} \in \mathcal{B}(H)\overline{\otimes}L^{\infty}(\mathbb{G}) $ a unitary implementing $\delta$. Assume that $\Omega$ is the $2$-cocycle associated to $v$, and $\widetilde{\Omega}$ the $2$-cocycle associated to $\widetilde{u}$. Then it is easily seen that $(p\otimes 1)\widetilde{u}$ is a $\widetilde{\Omega}^*$-representation implementing $\delta_p$. Hence $\Omega$ and $\widetilde{\Omega}$ are cohomologous, say $\widetilde{\Omega} = (X^*\otimes X^*)\Omega\Delta(X)$ where $\widetilde{u}(p\otimes 1) = v(1\otimes X)$. Hence $u = \widetilde{u}(1\otimes X^*)$ is an $\Omega^*$-representation implementing $\delta$ with $u(p\otimes 1) =v$.  
\end{proof} 

%KDC: Actually, it is not really true that the general theorem was proven independently, since we use the corollary of the next theorem in the very last step. Hence, I reformulated.
We can now generalize Theorem \ref{theo.vNCleftProjectiveTorsion} to the infinite dimensional setting for \emph{continuous} projective representations. As mentioned, we will later then strengthen this and show that this holds in fact in general.

\begin{theo}\label{theo.ContProjCleft} Let $\mathbb{G}$ be a compact quantum group. Then any continuous projective representation is cleft. 
\end{theo}
\begin{proof}
Assume that $\delta: \mathcal{K}(H) \rightarrow \mathcal{K}(H)\otimes C(\mathbb{G})$ is a continuous projective representation. Recall from Remark \ref{rem.nondegact} that $\mathcal{K}(H)^{\delta}$ acts non-degenerately on $H$. Hence, as $\mathcal{K}(H)^{\delta}$ is a (separable) C$^*$-algebra of compact operators, we can find an ascending sequence $p_i$ of finite rank projections in $\mathcal{K}(H)^{\delta}$ converging strongly to $1$. By Theorem \ref{theo.vNCleftProjectiveTorsion} and Lemma \ref{lem.indsyst}, we can find a $2$-cocycle $\Omega \in L^{\infty}(\mathbb{G})\overline{\otimes} L^{\infty}(\mathbb{G})$ and $\Omega^*$-representations $u_i \in \mathcal{K}(p_iH)\otimes L^{\infty}(\mathbb{G})$ such that $u_i$ implements $\delta_{p_i}$ and such that $u_i(p_j\otimes 1) = u_j$ for $j\leq i$. Then clearly the $u_i$ converge $\sigma$-strongly to a unitary $u \in \mathcal{B}(H)\overline{\otimes} L^{\infty}(\mathbb{G})$, and $u$ is an $\Omega^*$-representation implementing $\delta$. 
\end{proof}

Let now $\delta: \mathcal{B}(H) \rightarrow \mathcal{B}(H) \overline{\otimes} L^{\infty}(\mathbb{G})$ be a measurable projective representation of $\mathbb{G}$, and consider in this setting the averaging operator: 
\[
E_{\delta}: \mathcal{B}(H) \rightarrow \mathcal{B}(H),\quad x\mapsto (id\otimes h_{\mathbb{G}})\delta(x). 
\]
Then $E_{\delta}$ is a normal conditional expectation on $\mathcal{B}(H)^{\delta}$, and it is well-known that then necessarily $\mathcal{B}(H)^{\delta}$ is a (possibly infinite) direct sum of type $I$ factors. In particular, $\mathcal{B}(H)^{\delta}$ contains a minimal projection $p$, and $\delta_p$ is then an irreducible projective representation, meaning $\mathcal{B}(pH)^{\delta_p} = \mathbb{C} p$. This leads to the following:

\begin{cor}\label{cor.projKacclef}
If $\mathbb{G}$ is of Kac type, then all measurable projective representations of $\mathbb{G}$ are continuous, hence cleft.
	\end{cor}
\begin{proof}
Let $\delta:  \mathcal{B}(H) \rightarrow \mathcal{B}(H) \overline{\otimes} L^{\infty}(\mathbb{G})$ be a measurable projective representation of $\mathbb{G}$. If $p_i$ are a maximal set of orthogonal minimal projections in $\mathcal{B}(H)^{\delta}$, then each $\delta_{p_i}$ is an irreducible projective representation. By \cite[Corollary 5.2]{KennyProjective}, it follows that each $p_iH$ is finite dimensional, and in particular each $\delta_{p_i}$ is continuous. Since $\mathcal{K}(H)$ is the directed union of all $\mathcal{K}(qH)$ with $q$ a finite sum of $p_i$, it follows that $\delta$ is continuous, and in particular cleft. 
\end{proof} 

Recall now that any compact quantum group allows a maximal compact quantum subgroup of Kac type \cite[Appendix A.1]{SoltanKac}. We will slightly modify this construction as follows. 

\begin{lem}
Let $\mathbb{G}$ be a compact quantum group. Let $p$ be the maximal properly infinite projection of $M =L^{\infty}(\mathbb{G})$, and put $q = 1-p$. Then $qM$ defines a von Neumann algebraic compact quantum group of Kac type with coproduct $\Delta_q(x) = \Delta(x)(q\otimes q)$. 
\end{lem} 
We call $qM$ the \emph{normal Kac quotient} of $M$. 
\begin{proof} 
Note that $pM = \{x\in M \mid \tau(x^*x) =0 \textrm{ for all normal tracial states }\tau\}$. Since the convolution product of two normal tracial states is still a normal tracial state, it follows that $\Delta$ descends to a coproduct $\Delta_q$ on $qM = M/pM$. We are to show that $qM$ has an invariant normal tracial state. 

Note first that $p$ is invariant under the scaling group and the unitary antipode. It follows that $pM \cap \text{Pol}(\mathbb{G})$ is preserved under the antipode, hence $\text{Pol}(\mathbb{H}) := q \text{Pol}(\mathbb{G}) \cong \text{Pol}(\mathbb{G})/\text{Pol}(\mathbb{G})\cap pM$ is a Hopf $*$-algebra with respect to $\Delta_q$. As it spanned by matrix coefficients of finite-dimensional corepresentations, it defines indeed a compact quantum group $\mathbb{H}$. As $\text{Pol}(\mathbb{H})$ has a separating family of tracial states by construction, it follows that $\mathbb{H}$ must be of  Kac type.

Write $\pi: \text{Pol}(\mathbb{G}) \rightarrow \text{Pol}(\mathbb{H})$ for the natural quotient map, and consider $N = L^{\infty}(\mathbb{H})$. As $\mathbb{H}$ is a compact quantum subgroup of $\mathbb{G}$, we have a normal coaction $\alpha: M \rightarrow M\bar{\otimes} N$ restricting to $(id\otimes \pi)\Delta$ on $\text{Pol}(\mathbb{G})$. As the Haar state of $\mathbb{H}$ is tracial, $\alpha$ descends to a normal coaction  $\alpha_q: qM \rightarrow qM \bar{\otimes} N$. Moreover, we have
\begin{equation}\label{EqCommute}
(\Delta_q \otimes id)\alpha_q = (id\otimes \alpha_q)\Delta_q.
\end{equation}
Let $E: qM \rightarrow qM$ be given by $E(x) =  (id\otimes h_{\mathbb{H}})\alpha_q(x)$, for all $x\in qM$, where $h_{\mathbb{H}} \in N_*$ is the Haar state for $\mathbb{H}$. Since $E$ is normal and $E(x) = h_{\mathbb{H}}(x)1$ for $x\in \text{Pol}(\mathbb{H})$, we have by $\sigma$-weak density of $\text{Pol}(\mathbb{H})$ in $qM$ that there exists a normal state $h_{qM}$ on $qM$ with $E(x) = h_{qM}(x)1$, for all $x\in qM$. From \eqref{EqCommute}, it then easily follows that $h_{qM}$ is left invariant, i.e.\ $(id\otimes h_{qM})(\Delta_q(x)) =h_{qM}(x)$, for all $x\in qM$. As the unitary antipode of $M$ descends to $qM$, we also have that $qM$ has a right invariant normal state, hence $(qM,\Delta_q)$ defines a compact quantum group in its own right. It is then clear that $(qM,\Delta_q)$ in fact equals $L^{\infty}(\mathbb{H})$, and in particular defines a compact quantum group of Kac type. 
\end{proof}

We are now ready to prove the main theorem of this section.

	\begin{theo}\label{theo.vNCleft}
Let $\mathbb{G}$ be a compact quantum group. Then all measurable projective representations of $\mathbb{G}$ are cleft.
\end{theo} 
	
\begin{proof}
Let $\delta:  \mathcal{B}(H) \rightarrow \mathcal{B}(H) \overline{\otimes} L^{\infty}(\mathbb{G})$ be a measurable projective representation of $\mathbb{G}$. To show that $\delta$ is cleft, we may by the discussion before Corollary \ref{cor.projKacclef} assume that $\delta$ is ergodic. As we are assuming that $C(\mathbb{G})$ is separable, it then follows in particular that $H$ is separable. We can then moreover find a unique normal state $\Phi$ on $\mathcal{B}(H)$ such that $\Phi(x)1 = E_{\delta}(x)$ for all $x\in \mathcal{B}(H)$. We necessarily have that $\Phi \leq \mathrm{Tr}$, with $\mathrm{Tr}$ the usual trace of $\mathcal{B}(H)$. 

Let $e$ be a minimal projection in $\mathcal{B}(H)$. It is sufficient to show that $\delta(e)$ and $e\otimes 1$ are unitarily equivalent in $\mathcal{B}(H) \otimes M$, as we can then proceed as in Theorem \ref{theo.vNCleftProjectiveTorsion}.

Let $p$ be the (central) maximal properly infinite projection of $L^{\infty}(\mathbb{G})$, and put $q = 1-p$. We are to show that $\delta(e)(1\otimes p) \sim e\otimes p$ and $\delta(e)(1\otimes q)\sim e\otimes q$, where $\sim$ denotes Murray-von Neumann equivalence. 

%KDC: Added that we take $\tau$ tracial
To show that $\delta(e)(1\otimes p) \sim e\otimes p$, let us show first that $\delta(e)(1\otimes p)$ is properly infinite. Assume this were not the case.  Then there exists a non-zero semi-finite normal tracial weight $\tau$ on $pL^{\infty}(\mathbb{G})$ with $(\mathrm{Tr}\otimes \tau)(\delta(e)(1\otimes p))< \infty$. But the left hand side is larger than $(\Phi\otimes \tau)(\delta(e)(1\otimes p)) = \Phi(e)\tau(p)$, which is infinite since $p$ is maximally properly infinite and $\Phi$ is faithful. This contradiction shows that   $\delta(e)(1\otimes p)$ is necessarily properly infinite. Since  $\delta(e)(1\otimes w)\neq0$ for any non-zero $w$, again using faithfulness of $\Phi$, it follows that the properly infinite projections $\delta(e)(1\otimes p)$ and $e\otimes p$ have the same central support $1\otimes p$, and hence $\delta(e)(1\otimes p)\sim e\otimes p$. 

To show that also $\delta(e)(1\otimes q)\sim e\otimes q$, we note that $x \mapsto \delta(x)(1\otimes q)$ defines a projective representation of the Kac type compact quantum group $(qL^{\infty}(\mathbb{G}),\Delta_q)$. By Corollary \ref{cor.projKacclef} we have that this projective representation is necessarily cleft, which implies that $\delta(e)(1\otimes q)\sim e\otimes q$.
\end{proof} 
	
\subsection{Measurable $\Omega$-representations}\label{sec.omreps}

We recall some of the results of \cite{KennyProjective}, where by Theorem \ref{theo.vNCleft} we can restrict to the cleft case.

	\begin{defi}
		Let $\mathbb{G}$ be a compact quantum group and $\Omega$ a  $2$-cocycle on $\mathbb{G}$. Let $(u,H_u)$ be a  measurable  $\Omega$-representation of $\mathbb{G}$. A (closed) subspace $E\subset H$ is called \emph{$u$-invariant} if $(p_E\otimes 1)u(p_E\otimes 1)=u(p_E\otimes 1)$, where $p_E$ denotes the orthogonal projection from $H$ onto $E$. We say that $u$ is \emph{irreducible} if the only $u$-invariant subspaces are $(0)$ and $H_{u}$, and we say that $u$ is \emph{indecomposable} if $H$ cannot be written as a direct sum of two non-zero $u$-invariant subspaces. 

If $(v,H_v)$ is another  $\Omega$-representation, an intertwiner between $u$ and $v$ is a linear bounded operator $T: H_{u}\longrightarrow H_{v}$ such that $(T\otimes 1)u=v(T\otimes 1)$. The space of all intertwiners between $u$ and $v$ is denoted by $Hom_{\mathbb{G}}(u, v)$. We write $End_{\mathbb{G}}(u)$ if $u=v$.
		
		We say that $(u, H_{u})$ and $(v, H_{v})$ are (unitary) equivalent if $Hom_{\mathbb{G}}(u, v)$ contains a unitary operator.  
	\end{defi}

It is a straightforward computation, using that $u$ and $v$ are unitary, to show that $T^*\in Hom_{\mathbb{G}}(v, u)$ whenever $T\in Hom_{\mathbb{G}}(u, v)$. Then clearly $End_{\mathbb{G}}(u)$ is a von Neumann algebra.

Similar definitions can be stated for  $\Omega^*$-representations.

\begin{defi}
We denote by $\text{Irr}(\mathbb{G}, \Omega)$ and by $\text{Irr}(\mathbb{G}, \Omega^*)$ the set of all  equivalence classes of irreducible  $\Omega$-representations of $\mathbb{G}$ and the one of irreducible  $\Omega^*$-representations of $\mathbb{G}$, respectively.
\end{defi}

If $x\in \text{Irr}(\mathbb{G}, \Omega)$ is such a class, we denote by $u^x\in \mathcal{B}(H_x)\overline{\otimes}L^{\infty}(\mathbb{G})$ a representative of $x$, where $H_x$ denotes the Hilbert space on which $u^x$ acts. We put $n_{x}:=dim(H_{x})$, where possibly $n_x = \infty$. Notice that these notations are similar to the ones used for ordinary representations of $\mathbb{G}$, so that the context will explain in which situation the notations are used.

For further use, let us note here that if $\Omega$ is a  $2$-cocycle for $\mathbb{G}$, then $\Omega^*$ is a  $2$-cocycle for $\mathbb{G}^{op} = (C(\mathbb{G})^{op},\Delta)$. Hence all arguments valid for  $\Omega$-representations are also valid for unitary $\Omega^*$-representations. 

Another way to link up  $\Omega$-representations with $\Omega^*$-representations is given by the following result. 
%will be given by Corollary  \ref{cor.switchOmOmStar}.

		\begin{lem}\label{lem.DecompVTand}
			Let $\mathbb{G}$ be a compact quantum group and $\Omega$ a $2$-cocycle on $\mathbb{G}$. Let $u \in \mathcal{B}(H) \overline{\otimes} L^{\infty}(\mathbb{G})$ be a (measurable)  $\Omega^*$-representation on $H$, and put $\delta = \delta_u: \mathcal{B}(H) \rightarrow \mathcal{B}(H) \overline{\otimes} L^{\infty}(\mathbb{G})$ as in \eqref{eq.deltau2}. With $T = \mathcal{B}(H)$, the unitary operator $V_T\in \mathcal{B}(L^2(T))\otimes C(\mathbb{G})$ implementing the action $\delta$ (recall Remark \ref{rem.ActExtvN}) can then be written as $V_T=u_{13} u^\circ_{23}\in T\otimes T^{op}\otimes C(\mathbb{G})$, where $u^\circ$ is an $\Omega$-representation implementing $\overline{\delta}$ in the sense of \eqref{eq.deltau}.
		\end{lem}
		\begin{proof}
We have by construction that $\delta(t)=u(t\otimes 1)u^*$, for all $t\in T$. On the other hand, we have by Remark \ref{rem.ActExtvN} that there exists a canonical unitary operator $V_T\in \mathcal{B}(L^2(T))\otimes C(\mathbb{G})$ implementing $\delta$, that is, $\delta(t)=V_T(t\otimes 1)V^*_T$, for all $t\in T$. Hence, for all $t\in T$ we write, upon identifying $T\overline{\otimes} T^{op} \cong \mathcal{B}(L^2(T))$:
			\begin{equation*}
			\begin{split}
				V^*_T\underbracket[0.8pt]{u_{13}(t\otimes 1\otimes 1)}=\underbracket[0.8pt]{V^*_T\delta(t)_{13}}u_{13}=(t\otimes 1\otimes 1)V^*_Tu_{13},
			\end{split}
			\end{equation*}
			which shows that there exists a unitary operator $u^{\circ} \in T^{op}\otimes C(\mathbb{G})$ such that $V_T=u_{13}u^{\circ}_{23}$. As $V_T$ is a corepresentation, it is easily seen that $u^{\circ}$ must necessearily be an $\Omega$-representation. Finally, we have by \cite[Proposition 3.7.3]{VaesUnitImpl} that $(J_T\otimes \widehat{J}_{\mathbb{G}})V_T (J_T\otimes \widehat{J}_{\mathbb{G}}) = V_T^*$. Since $\widehat{J}_{\mathbb{G}}$ implements the unitary antipode $R$, and since for $x\in T^{op}$ we have $J_T(1\otimes x)J_T = x^*\otimes 1$ by definition, it then follows that for $x\in T^{op}$ we have $1\otimes (u^\circ)^*(x\otimes 1)u^\circ = V_T^*(1\otimes x \otimes 1)V_T = \overline{\delta}(x)_{32}$. 
		\end{proof}

\begin{cor}\label{cor.opp2coc}
There exists a canonical element $X_{\Omega} \in L^{\infty}(\mathbb{G})$ such that: 
\[
u^{\circ} = (j\otimes R)(u)(1\otimes X_{\Omega}^*)
\]
for all $\Omega^*$-representations. Moreover, $X_{\Omega}$ is then a coboundary between $\Omega$ and the $2$-cocycle $\widetilde{\Omega} = (R\otimes R)\Omega_{21}^*$, so $\widetilde{\Omega} = (X_{\Omega}^*\otimes X_{\Omega}^*)\Omega \Delta(X_{\Omega})$. We obtain in particular a one-to-one correspondence between $\Omega$-representations and  $\Omega^*$-representations by the map $u \mapsto u^{\circ} = (j\otimes R)(u)(1\otimes X_{\Omega}^*)$.
\end{cor} 
Note that the fact that $\Omega$ and $\widetilde{\Omega}$ are canonically coboundary equivalent holds in the general context of locally compact quantum groups, see \cite[Proposition 6.3.(iii)]{KennyGaloisObjTwistings}, but we can give in our setting an easier, more direct proof. It can be shown that the coboundary element obtained here indeed coincides with the one from \cite[Proposition 6.3.(iii)]{KennyGaloisObjTwistings}, but we refrain from showing this explicitly.

\begin{proof}
If $u$ is an  $\Omega^*$-representation, it is easily seen that $u^{\circ}$ and $(j\otimes R)(u)$ both implement the same left coaction on $\mathcal{B}(H)^{op}$, hence by Note \ref{note.UnicityImplementation} we have $u^{\circ} = (j\otimes R)(u)(1\otimes X_{u}^*)$ for some unitary $X_u$. It is also easily seen that $\delta_{(j\otimes R)(u)}$ is cleft with associated $2$-cocycle $\widetilde{\Omega}$, showing that $X_u$ is a coboundary between $\Omega$ and $\widetilde{\Omega}$. 

It remains to show that $X_{u}$ is independent of $u$. But by construction, it is easily shown that $(u\oplus v)^{\circ} = u^{\circ}\oplus v^{\circ}$. It then follows that $X_u = X_{u\oplus v} = X_v$ for any two  $\Omega^*$-representations $u,v$. Namely, by the previous discussion we have $(j\otimes R)(u)=u^{\circ}(1\otimes X_{u})$, $(j\otimes R)(v)=v^{\circ}(1\otimes X_{v})$ and $(j\otimes R)(u\oplus v)=(u\oplus v)^{\circ}(1\otimes X_{u\oplus v})$, hence:
	\begin{equation*}
	\begin{split}
		u^{\circ}(1\otimes X_{u})\oplus v^{\circ}(1\otimes X_{v})&=(j\otimes R)(u)\oplus (j\otimes R)(v)=(j\otimes R)(u\oplus v)\\
		&=(u\oplus v)^{\circ}(1\otimes X_{u\oplus v})=(u^{\circ}\oplus v^{\circ})(1\otimes X_{u\oplus v})\\
		&=u^{\circ}(1\otimes X_{u\oplus v})\oplus v^{\circ}(1\otimes X_{u\oplus v}).
	\end{split}
	\end{equation*}
	
	Multiplying both sides of this equation by $(1\otimes X_{u}^*)$ on the right, we obtain that $u^{\circ}\oplus v^{\circ}(1\otimes X_{v}X_{u}^*)=u^{\circ}(1\otimes X_{u\oplus v}X_{u}^*)\oplus v^{\circ}(1\otimes X_{u\oplus v}X_{u}^*)$, hence $X_{u\oplus v}X_{u}^*=1$ and $X_{v}X_{u}^*=X_{u\oplus v}X_{u}^*$; which yields $X_{u\oplus v}=X_u$ and $X_v=X_{u\oplus v}$ as claimed.
\end{proof}

	\begin{lem}[Twisted Schur's lemma]\label{lem.TwistedSchur}
 If $(u, H_{u})$ and $(v, H_{v})$ are two irreducible  $\Omega$-representations (resp.\ $\Omega^*$-representations) of $\mathbb{G}$, then either $u$ is not unitary equivalent to $v$ and $Hom_{\mathbb{G}}(u, v)=(0)$; or $u$ is unitary equivalent to $v$ and $Hom_{\mathbb{G}}(u, v)$ is a $1$-dimensional subspace of $\mathcal{B}(H_{u}, H_{v})$.
		In particular, $End_{\mathbb{G}}(u)=\mathbb{C}$, and $u$ is irreducible if and only if $u$ is indecomposable. 
	\end{lem}
\begin{proof}
Let us prove this for  $\Omega^*$-representations, the result for  $\Omega$-representations then follows by considering $(C(\mathbb{G})^{op},\Delta)$. 

Let $u$ be an $\Omega^*$-representation, and let $E \subset H_u$ be an invariant closed subspace. Let $p_E$ be the projection onto $E$.  Let $\omega$ be a faithful normal state on $\mathcal{B}(H)$. By possibly replacing $\omega$ by $(\omega \otimes h_{\mathbb{G}})\delta_u$, we may assume that $(\omega\otimes id)\delta_u(x) = \omega(x)1$. The operator $V_{\delta_u}$ on $L^2(\mathcal{B}(H),\omega)\otimes L^2(\mathbb{G})$ sending $\Lambda_{\omega}(x)\otimes \Lambda(a)$ to $(\Lambda_{\omega}\otimes \Lambda)(\delta_u(x)(1\otimes a))$ is a  corepresentation implementing $\delta_u$ by $\delta_u(x) = V_{\delta_u}(x\otimes 1)V_{\delta_u}^*$. Then the invariance of $E$ gives that $(p_E\otimes 1)\delta_u(p_E) = \delta_u(p_E)$, which implies $(p_E\otimes 1)V_{\delta_u}(p_E\otimes 1) = V_{\delta_u}(p_E\otimes 1)$. But it is well-known that this implies $p_E \otimes 1$ commutes with $V_{\delta_u}$. In particular, $\delta_u(p_E) = p_E\otimes 1$, so $p_E\otimes 1$ commutes with $u$. 

It is now immediate that $u$ is indecomposable if and only if it is irreducible, and that for $u,v$ irreducible one has that $Hom_{\mathbb{G}}(u, v)$ is either $0$ or one-dimensional, the latter case occurring if $u$ and $v$ are unitarily equivalent.
\end{proof}

	\begin{lem}\label{lem.AverageInterwinersMeas}
		Let $\mathbb{G}$ be a compact quantum group and $\Omega$ a $2$-cocycle on $\mathbb{G}$. Let $(u, H_{u})$ and $(v, H_{v})$ be two measurable  $\Omega$-representations of $\mathbb{G}$.
		\begin{enumerate}[i)]
			\item If $T: H_{u}\longrightarrow H_{v} $ is a linear bounded operator, then the bounded  linear operator $S:=(id\otimes h_{\mathbb{G}})\big(v^*(T\otimes 1)u\big)$ lies in $Hom_{\mathbb{G}}(u,v)$. It is called \emph{average intertwiner with respect to $T$}.
\item\label{It.DecompIrr} Every  $\Omega$-representation $u$ of $\mathbb{G}$ decomposes into a direct sum of irreducible  $\Omega$-representations.
		\end{enumerate}
	\end{lem}
	\begin{proof}
		\begin{enumerate}[i)]
			\item Assume that $T\in\mathcal{B}(H_{u}, H_{v})$. Clearly the linear operator $S=(id\otimes h_{\mathbb{G}})\big(v^*(T\otimes 1)u\big)$ is a well-defined bounded operator.						By definition of  $\Omega$-representation, we have 
				$$(id\otimes {}_{\Omega}\Delta)(u)=u_{12}u_{13}\mbox{ and }(id\otimes{}_{\Omega}\Delta)(v)=v_{12}v_{13},$$
				which, using the definition of ${}_{\Omega}\Delta$, can be written as:
				$$(id\otimes \Delta)(u)=(1\otimes \Omega^*)u_{12}u_{13}\mbox{ and }(id\otimes \Delta)(v)=(1\otimes\Omega^*)v_{12}v_{13}.$$
			
			Apply the unital $*$-homomorphism $(id\otimes \Delta)$ to $v^*(T\otimes 1)u$:
			\begin{equation*}
			\begin{split}
				(id\otimes \Delta)\big(v^*(T\otimes 1)u\big)&=(id\otimes \Delta)(v^*)(T\otimes 1)(id\otimes \Delta)(u)\\
				&=v_{13}^*v_{12}^*(1\otimes \Omega)(T\otimes 1\otimes 1)(1\otimes \Omega^*)u_{12}u_{13}\\
				&=v_{13}^*v_{12}^*(T\otimes 1\otimes 1)u_{12}u_{13}.
			\end{split}
			\end{equation*}
			
			Next, the $\mathbb{G}$-invariance of the Haar state of $\mathbb{G}$ yields that $(id\otimes h_{\mathbb{G}} \otimes id)\big((id\otimes \Delta)\big(v^*(T\otimes 1)u\big)\big)=S\otimes 1.$ Also, we have:
			\begin{equation*}
			\begin{split}
				(id\otimes  h_{\mathbb{G}} \otimes id)&\big(v_{13}^*v_{12}^*(T\otimes 1\otimes 1)u_{12}u_{13}\big)\\
				&=v_{12}^*\Big((id\otimes h_{\mathbb{G}})\big(v^*(T\otimes 1)u\big)\Big)_{1}u_{12}\\
				&=v^*(S\otimes 1)u
			\end{split}
			\end{equation*}
			and the conclusion follows.

\item If $u$ is an $\Omega$-representation, then the averaging operator $\mathcal{B}(H) \rightarrow End_{\mathbb{G}}(u) = \mathcal{B}(H)^{\delta_u}$ sending $T$ to $(id\otimes h_{\mathbb{G}})\big(u^*(T\otimes 1)u\big)$ is a normal conditional expectation onto $End_{\mathbb{G}}(u)$, as already observed, hence $End_{\mathbb{G}}(u)$ is a direct sum of type $I$-factors, proving that $u$ is a direct sum of irreducible  $\Omega$-representations.
		\end{enumerate}
	\end{proof}

%KDC: added \overline{\bar}
Again, a similar statement holds for  $\Omega^*$-representations. 
	
	\begin{rem}\label{rem.ergirr}
Let $\delta:\mathcal{B}(H)\longrightarrow \mathcal{B}(H)\overline{\otimes} L^\infty(\mathbb{G})$ be a measurable projective representation on $H$ with implementing unitary $u$.  By the argument used in point $ii)$ of Lemma \ref{lem.AverageInterwinersMeas}, one has a one-to-one correspondence between the set of all $\delta$-invariant projections in $\mathcal{B}(H)$ and the $u$-invariant subspaces of $H$. Accordingly, $\delta$ is ergodic if and only if $u$ is irreducible. In particular, $\delta$ is a torsion action of projective type (recall Theorem-Definition \ref{theo.TorsionFreeMeyerNest}) if and only if $u$ is finite dimensional and irreducible.
\end{rem}

To any $2$-cocycle $\Omega$ one can associate a canonical $\Omega$-representation. 

\begin{theodefi}[Projective regular representation]\label{theo.ProjectiveRightRepMeas}
		Let $\mathbb{G}$ be a compact quantum group and $\Omega$ a $2$-cocycle on $\mathbb{G}$. Defining $V^{\Omega} =\Omega V_{\mathbb{G}}$, the following properties hold:
			\begin{enumerate}[i)]
				\item For all $x\in L^{\infty}(\mathbb{G})$ and $\xi\in L^2(\mathbb{G})$ we have $V^{\Omega}(\Lambda(x)\otimes \xi)={}_{\Omega}\Delta(x)(\xi_{\mathbb{G}}\otimes \xi)$.
				\item For all $x\in L^{\infty}(\mathbb{G})$ we have ${}_{\Omega}\Delta(x)=V^\Omega(x\otimes 1)V_{\mathbb{G}}^*$.
				\item The following identity holds: $(id\otimes {}_{\Omega}\Delta)(V^\Omega)=V^\Omega_{12}V^\Omega_{13}$, so $V^{\Omega} \in \mathcal{B}(L^2(\mathbb{G}))\overline{\otimes} L^{\infty}(\mathbb{G})$ is an $\Omega$-representation.
				\item The following pentagonal equation holds: $V^\Omega_{12}V^\Omega_{13}(V_{\mathbb{G}})_{23}=V^{\Omega}_{23}V^\Omega_{12}$.
			\end{enumerate}
			The unitary $V^\Omega$ is called \emph{right projective regular representation of $\mathbb{G}$ on $L^2(\mathbb{G})$ with respect to $\Omega$} or simply \emph{right $\Omega$-regular representation of $\mathbb{G}$ on $L^2(\mathbb{G})$}. 
	\end{theodefi}
	\begin{rem}\label{rem.ProjectiveLeftRep}
		Similarly, defining $W^{\Omega} = W_{\mathbb{G}}\Omega^*$, we have that  $(W^{\Omega})^*(\xi\otimes \Lambda(x))={}_{\Omega}\Delta(x)(\xi\otimes \xi_{\mathbb{G}})\mbox{,}$
		for all $x\in L^{\infty}(\mathbb{G})$ and $\xi\in L^2(\mathbb{G})$. For all $x\in L^{\infty}(\mathbb{G})$ we have ${}_{\Omega}\Delta(x)=(W^\Omega)^*(1\otimes x)W_{\mathbb{G}}$ and the pentagonal equation: $(W_{\mathbb{G}})_{12}W^\Omega_{13}W^\Omega_{23}=W^{\Omega}_{23}W^\Omega_{12}$. Moreover, the following identity holds: $(\Delta_{\Omega^*}\otimes id)(W^\Omega)=W^\Omega_{13}W^\Omega_{23}$, so $\Sigma W^{\Omega}\Sigma$ is an $\Omega^*$-projective representation.
					
		The unitary $W^\Omega$ is called \emph{left projective regular representation of $\mathbb{G}$ on $L^2(\mathbb{G})$ with respect to $\Omega$} or simply \emph{left $\Omega^*$-regular representation of $\mathbb{G}$ on $L^2(\mathbb{G})$}. 
	\end{rem}
	
 The following lemma follows from direct computations by using the relations from Theorem \ref{theo.KacSystemG}.
	\begin{lem}\label{lem.IdentityTwistDualAction}
		Let $\Omega$ be a $2$-cocycle for $\mathbb{G}$. Given the canonical Kac system, $(V_{\mathbb{G}}, U_{\mathbb{G}})$, associated to $\mathbb{G}$, the following identities hold:
		\begin{enumerate}[i)]
			\item $(\widetilde{V}_{\mathbb{G}})_{12}(V^\Omega)_{13}(\widetilde{V}_{\mathbb{G}})_{12}^*=(V^\Omega)_{13}(V_{\mathbb{G}})_{23}$.
			\item $(V^{\Omega})_{12}(\widetilde{V}_{\mathbb{G}})_{23}=(\widetilde{V}_{\mathbb{G}})_{23}(V^{\Omega})_{12}$.
		\end{enumerate}
	\end{lem}

	\begin{lem}\label{lem.IdentificationLInftyOmega}
		Let $\mathbb{G}$ be a compact quantum group and $\Omega$ a $2$-cocycle on $\mathbb{G}$. We have: 
		\begin{equation*}
		\begin{split}
			L^{\infty}(\mathbb{G})&\cong\overline{span}^{\sigma-weakly}\{(\eta\otimes id)(V^\Omega)\ |\ \eta\in \mathcal{B}(L^2(\mathbb{G}))_*\}\\
			&=\overline{span}^{\sigma-weakly}\{(id\otimes \eta)((W^\Omega)^*)\ |\ \eta\in \mathcal{B}(L^2(\mathbb{G}))_*\}.
		\end{split}
		\end{equation*}
	\end{lem}
	\begin{proof}
		Let us show the first identification. The second one follows analogously. Given $x, y\in L^{\infty}(\mathbb{G})$ consider the coordinate linear functional $\omega_{x\xi_{\mathbb{G}}, y\xi_{\mathbb{G}}}\in \mathcal{B}(L^2(\mathbb{G}))_*$ and write the following:
		\begin{equation*}
		\begin{split}
			\langle (\omega_{x\xi_{\mathbb{G}}, y\xi_{\mathbb{G}}}\otimes id)(V^\Omega)(\xi), \xi'\rangle&=\langle V^\Omega(x\xi_{\mathbb{G}}\otimes \xi), y\xi_{\mathbb{G}}\otimes \xi'\rangle\\
			&=\langle {}_{\Omega}\Delta(x)(\xi_{\mathbb{G}}\otimes \xi),  y\xi_{\mathbb{G}}\otimes \xi'\rangle\\
			&= \langle (y^*\otimes 1){}_{\Omega}\Delta(x)(\xi_{\mathbb{G}}\otimes \xi), \xi_{\mathbb{G}}\otimes \xi'\rangle\\
			&=\langle (h_{\mathbb{G}}\otimes id)\big((y^*\otimes 1){}_{\Omega}\Delta(x)\big)\xi,\xi' \rangle\mbox{,}
		\end{split}
		\end{equation*}
		 for all $\xi, \xi'\in L^2(\mathbb{G})$. Hence, $(\omega_{x\xi_{\mathbb{G}}, y\xi_{\mathbb{G}}}\otimes id)(V^\Omega)=(h_{\mathbb{G}}\otimes id)\big((y^*\otimes 1){}_{\Omega}\Delta(x)\big)$, for all $x,y\in L^\infty(\mathbb{G})$. It is enough to show that the linear span of these elements is $\sigma$-weakly dense in $L^\infty(\mathbb{G})$. 

As $h_{\mathbb{G}}$ is a faithful normal state on $L^{\infty}(\mathbb{G})$, it is in fact sufficient to show that the linear span of elements of the form $(h_{\mathbb{G}}\otimes id)\big({}_{\Omega}\Delta(x)(y^*\otimes 1)\big)$ is $\sigma$-weakly dense in $L^\infty(\mathbb{G})$. But by the cancellation property of $\Delta$ we find immediately that ${}_{\Omega}\Delta(L^\infty(\mathbb{G}))(L^\infty(\mathbb{G})\otimes 1)$ is $\sigma$-weakly dense in $L^\infty(\mathbb{G})\otimes L^\infty(\mathbb{G})$, which yields the conclusion.
	\end{proof}
	
	A standard argument by combining the previous lemma and Lemma \ref{lem.AverageInterwinersMeas} yields the following Peter-Weyl theorem.
	\begin{theo}[Twisted Peter-Weyl theorem I]\label{theo.PeterWeyl1Meas}
		Let $\mathbb{G}$ be a compact quantum group and $\Omega$ a $2$-cocycle. The right projective regular representation $(V^\Omega, L^2(\mathbb{G}))$ contains all irreducible $\Omega$-representations of $\mathbb{G}$ in its direct sum decomposition.
	\end{theo}
	
	Following \cite{KennyProjective} we have a twisted version of the Schur's orthogonality relations. This theorem follows straightforwardly by applying Lemma \ref{lem.AverageInterwinersMeas}.i) with respect to rank one operators.

	\begin{theo}[Twisted Schur's orthogonality relations]\label{theo.TwistedOrthogonalityRel}
		Let $\mathbb{G}$ be a compact quantum group and $\Omega$ a $2$-cocycle on $\mathbb{G}$. Let $\{u^x\}_{x\in Irr(\mathbb{G}, \Omega)}$ be a complete set of mutually inequivalent, irreducible $\Omega$-representations, with fixed bases for the associated Hilbert spaces $H_x$. For each $x\in Irr(\mathbb{G}, \Omega)$ there exists a positive trace class operator $F^x\in\mathcal{B}(H_x)$ with zero kernel such that the following orthogonality relations hold:
		$$h_{\mathbb{G}}\big((u^y_{kl})^*u^x_{ij}\big)=\delta_{x y}\delta_{l j}F^x_{ik},$$
		for every $x,y\in Irr(\mathbb{G}, \Omega)$, $i,j=1,\ldots, n_x$ and $k,l=1,\ldots, n_y$.
	\end{theo}
	
The matrix $F^x$ is nothing but the density matrix of the $\delta_{u^x}$-invariant state $\varphi_x$ with $\varphi_x(T) =(h_{\mathbb{G}}\otimes id)\delta(T)$ for $T\in \mathcal{B}(H)$, so $\varphi_x = Tr(F^x-)$. 

	Given $x\in Irr(\mathbb{G}, \Omega)$ and the corresponding positive operator $F^x\in\mathcal{B}(H_x)$ from the previous theorem, we fix an orthonormal basis of $H_x$, $\{\xi^x_i\}_{i=1,\ldots, n_x}$, that diagonalises $F^x$. If $F^x_j\in\mathbb{R}^+$ denotes the eigenvalue of $F^x$ for the eigenvector $\xi^x_j$, for every $j=1,\ldots, n_x$, then the orthogonality relations become $h_{\mathbb{G}}\big((u^y_{kl})^*u^x_{ij}\big)=\delta_{x y}\delta_{ki}\delta_{l j}F^x_{k}$. Following these notations, we obtain as an immediate corollary of the previous two theorems the following decomposition for $L^2(\mathbb{G})$.
	
	\begin{theo}[Twisted Peter-Weyl theorem II]\label{theo.TwistedPeterWeyL2}
		Let $\mathbb{G}$ be a compact quantum group and $\Omega$ a $2$-cocycle on $\mathbb{G}$. We have a unitary transformation $L^2(\mathbb{G})\cong \underset{x\in Irr(\mathbb{G}, \Omega)}{\bigoplus} H_x\otimes \overline{H_{x}}$ such that $\Lambda(u^x_{ij})\mapsto \sqrt{F^x_i}\ \xi^x_{i}\otimes \overline{\xi^{x}_{j}}$, for all $j=1,\ldots, n_x$, $x\in \text{Irr}(\mathbb{G}, \Omega)$.
	\end{theo}

\subsection{Continuous  $\Omega$-representations and $2$-cocycles of finite type}\label{sec.contprojrep}

We now consider a special type of measurable $2$-cocycles. 

\begin{defi}\label{def.fintype}
We say that a  $2$-cocycle $\Omega$ on $\mathbb{G}$ is of \emph{finite type} if there exists a finite dimensional  $\Omega$-representation. 
\end{defi}

Not all  $2$-cocycles $\Omega$ are of finite type, see e.g.\ \cite[Section 8]{KennyProjective} for an example of a $2$-cocycle which is not of finite type. These types of $2$-cocycles will however be sufficient for our needs. The following lemma shows that being of finite type is an ambidextrous notion. 

\begin{lem}
A  $2$-cocycle $\Omega$ is finite type if and only if there exists a finite dimensional  $\Omega^*$-representation.
\end{lem} 
\begin{proof}
This follows immediately from Corollary \ref{cor.opp2coc}.
\end{proof}

Recall the notation introduced in \eqref{eq.deltau2} and \eqref{eq.deltau}.

\begin{defi}
Let $\Omega \in L^{\infty}(\mathbb{G})\overline{\otimes}L^{\infty}(\mathbb{G})$ be a measurable  $2$-cocycle.
 We say that a measurable  $\Omega$-representation (resp.\ $\Omega^*$-representation) $u$ is \emph{continuous} if $\delta_u$ is a continuous right (resp.\ left) projective representation. 
\end{defi} 

Note that this notion is strictly weaker than demanding that $u \in M(\mathcal{K}(H)\otimes C(\mathbb{G}))$, which is a too strong condition in practice.  Note also that any  $\Omega$-representation on a \emph{finite dimensional} Hilbert space is automatically continuous. 

We will show that continuous  $\Omega$-representations can only exist if $\Omega$ is of finite type, and that then all  $\Omega$-representations are continuous. 

	\begin{theo}[Twisted Maschke's theorem]\label{theo.AverageInterwinersCont}
		Let $\mathbb{G}$ be a compact quantum group and $\Omega$ a  $2$-cocycle on $\mathbb{G}$. Let $(u, H_{u})$ and $(v, H_{v})$ be two continuous  $\Omega$-representations of $\mathbb{G}$.
		\begin{enumerate}[i)]
\item If $T: H_{u}\longrightarrow H_{v} $ is a linear compact operator, then the average intertwiner $S$ with respect to $T$ is again compact. 
			\item The C$^*$-algebra $\mathcal{D}_u = \mathcal{K}(H_{u})^{\delta_u}$ acts non-degenerately on $H_{u}$, that is, $[\mathcal{D}_{u}H_{u}]=H_{u}$.
\item\label{It.DecompIrr} If $u$ is irreducible, then $u$ is finite dimensional. 
\item Every continuous  $\Omega$-representation of $\mathbb{G}$ decomposes into a direct sum of finite dimensional irreducible  $\Omega$-representations. 
		\end{enumerate}
	\end{theo}
	\begin{proof}
%Changed h \otimes id to id\otimes h in first bullet point
		\begin{enumerate}[i)]
\item Assume that $u$ is a continuous  $\Omega$-representation. If then $T \in \mathcal{K}(H_u)$, we have that the average $S=(id\otimes h_{\mathbb{G}})\big(u^*(T\otimes 1)u\big) = (h_{\mathbb{G}}\otimes id)\delta_u(T) \in \mathcal{K}(H)$. This proves the first point when $u=v$. 

\item This is just a special case of the observation in Remark \ref{rem.nondegact}.

\item Since $\mathcal{K}(H)^{\delta}$ is necessarily non-trivial when $H$ is infinite dimensional, it follows that $u$ can only be irreducible when $H$ is finite dimensional. A general continuous  $\Omega$-representation $u$ must then be a direct sum of finite dimensional (and hence continuous)  $\Omega$-representations. It then also follows immediately that if $u,v$ are continuous  $\Omega$-representations, an average intertwiner with respect to an operator in $\mathcal{K}(H_u,H_v)$ remains in this space, proving the first point in general.

\item The last point follows from iii) and Lemma \ref{lem.AverageInterwinersMeas}.\ref{It.DecompIrr}).
		\end{enumerate}
	\end{proof}

Note that the previous lemma shows in particular that the class of continuous  $\Omega$-representations is stable under direct sums, which is not immediately obvious. By Remark \ref{rem.ergirr}, we moreover see that $\widehat{\mathbb{G}}$ is projective torsion-free if and only if all $2$-cocycles of finite type on $\mathbb{G}$ are cohomologous to the trivial one.

\begin{theo}\label{theo.OmegaFiniteTypeCharact}
Let $\mathbb{G}$ be a compact quantum group, and let $\Omega$ be a  $2$-cocycle. Then the following are equivalent:
\begin{enumerate}[i)]
\item\label{theo.cocequi3} There exists a continuous  $\Omega$-representation.
\item\label{theo.cocequi1} $\Omega$ is of finite type.
\item\label{theo.cocequi4} All irreducible  $\Omega$-representations are finite dimensional. 
\end{enumerate}
Moreover, if one (hence any) of these conditions hold, then all  $\Omega$-representations of $\mathbb{G}$ are continuous. 
\end{theo}	
\begin{proof}
The implication \ref{theo.cocequi3}) $\Rightarrow $ \ref{theo.cocequi1}) follows from Theorem \ref{theo.AverageInterwinersCont}.iv). The implication \ref{theo.cocequi1}) $\Rightarrow$ \ref{theo.cocequi4}) follows from \cite[Proposition 4.3]{KennyProjective}. The implication \ref{theo.cocequi4}) $\Rightarrow$ \ref{theo.cocequi3}) is trivial, and the final statement follow from Lemma \ref{lem.AverageInterwinersMeas}.\ref{It.DecompIrr}).
\end{proof} 

\begin{cor}
Assume that $\Omega \in C(\mathbb{G})\otimes C(\mathbb{G})$ is a continuous  $2$-cocycle. Then $\Omega$ is of finite type.
\end{cor}
\begin{proof}
In this case we have that the right regular projective representation $V^{\Omega} \in M(\mathcal{K}(L^2(\mathbb{G})) \otimes C(\mathbb{G}))$, so a fortiori the associated projective representation is continuous. By the previous theorem, this forces $\Omega$ to be of finite type.
\end{proof} 

%Added remark on continuous 2-cocycles

Assume now again that  $\Omega$ is a general measurable  $2$-cocycle. Then clearly 
\[
{}_{\Omega}\Delta_{\Omega^*}: L^{\infty}(\mathbb{G}) \rightarrow L^{\infty}(\mathbb{G})\overline{\otimes}L^{\infty}(\mathbb{G}),\qquad x \mapsto \Omega\Delta(x)\Omega^*
\]
defines a coassociative coproduct on $L^{\infty}(\mathbb{G})$. It follows from \cite{KennyGaloisObjTwistings} that $( L^{\infty}(\mathbb{G}), {}_{\Omega}\Delta_{\Omega^*})$ is again a locally compact quantum group. It need not necessarily be compact, as the example in  \cite{KennyCocycles} shows. However, one has the following theorem as a particular case of Proposition $4.3.2$ in \cite{KennyProjective}. 
\begin{theo}\label{theo.TwistCompact}
Let $\mathbb{G}$ be a compact quantum group and let $\Omega$ be a  $2$-cocycle. Then the couple $(L^{\infty}(\mathbb{G}),{}_{\Omega}\Delta_{\Omega^*})$ defines a \emph{compact} quantum group  if and only if $\Omega$ is of finite type. 
\end{theo}

We will use the  notation $(L^{\infty}(\mathbb{G}_{\Omega}),\Delta_{\mathbb{G}_{\Omega}})= (L^{\infty}(\mathbb{G}),{}_{\Omega}\Delta_{\Omega^*})$. We denote $C(\mathbb{G}_{\Omega}) \subseteq L^{\infty}(\mathbb{G}_{\Omega})$ for the associated reduced C$^*$-algebra, and $\text{Pol}(\mathbb{G}_{\Omega})$ for the polynomial Hopf $*$-subalgebra.  

By standard von Neumann algebra theory \cite[Theorem IX.1.14]{Takesaki2}, there is a canonical identification $L^2(\mathbb{G}_{\Omega}) \cong L^2(\mathbb{G})$ intertwining in particular the modular conjugations $J_{\mathbb{G}_{\Omega}}$ and $J_{\mathbb{G}}$ and the left action of $L^{\infty}(\mathbb{G})$. In the following, we will then simply identify $L^2(\mathbb{G}_{\Omega}) \cong L^2(\mathbb{G})$.

\begin{rem}\label{rem.PKactype}
An interesting situation where $\mathbb{G}_{\Omega}$ is always compact is when $\mathbb{G}$ is of Kac type (see \cite[Proposition 5.1]{KennyProjective}). More precisely, by \cite[Proposition 4.3]{KennyProjective} it follows that any $2$-cocycle of a compact quantum group of Kac type is of finite type.
	\end{rem}

In \cite[Proposition 4.3]{KennyProjective}, the language of \emph{Galois co-objects} is used (in the measurable setting), of which $(L^{\infty}(\mathbb{G}),{}_{\Omega}\Delta)$ is a particular case, see Example 1.20 in \cite{KennyProjective}. Although we are able to avoid this more abstract theory in the measurable setting, it is necessary to use this formalism in the C$^*$-algebraic and algebraic setting, due to the fact that most $2$-cocycles are in practice not cohomologous to continuous or algebraic ones, even when of finite type, and that in general one can expect $C(\mathbb{G}_{\Omega}) \neq C(\mathbb{G})$ inside $L^{\infty}(\mathbb{G})$. 

Let us provide now some more information on the relation between $\text{Pol}(\mathbb{G})$ and $\text{Pol}(\mathbb{G}_{\Omega})$.

	\begin{defi} Let $\Omega$ be a  $2$-cocycle of finite type.		We denote by $\text{Pol}(\mathbb{G}, \Omega)\subset L^{\infty}(\mathbb{G})$ the linear span of matrix coefficients of all irreducible  $\Omega$-representations of $\mathbb{G}$, and by $C(\mathbb{G},\Omega)$ its normclosure.
	\end{defi}
By Lemma \ref{lem.IdentificationLInftyOmega} and Theorem \ref{theo.AverageInterwinersCont}.iv), it follows that $\text{Pol}(\mathbb{G}, \Omega)$ is $\sigma$-weakly dense in $L^\infty(\mathbb{G})$.

Contrary to the ordinary case when $\Omega=1\otimes 1$, $\text{Pol}(\mathbb{G}, \Omega)$ is not a Hopf $*$-algebra. More precisely, it is a coalgebra but not an algebra. However, $\text{Pol}(\mathbb{G},\Omega)$ will be a $\text{Pol}(\mathbb{G}_{\Omega})$-$\text{Pol}(\mathbb{G})$-bimodule (and in fact bimodule coalgebra) such that: $$\text{Pol}(\mathbb{G},\Omega)^* \text{Pol}(\mathbb{G},\Omega) = \text{Pol}(\mathbb{G}),\qquad \text{Pol}(\mathbb{G},\Omega)\text{Pol}(\mathbb{G},\Omega)^* = \text{Pol}(\mathbb{G}_{\Omega}).$$  Indeed, if $u$ is a finite dimensional  $\Omega$-representation and $v$ a  finite dimensional  $\mathbb{G}$-representation, then $u_{13}v_{23}$ is a  finite dimensional  $\Omega$-representation, showing that $\text{Pol}(\mathbb{G},\Omega)$ is a right $\text{Pol}(\mathbb{G})$-module. We obtain then for example the equality $\text{Pol}(\mathbb{G},\Omega)^* \text{Pol}(\mathbb{G},\Omega) = \text{Pol}(\mathbb{G})$ as clearly the left hand side is a $\sigma$-weakly dense $*$-subbialgebra of $L^{\infty}(\mathbb{G})$. The analogous properties for $C(\mathbb{G},\Omega)$ then follow immediately.

From Theorem \ref{theo.TwistedOrthogonalityRel}, we also deduce the following. 
	\begin{cor}
		Let $\mathbb{G}$ be a compact quantum group and $\Omega$ a  $2$-cocycle of finite type on $\mathbb{G}$. The set of matrix coefficients of all representatives of irreducible  $\Omega$-representations of $\mathbb{G}$ (with respect to fixed bases) form a basis for $\text{Pol}(\mathbb{G}, \Omega)$.
	\end{cor}
	
	\subsection{Example: $q$-deformations of compact connected Lie groups}
		In order to illustrate the previous theory, let us study the projective representations of $q$-deformations of connected semisimple compact Lie groups. If $G$ is such a group which is in addition simply connected, and $q>0$, we denote by $G_q$ the Drinfeld-Jimbo $q$-deformation of $G$, which is a compact quantum group \cite[Section 2.4]{Sergey}. We put $G_1:=G$.
		
%KDC: the following paragraph needed rewording - it is the cocycles which are hard to classify, not the associated reps. 		
First of all, it is important to say that the measurable $2$-cohomology of $G_q$, i.e.\ the set of all $2$-cocycles of $G_q$ up to coboundaries, can be hard to determine. Indeed, this is already the case for $G_q:=SU_q(2)$, for which the measurable $2$-cohomology contains at least $3$ elements \cite{KennyRegularSemiRegular}, but where to the best knowledge of the authors, it is open to decide whether this is sharp. 
		
		Therefore, it is more reasonable to consider projective representations of $G_q$ with respect to a \emph{finite type} $2$-cocycle. Then we restrict our attention to irreducible ones by Theorem \ref{theo.AverageInterwinersCont}, which are finite dimensional. In this case, we recall that $\widehat{SU_q(2)}$ is torsion-free \cite{VoigtBaumConnesFree}. In particular, $\widehat{SU_q(2)}$ is projective torsion-free, which in our language means that all $2$-cocycles of finite type on $SU_q(2)$ are cohomologous to the trivial one. In other words, all finite dimensional irreducible  projective representations of $SU_q(2)$ are trivial. This example can be generalised as follows.
		\begin{theo}\label{theo.TorsionFreeQDefor}
			Let $G$ be a connected simply connected semisimple compact Lie group. Given a parameter $q>0$, $\widehat{G_q}$ is torsion-free.
		\end{theo}
		\begin{proof}
			This is shown in \cite[Theorem 5.3]{GoffengTorsionQDeformations}. 
		\end{proof}
		
		Related to $SU_q(2)$ there is the $q$-deformation of the rotation group, $SO_q(3)$. As indicated in Examples \ref{ex.ActionsCQG}, $\widehat{SO_q(3)}$ is not projective torsion-free and its projective representation can be viewed as a $q$-deformation of the projective representation of $SO(3)$ associated to its universal cover $SU(2)$. Moreover, this projective representation is unique up to equivariant Morita equivalence (Example \ref{ex.ProjectiveTorsionSO(3)}). As in the classical situation, one has $SO_q(3)=SU_q(2)/\mathbb{Z}_2$ \cite{PodlesQuantumSpaces}. This example can be generalised as follows. 
		
%KDC: changed discrete into finite + added multiplier in containment below
		Let $G$ be a connected semisimple compact Lie group and $q>0$. We denote by $\widetilde{G}$ the universal cover of $G$. Recall that there exists a central (finite) subgroup $\Gamma<\widetilde{G}$ such that $G\cong \widetilde{G}/\Gamma$. Observe that $Z(\widetilde{G})\subset Z(M(c_0(\widehat{\widetilde{G}}_q)))$. Hence given any central subgroup $\Gamma< Z(\widetilde{G})$, it is legitimate to consider the quotient quantum group $\widetilde{G}_q/\Gamma$. In this way we can define $q$-deformations of connected semisimple compact Lie groups which are not necessarily simply connected by putting $G_q:=(\widetilde{G}/\Gamma)_q:=\widetilde{G}_q/\Gamma$.
		
		Coming back to the classical case $SO(3)\cong SU(2)/\mathbb{Z}_2$, recall that (finite dimensional) irreducible representations of $SU(2)$ are classified by positive integers. The following is a well-known fact from representation theory of compact Lie groups. If $x(n)\in\text{Irr}(SU(2))$ with $n$ an odd positive integer, then $x(n)$ descends to an ordinary (finite dimensional) representation of $SO(3)$. If $x(n)\in\text{Irr}(SU(2))$ with $n$ an even positive integer, then $x(n)$ descends to a \emph{projective} (finite dimensional) representation of $SO(3)$. This establish a one-to-one correspondence between irreducible projective representations of $SO(3)$ and ordinary irreducible representations of $SU(2)$. The same correspondence holds for connected semisimple compact Lie groups. The following theorem generalises these results for $q$-deformations.
		
		\begin{theo}
			Let $G$ be a connected non-simply connected semisimple compact Lie group. Given a paramenter $q>0$, $\widehat{G}_q$ is not projective torsion-free and there is a one-to-one correspondence between finite dimensional irreducible representations of $\widetilde{G}_q$ and finite dimensional irreducible projective representations of $G_q$.
		\end{theo}
		\begin{proof}
		
		First of all, since $\Gamma$ is a central subgroup of $\widetilde{G}_q$, it is a normal quantum subgroup of $\widetilde{G}_q$. Therefore, the left coset space $\widetilde{G}_q/\Gamma$ defines a compact quantum group. Moreover, one has that the dual of $\widetilde{G}_q/{\Gamma}$ is a discrete quantum subgroup of the dual of $\widetilde{G}_q$ (cf. \cite[Proposition 2.1]{WangSimple}), i.e. $C(\widetilde{G}_q/{\Gamma})$ embeds into $C(\widetilde{G}_q)$ intertwining the co-multiplications.
		
		Assume that $\delta: \mathcal{B}(H)\rightarrow \mathcal{B}(H)\otimes C(G_q)$ is a finite dimensional irreducible projective representation of $G_q=\widetilde{G}_q/\Gamma$. By composing with the inclusion $C(\widetilde{G}_q/{\Gamma})\subset C(\widetilde{G}_q)$, we obtain a finite dimensional irreducible projective representation of $\widetilde{G}_q$, say $\widetilde{\delta}: \mathcal{B}(H)\rightarrow \mathcal{B}(H)\otimes C(\widetilde{G}_q)$. However, by virtue of Theorem \ref{theo.TorsionFreeQDefor}, the dual of $\widetilde{G}_q$ is in particular projective torsion-free. Hence there exists a finite dimensional irreducible representation $u$ of $\widetilde{G}_q$ such that $\widetilde{\delta}=\text{Ad}(u)$.
		
%KDC: some corrections to the argument below 
		Conversely, assume that $u\in\mathcal{B}(H)\otimes C(\widetilde{G}_q)$ is a finite dimensional irreducible representation of $\widetilde{G}_q$. If we denote by $\rho: C(\widetilde{G}_q)\twoheadrightarrow C(\Gamma)$ the surjective $*$-homomorphism realising $\Gamma$ as a subgroup of $\widetilde{G}_q$, then it is clear that $(id\otimes \rho)(u)$ is a finite dimensional representation of $\Gamma$. But since the $u_g$ for $g\in \Gamma$ then commute with all $(id\otimes \omega)(u)$ for $\omega \in C(\widetilde{G}_q)^*$, by centrality of $\Gamma$, it follows by irreducibility of $u$ that $(id\otimes \rho)(u) = 1 \otimes c$ for some group-like unitary in $C(\Gamma)$. In particular, $(id\otimes \rho)(u(x\otimes 1)u^*) = x \otimes 1$ for all $x\in \mathcal{B}(H)$, showing that $\delta := \text{Ad}(u)$ is a well-defined $*$-homomorphism defining a finite dimensional irreducible projective representation of $G_q$.	
		\end{proof}

\section{\textsc{Twisted Baaj-Skandalis duality}}\label{sec.TwistedBaajSkand}
	We study the construction of a twisted crossed product by a compact quantum group based on the notion of twisted dynamical system. Twisted C$^*$-algebras associated to classical locally compact groups and, more generally, twisted crossed products with respect to a $2$-cocycle have been studied in the literature by several hands (standard references are 	\cite{PackerTwistSurvey,BedosDiscDynSyst}). The more general framework of locally compact quantum groups has been addressed for instance in \cite{MakotoDeformation,SergeyDeformation2}, see also the work of S. Vaes and L. Vainerman in the von Neumann algebraic setting \cite{VaesBiproduit}. 

In this paper, we focus on the case of compact quantum groups. First, we relate the regularity of a $2$-cocycle as defined in \cite{SergeyDeformation2} to our notion of being of finite type. Then we study twisted crossed products coming from projective representations. We end by considering a twisted version of the Takesaki-Takai duality and the Baaj-Skandalis duality. 

The contents of this section have been initially inspired by \cite{QuiggTwistDuality}.
	
	\subsection{Twisted group C$^*$-algebras and regularity}
	
	\begin{defi}\label{defi.TwistedCStarG}
		Let $\mathbb{G}$ be a compact quantum group and $\Omega$ a measurable $2$-cocycle on $\mathbb{G}$. The twisted reduced C$^*$-algebra of $\mathbb{G}$ with respect to $\Omega$ is the C$^*$-algebra defined by:
		$$C^*_r(\mathbb{G}, \Omega)\equiv c_0(\widehat{\mathbb{G}}, \Omega):=C^*\langle(id\otimes \eta)(V^\Omega)\ |\ \eta\in \mathcal{B}(L^2(\mathbb{G}))_*)\rangle\subset \mathcal{B}(L^2(\mathbb{G})).$$
	\end{defi}

The following result can be found in \cite[Lemma 4.9]{BichonVaesRijdt} when $\mathbb{G}$ is discrete, but is valid for general regular locally compact quantum groups as was already remarked in \cite{SergeyDeformation2}. As we are using a slightly different setup (on which we will comment after the lemma) we include a proof for $\mathbb{G}$ compact, following a different path.

	\begin{lem}\label{lem.TwistedCAlgebra}
We have an equality $C^*_r(\mathbb{G}, \Omega) = [\{(id\otimes \eta)(V^\Omega)\ |\ \eta\in \mathcal{B}(L^2(\mathbb{G}))_*\}]$. Moreover, $C^*_r(\mathbb{G}, \Omega)$ acts non-degenerately on $L^2(\mathbb{G})$.
\end{lem}
\begin{proof}
By a direct computation, we have $(id\otimes \eta)(V^{\Omega})\Lambda(x) = \Lambda((id\otimes \eta)({}_{\Omega}\Delta(x)))$ for $x\in L^{\infty}(\mathbb{G})$ and $\eta\in L^{\infty}(\mathbb{G})_*$. By Lemma \ref{lem.IdentificationLInftyOmega} and the fact that $V^{\Omega} \in \mathcal{B}(L^2(\mathbb{G})) \overline{\otimes} L^{\infty}(\mathbb{G})$, we have that $[\{(id\otimes \eta)(V^\Omega)\ |\ \eta\in \mathcal{B}(L^2(\mathbb{G}))_*\}] = [\{(id\otimes \omega_{\Lambda(u_{rs}^y),\xi_{\mathbb{G}}})(V^{\Omega}) \mid y\in \text{Irr}(\mathbb{G},\Omega),1\leq r,s\leq n_y\}$. But a direct computation using the twisted orthogonality relations in Theorem \ref{theo.TwistedOrthogonalityRel} shows that with respect to the basis in Theorem \ref{theo.TwistedPeterWeyL2} we have  that: 
\begin{equation}\label{EqActStGen}
(F_r^x)^{-1}(id\otimes \omega_{\Lambda(u_{rs}^y),\xi_{\mathbb{G}}})(V^{\Omega})(\xi_i^x \otimes \overline{\xi_j^x})=\delta_{x,y}\delta_{s,j} \xi_i^x \otimes \overline{\xi_r^x}.
\end{equation}
 Hence $[\{(id\otimes \eta)(V^\Omega)\ |\ \eta\in \mathcal{B}(L^2(\mathbb{G}))_*\}]$ forms a C$^*$-algebra acting non-degenerately on $L^2(\mathbb{G})$, and we moreover obtain the following corollary. 
\end{proof}

	\begin{cor}[Twisted Peter-Weyl theorem III]
		Let $\mathbb{G}$ be a compact quantum group and $\Omega$ a $2$-cocycle. Then we have a C$^*$-isomorphism $C^*_r(\mathbb{G}, \Omega)\cong \underset{x\in \text{Irr}(\mathbb{G}, \Omega)}{\bigoplus^{c_0}} \mathcal{K}(\overline{H_x})$.
	\end{cor}
	
Denoting $l^{\infty}(\widehat{\mathbb{G}},\Omega) = L(\mathbb{G},\Omega) = C^*_r(\mathbb{G},\Omega)''$, we then also have that $l^{\infty}(\widehat{\mathbb{G}},\Omega) \cong \underset{x\in \text{Irr}(\mathbb{G}, \Omega)}{\bigoplus^{l^{\infty}}} \mathcal{B}(\overline{H_x})$.
% To be consistent, the l^{\infty} should be of the dual of G.

\begin{rem}\label{rem.switchOmOmStar}
It follows from Corollary \ref{cor.opp2coc} that there is a one-to-one correspondence between projective representations of $\mathbb{G}$ and $\mathbb{G}_{\Omega}$, where we assume that $\Omega$ is of finite type and hence $\mathbb{G}_{\Omega}$ compact. First of all, $\Theta:=\Omega^*$ is then easily seen to be a $2$-cocycle of finite type for $\mathbb{G}_{\Omega}$ with $(\mathbb{G}_{\Omega})_{\Omega^*} =\mathbb{G}$. If then $u$ is a $\Omega$-representation of $\mathbb{G}$, it is also a $\Theta^*$-representation of $\mathbb{G}_{\Omega}$, leading to the left projective $\mathbb{G}_{\Omega}$-representation $\delta_{u}(a) = u(a\otimes 1)u^*$ for $a\in \mathcal{B}(H_u)$. By Corollary \ref{cor.opp2coc}, we obtain a one-to-one correspondence between $\Omega$-representations of $\mathbb{G}_{\Omega}$ and $\Omega^*$-representations of $\mathbb{G}$. 
\end{rem}

	\begin{rem}\label{rem.TwistedCStarGOmega}
		We can also relate the regular representations of $\mathbb{G}$ and $\mathbb{G}_{\Omega}$. Indeed, since $L^{\infty}(\mathbb{G}) = L^{\infty}(\mathbb{G}_{\Omega})$, we can canonically  identify $L^2(\mathbb{G}_{\Omega})$ and $L^2(\mathbb{G})$, and by \cite[Proposition 5.4]{KennyGaloisObjTwistings} the right regular representation of the twisted quantum group is given by $\Omega^* V_{\mathbb{G}_{\Omega}} = (J\otimes \hat{J}X_{\Omega}^*)(\Omega V_{\mathbb{G}})^* (J\otimes \hat{J})$. To see that we can use here the same $\Omega$ as before, it is sufficient to calculate that $(\Omega^*V_{\mathbb{G}_{\Omega}})^{\circ} = \Omega V_{\mathbb{G}}$ (after identifying $L^2(\mathbb{G})$ with $\overline{L^2(\mathbb{G})}$ by $J$), that is that $V_{\mathcal{B}(L^2(\mathbb{G}))} = \Omega_{13}^*(V_{\mathbb{G}_{\Omega}})_{13}\Omega_{23} (V_{\mathbb{G}})_{23}$, where $\mathcal{B}(L^2(\mathbb{G}))$ carries the coaction $Ad_{\Omega^*V_{\mathbb{G}_{\Omega}}}$. This can be verified using the techniques of \cite{KennyGaloisObjTwistings}. Alternatively, one can also follow more directly the proof of Proposition 4.1.5 of \cite{KennyProjective}.

In any case, from the above we immediately get that $C^*_r(\mathbb{G}_\Omega, \Omega^*)=JC^*_r(\mathbb{G}, \Omega)J$. Analogously, we have $\Omega^* W^*_{\mathbb{G}_{\Omega}} = (X_{\Omega}\hat{J}\otimes J)W_{\mathbb{G}}\Omega^* (\hat{J}\otimes J)$.
	\end{rem}
%added small clarification in above remark

\begin{rem}\label{rems.rightdualreg}
In \cite[Theorem 2.1]{SergeyDeformation2}, a different twisted group C$^*$-algebra is introduced, which we will write $C^*_r(\Omega,\mathbb{G}):= [(\omega \otimes id)(W\Omega^*)]$. Indeed, by a similar computation to \eqref{lem.TwistedCAlgebra}, we obtain that: 
\begin{equation}
(F_s^xF_r^x)^{-1/2} (\omega_{\Lambda(u_{rs}^y),\xi_{\mathbb{G}}}\otimes id)((W^{\Omega})^*)\xi_i^x \otimes \overline{\xi_j^x} = \delta_{x,y}\delta_{i,r} \xi_{s}^x\otimes \overline{\xi_j^x},
\end{equation} so that also $C^*_r(\Omega,\mathbb{G})$ forms a C$^*$-algebra, isomorphic to the direct sum  $\underset{x\in \text{Irr}(\mathbb{G}, \Omega)}{\bigoplus^{c_0}} \mathcal{K}(H_x)$, whose $\sigma$-weak closure equals the commutant $l^{\infty}(\widehat{\mathbb{G}},\Omega)'$. In fact, using that $V = (\widehat{J}\otimes \widehat{J})W_{21}^*(\widehat{J}\otimes \widehat{J})$, we  see that  $\widehat{J} C^*_r(\Omega,\mathbb{G}) \widehat{J} = C^*_r(\mathbb{G},\widetilde{\Omega})$, with $\widetilde{\Omega}$ as in Corollary \ref{cor.opp2coc}. It hence follows (cf.\ \cite[Proposition 3.12]{SergeyDeformation2}) that $C^*_r(\Omega,\mathbb{G}) = X_{\Omega} \widehat{J}C_r^*(\mathbb{G},\Omega)\widehat{J}X_{\Omega}^*$. 
\end{rem} 

	\begin{ex}\label{ex.TwistClassRegRep}
Let us briefly relate the above construction to the classical setting.  Let $G$ be a classical compact group. Let $\omega: G\times G\longrightarrow S^1$ be a (measurable) $2$-cocycle on $G$, that is, $\omega(x,y)\omega(xy,z)=\omega(x,yz)\omega(y,z)$ for all $x,y,z\in G$. By passing to a cohomologous $2$-cocycle, we may without loss of generality assume that $\omega$ is normalized, so $\omega(e,x)=1=\omega(x,e)$ for all $x\in G$. If $\rho$ denotes the right regular representation of $G$, then the \emph{right $\omega$-regular representation of $G$ on $L^2(G)$} is the map:
	\[
		\rho^{\omega}:G \longrightarrow \mathcal{B}(L^2(G)), \qquad g  \longmapsto \rho^{\omega}_g\mbox{, $\rho^{\omega}_g(f)(x):=\omega(x,g)\rho_g(f)(x)=\omega(x,g)f(xg)$,}
	\]
		for all $f\in L^2(G)$ and $x\in G$. A direct computation shows that $\rho^\omega$ is a $\omega$-representation of $G$ on $L^2(G)$. 
		The corresponding Peter-Weyl theory can be obtained in this context (for instance, see \cite{ChengProyectiveRep} for more details).
		The \emph{twisted reduced C$^*$-algebra of $G$ with respect to $\omega$} is defined as the C$^*$-algebra $C^*_r(G, \omega):=[\int_G f(g)\rho^\omega(g) \mathrm{d}g\mid f\in C(G)\}]$. 
	\end{ex}
		
Let us now relate the notion of `being of finite type' to the regularity of a $2$-cocycle. This will be essential for the twisted Takesaki-Takai duality in the next section. For what follows, we recall that  $\mathscr{C}(V):=[\{(id\otimes \eta)(\Sigma V)\ |\ \eta\in\mathcal{B}(H)_*\}]$ for any $V \in \mathcal{B}(H\otimes H)$. Regularity of a multiplicative unitary was introduced by S. Baaj and G. Skandalis in \cite{SkandalisUnitaries}. By analogy to \cite{SkandalisUnitaries}, we define regularity of $\Omega$ in terms of regularity of the unitary $V^\Omega$. We will see (Remarks \ref{rems.RegularityInBaajCrespo} below) that our notion of regularity for a $2$-cocycle coincides with the one by S. Neshveyev and L. Tuset in \cite{SergeyDeformation2}.

\begin{defi}
A  $2$-cocycle $\Omega$ is called \emph{regular} if $\mathscr{C}(V^\Omega)=\mathcal{K}(L^2(\mathbb{G}))$.
\end{defi}

	\begin{theo}\label{theo.TwistedRegularity}
		Let $\mathbb{G}$ be a compact quantum group and $\Omega$ a (measurable)  $2$-cocycle on $\mathbb{G}$.
		\begin{enumerate}[i)]
			\item The set $\mathscr{C}(V^\Omega)\mathscr{C}(V^\Omega)$ is linearly dense in $\mathscr{C}(V^\Omega)$. In particular, $\mathscr{C}(V^\Omega)$ is an algebra. Moreover, $\mathscr{C}(V^\Omega)$ acts non-degenerately on $L^2(\mathbb{G})$, so $[\mathscr{C}(V^\Omega)L^2(\mathbb{G})]=L^2(\mathbb{G})$.
			\item The $2$-cocycle $\Omega$ is regular if and only if $\Omega$ is of finite type. 
			\item We have $\overline{span}\{U_{\mathbb{G}}C(\mathbb{G})U_{\mathbb{G}} \cdot C^*_r(\mathbb{G}, \Omega)\}=\mathcal{K}(L^2(\mathbb{G}))$ if and only if $\Omega$ is of finite type.
		\end{enumerate}
	\end{theo}
	\begin{proof}
		\begin{enumerate}[i)]
			\item Applying the definition of $\mathscr{C}(V^\Omega)$ together with the pentagonal equation satisfied by $V^\Omega$ from Theorem \ref{theo.ProjectiveRightRepMeas}, we write the following:
			\begin{equation*}
			\begin{split}
				\overline{span}\{\mathscr{C}(V^\Omega)\mathscr{C}(V^\Omega)\}&=\overline{span}\{(id\otimes \eta)(\Sigma V^\Omega)\cdot (id\otimes \eta')(\Sigma V^\Omega)\ |\ \eta, \eta'\in\mathcal{B}(L^2(\mathbb{G}))_*\}\\
				%&=\overline{span}\{(id\otimes \eta)(\Sigma V^\Omega)\cdot (\eta'\otimes id)(V^\Omega \Sigma)\ |\ \eta, \eta'\in\mathcal{B}(L^2(\mathbb{G}))_*\}\\
				&=\overline{span}\{(id\otimes \eta'\otimes \eta)(\Sigma_{13} V^\Omega_{13} \Sigma_{12}V^\Omega_{12})\ |\ \eta, \eta'\in\mathcal{B}(L^2(\mathbb{G}))_*\}\\
				&=\overline{span}\{(id\otimes \eta'\otimes \eta)(\Sigma_{13} \Sigma_{12}V^\Omega_{23} V^\Omega_{12})\ |\ \eta, \eta'\in\mathcal{B}(L^2(\mathbb{G}))_*\}\\
				&=\overline{span}\{(id\otimes \eta'\otimes \eta)(\Sigma_{13} \Sigma_{12}V^\Omega_{12} V^\Omega_{13}(V_{\mathbb{G}})_{23})\ |\ \eta, \eta'\in\mathcal{B}(L^2(\mathbb{G}))_*\}\\
				&=\overline{span}\{(id\otimes \eta'\otimes \eta)(\Sigma_{23} \Sigma_{13}V^\Omega_{12} V^\Omega_{13}(V_{\mathbb{G}})_{23})\ |\ \eta, \eta'\in\mathcal{B}(L^2(\mathbb{G}))_*\}\\
				&=\overline{span}\{(id\otimes \eta'\otimes \eta)(\Sigma_{23} V^\Omega_{32}\Sigma_{13} V^\Omega_{13}(V_{\mathbb{G}})_{23})\ |\ \eta, \eta'\in\mathcal{B}(L^2(\mathbb{G}))_*\}\\
				&=\overline{span}\{(id\otimes \eta'\otimes \eta)(\Sigma_{23}\Sigma_{23} V^\Omega_{23}\Sigma_{23}\Sigma_{13} V^\Omega_{13}(V_{\mathbb{G}})_{23})\ |\ \eta, \eta'\in\mathcal{B}(L^2(\mathbb{G}))_*\}\\
				&=\overline{span}\{(id\otimes \eta'\otimes \eta)(V^\Omega_{23}\Sigma_{23}\Sigma_{13} V^\Omega_{13}(V_{\mathbb{G}})_{23})\ |\ \eta, \eta'\in\mathcal{B}(L^2(\mathbb{G}))_*\}\\
				&\overset{(*)}{=}\overline{span}\{(id\otimes \eta'')(\Sigma V^\Omega)\ |\ \eta''\in\mathcal{B}(L^2(\mathbb{G}))_*\}\\
				&=\mathscr{C}(V^\Omega),\\
				\end{split}
			\end{equation*}
			
			where $\eta''\in\mathcal{B}(L^2(\mathbb{G}))_*$ is defined by $\eta''(a):=(\eta'\otimes \eta)(V^{\Omega}\Sigma(1\otimes a)V_{\mathbb{G}})$, for all $a\in\mathcal{B}(L^2(\mathbb{G}))$; and $(*)$ follows from the fact that the vector space generated by these functionals is dense in $\mathcal{B}(L^2(\mathbb{G}))_*$.
			
To see that $\mathscr{C}(V^\Omega)$ acts non-degenerately, take $\zeta\in (\mathscr{C}(V^\Omega)L^2(\mathbb{G}))^{\perp}$. For every $\xi,\xi', x\in L^2(\mathbb{G})$ we have $0=\langle \zeta, (id\otimes \omega_{\xi,\xi'}(\Sigma V^\Omega)(x))\rangle=\langle \zeta\otimes \xi, \Sigma V^\Omega(x\otimes \xi')\rangle$. Since $\Sigma V^\Omega$ is a unitary in $\mathcal{B}(L^2(G)\otimes L^2(G))$ so surjective, the above equality implies $(\mathscr{C}(V^\Omega)L^2(\mathbb{G}))^{\perp}=(0)$.
			\item  		
By Theorem \ref{theo.TwistedOrthogonalityRel}, the  set $\{(id\otimes \omega_{\Lambda(u^y_{kl}), \Lambda(a)})(\Sigma V^\Omega)\ |\ y\in \text{Irr}(\mathbb{G}, \Omega), k,l=1,\ldots, n_y, a\in \text{Pol}(\mathbb{G})\}$ is dense in $\mathscr{C}(V^\Omega)$. Given $u^x_{ij}\in \text{Pol}(\mathbb{G},\Omega)$ we compute with the help of the twisted orthogonality relations from Theorem \ref{theo.TwistedOrthogonalityRel} that for $\zeta\in L^2(\mathbb{G})$:
			\begin{equation*}
			\begin{split}
				\langle \zeta, (id\otimes \omega_{\Lambda(u^y_{kl}), \Lambda(a)})(\Sigma V^\Omega)\Lambda(u^x_{ij})\rangle &=\langle  \Lambda(u^y_{kl})\otimes \zeta, V^\Omega(\Lambda(u^x_{ij})\otimes \Lambda(a))\rangle\\
				&=\langle  \Lambda(u^y_{kl})\otimes \zeta, \Delta_\Omega(u^x_{ij})(\xi_{\mathbb{G}}\otimes \Lambda(a))\rangle\\
				&=\langle \Lambda(u^y_{kl})\otimes \zeta, \big(\overset{n_x}{\underset{r=1}{\sum}}u^x_{ir}\otimes u^x_{rj}\big)(\xi_{\mathbb{G}}\otimes \Lambda(a))\rangle\\
				&= \overset{n_x}{\underset{r=1}{\sum}} h_{\mathbb{G}}\big((u^y_{kl})^*u^x_{ir}\big) \langle \zeta, \lambda(u^x_{rj})\Lambda(a) \rangle\\
				&=F^x_k\delta_{x,y}\delta_{k,i}\langle \zeta,\lambda(u^x_{lj})\Lambda(a)\rangle,
			\end{split}
			\end{equation*}
so with respect to the orthonormal basis of Theorem \ref{theo.TwistedPeterWeyL2} we have: 
\begin{equation}
(F^x_k)^{-1}(id\otimes \omega_{\Lambda(u^y_{kl}), \Lambda(a)})(\Sigma V^\Omega)(F_i^x)^{-1/2}\Lambda(u_{ij}^x) = \delta_{x,y}\delta_{k,i}(F_l^x)^{-1/2}\lambda(u^x_{lj})\Lambda(a).
\end{equation}

Assume now that $\Omega$ is regular. Taking $a=1$, we see from the above that $(id\otimes \omega_{\Lambda(u^y_{kl}), \xi_{\mathbb{G}}})(\Sigma V^\Omega)$ is of the form $T^x_{rs} \otimes 1$ for some non-zero operator $T^x_{rs}$ with respect to the model $L^2(\mathbb{G}) \cong \oplus_{x\in \text{Irr}(\mathbb{G},\Omega)} H_x \otimes \overline{H_x}$. As this operator needs to be compact, we see that all $n_x$ need to be finite, hence $\Omega$ is of finite type. 

Conversely, if $\Omega$ is of finite type, we see from the above computation (and the fact that the operators $\Lambda(b) \mapsto \Lambda(ba)$ are bounded for $a\in \text{Pol}(\mathbb{G})$) that $\mathscr{C}(V^{\Omega}) \subset \mathcal{K}(L^2(\mathbb{G}))$. To see that this is an equality, it is sufficient to show that the commutant $\mathscr{C}(V^{\Omega})' = \mathbb{C}$. Now if $x \in \mathscr{C}(V^{\Omega})'$, then in particular $x$ commutes with all $(id\otimes \omega_{\Lambda(u^y_{kl}), \xi_{\mathbb{G}}})(\Sigma V^\Omega)$, and it follows from the above computation and \eqref{EqActStGen} that $x\in l^{\infty}(\widehat{\mathbb{G}},\Omega)$. On the other hand, the computation above also shows that then $x \in \rho(L^{\infty}(\mathbb{G}))' = L^{\infty}(\mathbb{G})$, hence $x\in  l^{\infty}(\widehat{\mathbb{G}},\Omega) \cap L^{\infty}(\mathbb{G})$. But by Remark \ref{rems.rightdualreg}, this implies $(W^\Omega)^*(1\otimes x)W^{\Omega} = 1\otimes x$, hence ${}_{\Omega}\Delta_{\Omega^*}(x)=1\otimes x$ and so $x \in \mathbb{C}$.
	
			\item By \eqref{EqActStGen} we see that $C^*_r(\mathbb{G}, \Omega)$ is formed by compact operators if and only if $\Omega$ is of finite type. Hence $\overline{span}\{U_{\mathbb{G}}C(\mathbb{G})U_{\mathbb{G}} \cdot C^*_r(\mathbb{G}, \Omega)\} \subset \mathcal{K}(L^2(\mathbb{G}))$ if and only if $\Omega$ is of finite type. To see that this is an equality if $\Omega$ is of finite type, we can follow as similar strategy as in ii). Alternatively, conjugating with $\widehat{J}$ and taking into account Remark \ref{rems.rightdualreg}, we see that we have an inclusion $C_r^*(\Omega,\mathbb{G})JC(\mathbb{G})J \subseteq \mathcal{K}(L^2(\mathbb{G}))$, and this must be an equality by the discussion following \cite[Definition 2.9]{SergeyDeformation2}. Conjugating back with $\widehat{J}$, we see that $\overline{span}\{U_{\mathbb{G}}C(\mathbb{G})U_{\mathbb{G}} \cdot C^*_r(\mathbb{G}, \Omega)\}=\mathcal{K}(L^2(\mathbb{G}))$.
\end{enumerate}
\end{proof}

	\begin{rems}\label{rems.RegularityInBaajCrespo}
		\begin{enumerate}[i)]
\item As follows from the end of the above proof, our notion of regularity indeed coincides with the notion of regularity of a $2$-cocycle as introduced in \cite[Definition 2.9]{SergeyDeformation2}. 
			\item Up to unitary conjugation, our operator $V^{\Omega}$ also coincides with the operator $\Sigma (V_{21}^1)^*\Sigma$ with $V_{21}^1$ as it appears in \cite[Proposition 2.44]{BaajCrespo}. It hence follows from that result that $\Omega$ is regular if and only if $\mathbb{G}_{\Omega}$ is regular, and hence that $\mathbb{G}_{\Omega}$ is never regular if $\Omega$ is not of finite type. It is unclear at the moment if in general $\mathbb{G}_{\Omega}$ is semi-regular (this holds in all known cases). By  \cite[Proposition 2.44]{BaajCrespo}, this is equivalent to $\mathcal{K}(L^2(\mathbb{G})) \subseteq \mathscr{C}(V^{\Omega})$. 
		\end{enumerate}
	\end{rems}

The following proposition is a straightforward adaptation of \cite[Proposition 2.3]{SergeyDeformation2} to our setting. 
	
	\begin{pro}\label{pro.TwistedDualCoMultiplication}
Let $\mathbb{G}$ be a compact quantum group and $\Omega$ a  $2$-cocycle. Then the twisted reduced C$^*$-algebra $C^*_r(\mathbb{G}, \Omega)$ is a right $\widehat{\mathbb{G}}$-C$^*$-algebra with action $\alpha^\Omega$ defined by:
		$$\alpha^\Omega(x)=\Sigma (V^\Omega)^*(1\otimes x)V^\Omega\Sigma,$$ for all $x\in C^*_r(\mathbb{G}, \Omega)$.
	\end{pro}
	\begin{proof}
Given $x\in C^*_r(\mathbb{G}, \Omega)$, assume without loss of generality that $x:=(id\otimes \eta)(V^\Omega)$ for some $\eta\in \mathcal{B}(L^2(\mathbb{G}))_*$. Then, with the help of the pentagonal equation satisfied by $V^\Omega$ (see   Theorem-Definition \ref{theo.ProjectiveRightRepMeas}), we have $\alpha^\Omega(x)=(id\otimes id\otimes \eta)V^\Omega_{23}(V_{\mathbb{G}})_{13}\in \widetilde{M}(c_0(\widehat{\mathbb{G}})\otimes C^*_r(\mathbb{G}, \Omega))$, 		which shows that $\alpha^\Omega$ is well-defined as a (injective) $*$-homomorphism $C^*_r(\mathbb{G}, \Omega)\longrightarrow \widetilde{M}(c_0(\widehat{\mathbb{G}})\otimes C^*_r(\mathbb{G}, \Omega)$). Next, we are going to show that $\alpha^\Omega$ defines an action of $\widehat{\mathbb{G}}$ on $C^*_r(\mathbb{G}, \Omega)$. On the one hand, since the elements of the form $x=(id\otimes \eta)(V^\Omega)\in C^*_r(\mathbb{G}, \Omega)$ with $\eta\in C(\mathbb{G})_*$ are dense in $C^*_r(\mathbb{G}, \Omega)$, then the previous computation shows also that the subspace $\alpha^\Omega(C^*_r(\mathbb{G}, \Omega))(c_0(\widehat{\mathbb{G}})\otimes 1)$ is dense in $c_0(\widehat{\mathbb{G}})\otimes C^*_r(\mathbb{G}, \Omega)$. On the other hand, applying $id\otimes \alpha^\Omega$ and $\widehat{\Delta}\otimes id$ to the above expression, we obtain $(id\otimes \alpha^\Omega)\alpha^\Omega(x)=(id\otimes id\otimes id\otimes \eta)V^\Omega_{34}(V_{\mathbb{G}})_{24}(V_{\mathbb{G}})_{14}=(\widehat{\Delta}\otimes id)\alpha^\Omega(x)$ by a direct computation.
	\end{proof}
	\begin{rem}\label{rem.BetaOmegaAction}
By the formula $\beta^\Omega(x):=(V^\Omega)^*(1\otimes x)V^\Omega$ for $x\in C^*_r(\mathbb{G}, \Omega)$, we can also view $(C^*_r(\mathbb{G}, \Omega), \beta^\Omega)$ as a (left) $\widehat{\mathbb{G}}^{cop}$-C$^*$-algebra.
	\end{rem}

	\subsection{Twisted crossed products}\label{sec.TwistedCrossedProd}
	
In this section, we consider twisted crossed products, cf.\ again \cite{SergeyDeformation2}. We start however from the twisted side, and work our way back to the original compact quantum group. As to spare the reader a battle with conventions, we spell out some of the details particular to our setting. 

	\begin{defi}\label{defi.TwistDynamSyst}
		A (measurable) left twisted dynamical system is the data $(\mathbb{G}, A, \delta, \Omega)$ where $\mathbb{G}$ is a compact quantum group, $\Omega$ is a  $2$-cocycle of finite type on $\mathbb{G}$, $A$ is a C$^*$-algebra and $\delta:A\longrightarrow A\otimes C(\mathbb{G}_{\Omega})$ is left action of $\mathbb{G}_{\Omega}$. 

		We write $\mathbb{G}\overset{(\delta, \Omega)}{\curvearrowright} A$, and say that $\delta$ is a twisted action of $\mathbb{G}$ on $A$ with respect to $\Omega$ or simply that $\delta$ is an $\Omega$-action of $\mathbb{G}$ on $A$. We say that $(A,\delta)$ is a left $\Omega$-$\mathbb{G}$-C$^*$-algebra if moreover $\delta$ is injective.
	\end{defi}

	\begin{defi}\label{defi.TwistedCrossedProduct}
		Let $(\mathbb{G}, A, \delta, \Omega)$ be a twisted dynamical system. The twisted reduced crossed product of $A$ by $\mathbb{G}$ with respect to $(\delta, \Omega)$, denoted by $A\underset{r, (\delta,\Omega)}{\rtimes} \mathbb{G}$, is the C$^*$-algebra defined by:
		$$A\underset{r, (\delta,\Omega)}{\rtimes} \mathbb{G}:=C^*\langle (id\otimes \lambda)\delta(A)(1\otimes C^*_r(\mathbb{G}, \Omega))\rangle\subset \mathcal{L}_A(A\otimes L^2(\mathbb{G})).$$
	\end{defi}

	\begin{note}\label{note.LeftTwistedDynSystCrossedProduct}
		To lighten the notation, we will omit the representation $\lambda$ appearing in the definition of $A\underset{r, (\delta,\Omega)}{\rtimes} \mathbb{G}$. Note that our convention of writing $\mathbb{G}$ on the right in the crossed product notation is in line with the notation followed e.g.\ in \cite{YukiBCTorsion}.
%If  $(\mathbb{G}, A, \delta, \Omega)$ is a \emph{left} twisted dynamical system, then analogously we can define the corresponding twisted reduced crossed product: $\mathbb{G}\underset{r, (\delta,\Omega)}{\ltimes} A:=C^*\langle(C^*_r(\mathbb{G}, \Omega)\otimes 1)(\rho\otimes id)(\delta(A))\rangle\subset \mathcal{L}_A(L^2(\mathbb{G})\otimes A)$. In this case it is convenient to work with the picture of $C^*_r(\mathbb{G}, \Omega)$ in terms of $W^\Omega$ instead of $V^\Omega$ (recall Lemma \ref{lem.TwistedCAlgebra}). In this way, all of the following constructions remain valid for \emph{left} twisted dynamical systems.
\end{note}

	\begin{lem}\label{lem.TwistedCrossedProd}
We have $A\underset{r, (\delta,\Omega)}{\rtimes} \mathbb{G}=\overline{span}\{\delta(A)(1\otimes C^*_r(\mathbb{G}, \Omega))\}$.
\end{lem}
\begin{proof} It is enough to show that 
\begin{equation}\label{EqInclTwist}
[(1\otimes C^*_r(\mathbb{G}, \Omega))\delta(A)] \subset [\delta(A)(1\otimes C^*_r(\mathbb{G}, \Omega))]
\end{equation} 
But, using the implementation of ${}_{\Omega}\Delta=\Omega\cdot \Delta$ in terms of $V_{\mathbb{G}}$ and $V^\Omega$, the compatibility of $\delta$ with $\Delta$ as a \emph{twisted} action of $\mathbb{G}$ on $A$ yields that $(1\otimes (id\otimes \eta)(V^\Omega))\delta(a)=(id\otimes id \otimes \eta)\big((\delta\otimes id)\delta(a)(1\otimes V^\Omega)\big) = \lim \sum_i  \delta(a_i)(1\otimes (id\otimes \eta\cdot u_i)(V^\Omega))$, for all $a\in A$ and all $\eta\in \mathcal{B}(L^2(\mathbb{G}))_*$, where $\delta(a) = \lim \sum_i a_i \otimes u_i$. This proves \eqref{EqInclTwist}.
\end{proof}
		As a consequence, the maps $A\longrightarrow\mathcal{L}_{A}(A\otimes L^2(\mathbb{G}))$ and $C^*_r(\mathbb{G}, \Omega)\longrightarrow\mathcal{L}_{A}(A\otimes L^2(\mathbb{G}))$ given by $a\mapsto \delta(a)$ and $x\mapsto 1\otimes x$, send $A$ and $C^*_r(\mathbb{G}, \Omega)$ respectively onto non-degenerate C$^*$-subalgebras of $M(A\underset{r, (\delta,\Omega)}{\rtimes} \mathbb{G})$.
	
	\begin{ex}
		We note that both Definition \ref{defi.TwistDynamSyst} and Definition \ref{defi.TwistedCrossedProduct} are natural dual versions of the classical framework. Let $G$ be a classical compact group and $A$ a unital C$^*$-algebra. Given a Borel measurable map $\omega: G\times G\longrightarrow \mathcal{U}(A)$, a $\omega$-action of $G$ on $A$ is a map $\alpha: G\longrightarrow Aut(A)$ such that $\alpha_{g_1}\circ \alpha_{g_2}=Ad_{\omega(g_1,g_2)} \circ \alpha_{g_1g_2}$ for all $g_1,g_2\in G$; $\omega(x,y)\omega(xy,z)=\alpha_x\big(\omega(y, z)\big)\omega(x, yz)$ for all $x,y,z\in G$ and $\omega(x,e)=1=\omega(e,x)$ for all $x\in G$. Consider the vector space of continuous functions on $G$ with values in $A$ equipped with the usual point-wise operations, $C(G, A)$. We define the \emph{twisted convolution product on $C(G, A)$ with respect to $\omega$} by:
		$$(f\underset{\omega}{\ast}g)(x):=\int_G f(y)\alpha_y\big(g(y^{-1}x)\big)\omega(y, y^{-1}x)dy,$$
		for all $f,g\in C(G, A)$ and $x\in G$. We define the \emph{twisted involution on $C(G, A)$ with respect to $\omega$} by:
		$$f^{\underset{\omega}{*}}(x):=\omega(x, x^{-1}) \alpha_x\big(f(x^{-1})^*\big),$$
		for all $f\in C(G, A)$ and $x\in G$. Straightforward computations show that $C(G, A)$ is a $*$-algebra with the product and involution above. Next, by applying standard arguments (analogous to the untwisted case) we find that $(L^2(G)\otimes H_0, \pi, \lambda^\omega)$ is a \emph{faithful} covariant $\omega$-representation of $(A, \alpha)$, where $\pi: A \longrightarrow \mathcal{B}(L^2(G)\otimes H_0)$ is such that $\pi(a)(f\otimes \xi)(x):=\Big(f\otimes\pi_0(\alpha_{x^{-1}}(a))(\xi)\Big)(x)$, for all $a\in A$, $f\in L^2(G)$, $\xi\in H_0$, $x\in G$; and $(\pi_0, H_0)$ is any faithful representation of $A$. Thus we define the \emph{reduced twisted crossed product} by $A\underset{r, (\alpha, \omega)}{\rtimes} G:=\overline{(\pi,\lambda^\omega)(C(G, A))}^{||\cdot||_{\mathcal{B}(L^2(G, H_0))}}$ and one shows that this definition does not depend on the faithful representation $\pi_0$. Alternatively, we have $A\underset{r, (\alpha,\omega)}{\rtimes} G:=\overline{span}\{\alpha(A)(1\otimes C^*_r(G, \omega))\}$, where the action $G\overset{\alpha}{\curvearrowright} A$ is viewed as a map $\alpha: A\longrightarrow M(A\otimes C(G))\hookrightarrow \mathcal{L}_A(A\otimes L^2(G))$.
		
		Specially interesting is the case when $\omega$ is scalar valued (that is, $\omega$ takes its values on $S^1\cong S^1\cdot 1_A$), then $\omega$ is a usual (normalized) $2$-cocycle on $G$ and $\alpha$ a group homomorphism. In this case, we observe that these constructions yield the usual C$^*$-algebras: if $\omega= 1$, we have $A\underset{r, (\alpha,\omega)}{\rtimes} G=A\underset{r, \alpha}{\rtimes} G$; if $\alpha$ is trivial, we have $A\underset{r, (\alpha,\omega)}{\rtimes} G=A\otimes C^*_r(G, \omega)$.

	\end{ex}

	\begin{prodefi}\label{pro.TwistedDualAction}
		Let $(\mathbb{G}, A, \delta, \Omega)$ be a twisted dynamical system. The twisted reduced crossed product, $A\underset{r, (\delta,\Omega)}{\rtimes} \mathbb{G}$, is a $\widehat{\mathbb{G}}^{cop}$-C$^*$-algebra with action $\widehat{\mathbb{G}}^{cop}\overset{\widehat{\delta}}{\curvearrowright} A\underset{r, (\delta,\Omega)}{\rtimes} \mathbb{G}$ such that: 
		$$\widehat{\delta}\big(\delta(a)(1\otimes x)\big)=(\delta(a)\otimes 1)\big(1\otimes (V^\Omega)^*(1\otimes x)V^\Omega\big),$$ for all $a\in A$ and all $x\in C^*_r(\mathbb{G}, \Omega)$. The action $\widehat{\mathbb{G}}^{cop}\overset{\widehat{\delta}}{\curvearrowright} A\underset{r, (\delta,\Omega)}{\rtimes} \mathbb{G}$ is called \emph{twisted dual action of $(\delta, \Omega)$} or \emph{$\Omega$-dual action of $\delta$}.
	\end{prodefi}
	\begin{proof}
		Let us consider the unitary $\widetilde{V}_{\mathbb{G}}$ as in Theorem-Definition \ref{theo.KacSystemG}. We are going to show that $\widehat{\delta}$ can be written as a conjugation by $1\otimes \widetilde{V}_{\mathbb{G}}$. Given $x\in C^*_r(\mathbb{G}, \Omega)$, assume without loss of generality that $x:=(id\otimes \eta)(V^\Omega)$ for some $\eta\in C(\mathbb{G})_*$. On the one hand, using the pentagonal equation satisfied by $V^\Omega$ (see  Theorem-Definition \ref{theo.ProjectiveRightRepMeas}) and Lemma \ref{lem.IdentityTwistDualAction}, a direct computation shows that $(V^\Omega)^*(1\otimes x)V^\Omega=\widetilde{V}_{\mathbb{G}}(x\otimes 1)\widetilde{V}_{\mathbb{G}}^*$.
		On the other hand, Theorem-Definition \ref{theo.KacSystemG} guarantees that $\widetilde{V}_{\mathbb{G}}(y\otimes 1)\widetilde{V}_{\mathbb{G}}^*=y\otimes 1$, for all $y\in C(\mathbb{G})'' = L^{\infty}(\mathbb{G})\supseteq C(\mathbb{G}_{\Omega})$. Combining these two expressions, it is easy to see that $(\delta(a)\otimes 1)\big(1\otimes (V^\Omega)^*(1\otimes x)V^\Omega\big)=(1\otimes \widetilde{V}_{\mathbb{G}})\big(\delta(a)(1\otimes x)\otimes 1\big)(1\otimes \widetilde{V}_{\mathbb{G}}^*)$, for all $a\in A$ and all $x\in C^*_r(\mathbb{G}, \Omega)$.
		In other words, these expressions show that the formula of the statement defines a (injective) $*$-homomorphism $A\underset{r, (\delta,\Omega)}{\rtimes} \mathbb{G}\longrightarrow \widetilde{M}((A\underset{r, (\delta,\Omega)}{\rtimes} \mathbb{G})\otimes c_0(\widehat{\mathbb{G}}))$ given precisely by $\widehat{\delta}(z)=(1\otimes \widetilde{V}_{\mathbb{G}})(z\otimes 1)(1\otimes \widetilde{V}_{\mathbb{G}}^*)$, for all $z\in A\underset{r, (\delta,\Omega)}{\rtimes} \mathbb{G}$. It remains to show that $\widehat{\delta}$ defines an action of $\widehat{\mathbb{G}}$ on $A\underset{r, (\delta,\Omega)}{\rtimes} \mathbb{G}$. On the one hand, the density condition for $\widehat{\delta}$ is obtained as follows:
		\begin{equation*}
		\begin{split}
			\Big[\widehat{\delta}(A\underset{r, (\delta,\Omega)}{\rtimes} \mathbb{G})(1\otimes c_0(\widehat{\mathbb{G}}))\Big]&=\Big[(\delta(A)\otimes 1)\big(1\otimes (V^\Omega)^*(1\otimes C^*_r(\mathbb{G}, \Omega))V^\Omega\big)(1\otimes 1\otimes c_0(\widehat{\mathbb{G}}))\Big]\\
			&\overset{(*)}{=}\Big[(\delta(A)\otimes 1)\big(1\otimes C^*_r(\mathbb{G}, \Omega) \otimes c_0(\widehat{\mathbb{G}})\big)\Big]\\
			&=(A\underset{r, (\delta,\Omega)}{\rtimes} \mathbb{G})\otimes c_0(\widehat{\mathbb{G}})\mbox{,}
		\end{split}
		\end{equation*}
		where in $(*)$ we use Remark \ref{rem.BetaOmegaAction}. On the other hand, the compatibility of $\widehat{\delta}$ with $\widehat{\Delta}^{cop}$ is obtained by a direct computation using again Remark \ref{rem.BetaOmegaAction}.
	\end{proof}

	To end this section, let us introduce the following nomenclature for a special type of quantum dynamical systems.	
	\begin{defi}\label{defi.TwistedInner}
		Let $\mathbb{G}$ be a compact quantum group and $\Omega$ a  $2$-cocycle of finite type on $\mathbb{G}$. Let $H$ be a Hilbert space. A left twisted dynamical system $(\mathbb{G},\mathcal{K}(H),\delta,\Omega)$ is called $\Omega$-inner if there exists a $\Omega$-representation $u\in \mathcal{B}(H)\overline{\otimes} L^\infty(\mathbb{G})$ such that $\delta(a)=\delta_u(a) = u(a\otimes 1)u^*$, for all $a\in \mathcal{K}(H)$. In this case, the data $(\mathbb{G}, \mathcal{K}(H), \delta, \Omega, u)$ is called a \emph{left twisted inner dynamical system} or \emph{right $\Omega$-inner dynamical system}.
	\end{defi}
	
So a left twisted dynamical system $(\mathbb{G},\mathcal{K}(H),\delta,\Omega)$ is nothing but a projective left $\mathbb{G}_{\Omega}$-representation induced from an $\Omega$-representation of $\mathbb{G}$. 

	It is well-known that an inner action is exterior equivalent to the trivial one, so that the corresponding crossed products are isomorphic (see for instance \cite{Phillips} for more details). The following proposition shows that a similar phenomenon occurs in the quantum group setting when the action is $\Omega$-inner.
	
	\begin{pro}\label{pro.OmegaInnerDynamicalSystem}
		Let $\mathbb{G}$ be a compact quantum group and $\Omega$ a $2$-cocycle of finite type on $\mathbb{G}$. Let $H$ be a Hilbert space. If $(\mathbb{G}, \mathcal{K}(H), \delta, \Omega, v)$ is a left $\Omega$-inner dynamical system, then $$\mathcal{K}(H)\underset{r, (\delta,\Omega)}{\rtimes} \mathbb{G}\cong \mathcal{K}(H)\otimes c_0(\widehat{\mathbb{G}}).$$
\end{pro}
	\begin{proof}
We can represent $\mathcal{K}(H)\underset{r, (\delta,\Omega)}{\rtimes} \mathbb{G}$ as the normclosure of $(1\otimes C_r^*(\mathbb{G},\Omega))\delta_u(\mathcal{K}(H))$ on $H\otimes L^2(\mathbb{G})$. It is then sufficient to show that $u^*[(1\otimes C_r^*(\mathbb{G},\Omega))\delta_u(\mathcal{K}(H))]u = \mathcal{K}(H)\otimes c_0(\widehat{\mathbb{G}})$. But since $u_{12}^*(V^{\Omega})_{23}u_{12} = u_{13}(V_{\mathbb{G}})_{23}$ as $u$ is a $\Omega$-representation, we have that $u^*[(1\otimes C_r^*(\mathbb{G},\Omega))\delta_u(\mathcal{K}(H))]u  = [\{(id\otimes id\otimes \omega)(u_{13}(x\otimes V_{\mathbb{G}}))\mid x\in \mathcal{K}(H),\omega \in \mathcal{B}(L^2(\mathbb{G}))_*\}]$. As $u$ is the direct sum of finite-dimensional $\Omega$-representations, it follows that this last set equals $[\{(id\otimes id\otimes \omega)(x\otimes V_{\mathbb{G}})\mid x\in \mathcal{K}(H),\omega \in \mathcal{B}(L^2(\mathbb{G}))_*\}] = \mathcal{K}(H)\otimes c_0(\widehat{\mathbb{G}})$. 
	\end{proof}

\begin{rem}\label{rem.OmegaInnerDynamicalSystem}
Note that if $v$ is a $\Omega^*$-representation, we also have an ordinary left action $\mathbb{G}\overset{\delta}{\curvearrowright}\mathcal{K}(H)$ by putting $\delta(a) = \delta_v(a)=v(a\otimes 1)v^*$, for all $a\in\mathcal{B}(H)$. We then say that the action $\mathbb{G}\overset{\delta}{\curvearrowright}\mathcal{K}(H)$ is $\Omega$-inner. An analogous computation as above yields that in this case we have for the untwisted crossed product that $ \mathcal{K}(H)\underset{r, \delta}{\rtimes}\mathbb{G}\underset{\mathrm{Ad}(v^*)}{\cong}  \mathcal{K}(H)\otimes C^*_r(\mathbb{G}, \Omega)$. Analogously, from an $\Omega$-representation $u$ of $\mathbb{G}$, one could consider an ordinary right action $\mathcal{K}(H)\overset{\delta}{\curvearrowleft}\mathbb{G}$ by putting $\delta(a)=\Sigma(u^*(a\otimes 1)u)$, for all $a\in\mathcal{B}(H)$. In this case, the (untwisted) crossed product is defined as $\mathbb{G}\underset{r, \delta}{\ltimes}\mathcal{K}(H) = C^*\langle \widehat{\lambda}(c_0(\widehat{\mathbb{G}})\otimes 1)(\rho\otimes id)\delta(A)\rangle$ where we recall that $\rho(x) = U_{\mathbb{G}}\lambda(x)U_{\mathbb{G}}$ for $x\in C(\mathbb{G})$. Then we now have, upon using the regular representation $W_{\mathbb{G}}= \check{V}_{\mathbb{G}}$ that $\mathbb{G}\underset{r, \delta}{\ltimes}\mathcal{K}(H)\underset{\mathrm{Ad}(u_{21}U_{\mathbb{G},1})}{\cong} C^*_r(\Omega,\mathbb{G})\otimes \mathcal{K}(H)$, where we use the notation of Remark \ref{rems.rightdualreg}.
\end{rem}

	\subsection{Twisted Takesaki-Takai duality, twisted descent map and twisted Baaj-Skandalis duality}\label{sec.TTDuality}

	\begin{defi}
		Let $(\mathbb{G}, A, \delta, \Omega)$ be a twisted dynamical system. The double reduced crossed product of $A$ by $\widehat{\mathbb{G}}^{cop}$ with respect to $\widehat{\delta}$, denoted by $\big(A\underset{r, (\delta,\Omega)}{\rtimes} \mathbb{G}\big)\underset{r, \widehat{\delta}}{\rtimes}\widehat{\mathbb{G}}^{cop}$, is the C$^*$-algebra defined by:
		$$\big(A\underset{r, (\delta,\Omega)}{\rtimes} \mathbb{G}\big)\underset{r, \widehat{\delta}}{\rtimes}\widehat{\mathbb{G}}^{cop}:=C^*\langle\widehat{\delta}_U(A\underset{r, (\delta,\Omega)}{\rtimes} \mathbb{G})(1\otimes C(\mathbb{G}))\rangle\subset \mathcal{L}_{A}((A\otimes L^2(\mathbb{G}))\otimes L^2(\mathbb{G})),$$ where $\widehat{\delta}_U(x) = (1\otimes U_{\mathbb{G}})\widehat{\delta}(x)(1\otimes U_{\mathbb{G}})$ for $x\in A\underset{r, (\delta,\Omega)}{\rtimes} \mathbb{G}$. 
	\end{defi}
	\begin{rem}
		Observe that the crossed product $\big(A\underset{r, (\delta,\Omega)}{\rtimes} \mathbb{G}\big)\underset{r, \widehat{\delta}}{\rtimes}\widehat{\mathbb{G}}^{cop}$ defined above is the usual one as, say, given in  \cite[Section 7]{SkandalisUnitaries} (we do not consider any deformation in this definition, contrary to Definition \ref{defi.TwistedCrossedProduct}). In particular, we have automatically that $\big(A\underset{r, (\delta,\Omega)}{\rtimes} \mathbb{G}\big)\underset{r, \widehat{\delta}}{\rtimes}\widehat{\mathbb{G}}^{cop}=\overline{span}\{\widehat{\delta}_U(A\underset{r, (\delta,\Omega)}{\rtimes} \mathbb{G})(1\otimes C(\mathbb{G}))\}$. 
	\end{rem}
	
	Since $\big(A\underset{r, (\delta,\Omega)}{\rtimes} \mathbb{G}\big)\underset{r, \widehat{\delta}}{\rtimes}\widehat{\mathbb{G}}^{cop}$ defined above is a usual crossed product, we can define the corresponding dual action $\widehat{\widehat{\delta}}$. Hence, the following proposition follows from the standard theory of crossed products.
	\begin{prodefi}
		Let $(\mathbb{G}, A, \delta, \Omega)$ be a twisted dynamical system. The double reduced crossed product $\big(A\underset{r, (\delta,\Omega)}{\rtimes} \mathbb{G}\big)\underset{r, \widehat{\delta}}{\rtimes}\widehat{\mathbb{G}}^{cop}$ is a $\mathbb{G}$-C$^*$-algebra with action $\mathbb{G}\overset{\widehat{\widehat{\delta}}}{\curvearrowright}\big(A\underset{r, (\delta,\Omega)}{\rtimes} \mathbb{G}\big)\underset{r, \widehat{\delta}}{\rtimes}\widehat{\mathbb{G}}^{cop}$ such that:
		$$\widehat{\widehat{\delta}}\big(\widehat{\delta}_U(z)(1\otimes y)\big)=(\widehat{\delta}_U(z)\otimes 1)(1\otimes V_{\mathbb{G}}(y\otimes 1)V_{\mathbb{G}}^*)\mbox{,}$$
		for all $z\in A\underset{r, (\delta,\Omega)}{\rtimes} \mathbb{G}$ and all $y\in C(\mathbb{G})$. The action $\mathbb{G}\overset{\widehat{\widehat{\delta}}}{\curvearrowright}\big(A\underset{r, (\delta,\Omega)}{\rtimes} \mathbb{G}\big)\underset{r, \widehat{\delta}}{\rtimes}\widehat{\mathbb{G}}^{cop}$ is called \emph{twisted double dual action of $(\delta, \Omega)$} or \emph{$\Omega$-double dual action of $\delta$}.
	\end{prodefi}

The following theorem is a special case of \cite[Theorem 3.6]{SergeyDeformation2} by means of \cite[Proposition 3.5]{SergeyDeformation2}, see also \cite[Section 1]{VaesBiproduit} in the von Neumann algebraic setting (where no regularity assumption is needed).

	\begin{theo}[Twisted Takesaki-Takai duality]\label{theo.TwistTakesakiTakai}
		Let $(\mathbb{G}, A, \delta, \Omega)$ be a twisted dynamical system with $\Omega$ of finite type.
		\begin{enumerate}[i)]
			\item There is a canonical isomorphism of C$^*$-algebras, $A\otimes \mathcal{K}(L^2(\mathbb{G})) \cong \big(A\underset{r, (\delta,\Omega)}{\rtimes} \mathbb{G}\big)\underset{r, \widehat{\delta}}{\rtimes}\widehat{\mathbb{G}}^{cop}$, given by the map $x \mapsto (U_{\mathbb{G}})_3(V^{\Omega})_{23}^*(\delta\otimes id)(x)(V^{\Omega})_{23}(U_{\mathbb{G}})_3$.
			\item 
Under the above isomorphism, the twisted double dual action $\widehat{\widehat{\delta}}$ of $(\delta, \Omega)$ is conjugate to the action $\mathbb{G}\overset{\widetilde{\delta}}{\curvearrowright} A\otimes \mathcal{K}(L^2(\mathbb{G}))$ defined by
			$\widetilde{\delta}:=Ad_{(W^{\Omega})_{32}}\circ \delta_{13}\mbox{,}$
			where $\delta_{13}$ denotes the amplified twisted action of $\mathbb{G}$ on $A\otimes \mathcal{K}(L^2(\mathbb{G}))$ such that: $$\delta_{13}(a\otimes T)=(1\otimes \Sigma)(\delta(a)\otimes id)(1\otimes \Sigma)(1\otimes T\otimes 1)\in A\otimes \mathcal{K}(L^2(\mathbb{G}))\otimes C(\mathbb{G}_{\Omega})\mbox{,}$$
			for all $a\in A$ and $T\in\mathcal{K}(L^2(\mathbb{G}))$.
		\end{enumerate}
	\end{theo}
\begin{proof}
By Theorem \ref{theo.TwistedRegularity}, we have $\mathcal{K}(L^2(\mathbb{G})) = [C^*_r(\mathbb{G})U_{\mathbb{G}}C(\mathbb{G})U_{\mathbb{G}}]$. Since $\delta$  is continuous, we can hence write $A\otimes \mathcal{K}(L^2(\mathbb{G})) = [\delta(A)(1\otimes \mathcal{K}(L^2(\mathbb{G})))] = [\delta(A)(1\otimes C_r^*(\mathbb{G},\Omega)U_{\mathbb{G}}C(\mathbb{G})U_{\mathbb{G}})]$, and the first item then follows by a straightforward computation. The second item then follows from Remark \ref{rems.rightdualreg} together with the fact that $W_{\mathbb{G}}(1\otimes UyU)(W_{\mathbb{G}})^* = (1\otimes U)\Delta^{op}(y)(1\otimes U)$, which follows from the identity $W = \check{V}_{\mathbb{G}}$ in Theorem-Definition \ref{theo.KacSystemG}.
\end{proof}

	As a corollary of the twisted Takesaki-Takai duality established in Theorem \ref{theo.TwistTakesakiTakai} we obtain the following generalization of the well known \emph{Packer-Raeburn's untwisting trick} or \emph{Packer-Raeburn's stabilisation trick} \cite{PackerRaeburnTrick1} to compact quantum groups.
	\begin{pro}[Quantum Packer-Raeburn's untwisting trick]\label{pro.QPackerRaeburnUntwistTrick}
		Let $(\mathbb{G}, A, \delta, \Omega)$ be a twisted dynamical system with $\Omega$ of finite type. Then:
		$$\big(A\underset{r, (\delta,\Omega)}{\rtimes} \mathbb{G}\big)\otimes \mathcal{K}(L^2(\mathbb{G}))\cong \big(A\otimes \mathcal{K}(L^2(\mathbb{G})\big)\underset{r, \widetilde{\delta}}{\rtimes} \mathbb{G}.$$
	\end{pro}
	\begin{proof}
		Put $B:=A\underset{r, (\delta,\Omega)}{\rtimes} \mathbb{G}$, which is a $\widehat{\mathbb{G}}^{cop}$-C$^*$-algebra with action $\widehat{\mathbb{G}}^{cop} \overset{\widehat{\delta}}{\curvearrowright} B$. Using the usual version of Takesaki-Takai duality for quantum groups, we can write $B\otimes \mathcal{K}(L^2(\mathbb{G}))\cong \big(B\underset{r,\widehat{\delta}}{\rtimes} \widehat{\mathbb{G}}^{cop}\big)\underset{r, \widehat{\widehat{\delta}}}{\rtimes} \mathbb{G}$. Next, using the twisted version of Takesaki-Takai duality from Theorem \ref{theo.TwistTakesakiTakai} we have the following:
		$$\big(B\underset{r,\widehat{\delta}}{\rtimes} \widehat{\mathbb{G}}^{cop}\big)\underset{r, \widehat{\widehat{\delta}}}{\rtimes} \mathbb{G}=\big(\big(A\underset{r, (\delta,\Omega)}{\rtimes} \mathbb{G}\big)\underset{r,\widehat{\delta}}{\rtimes} \widehat{\mathbb{G}}^{cop}\big)\underset{r, \widehat{\widehat{\delta}}}{\rtimes} \mathbb{G}\cong\big(A\otimes \mathcal{K}(L^2(\mathbb{G})\big)\underset{r, \widetilde{\delta}}{\rtimes} \mathbb{G}.$$
	\end{proof}
	
	This trick helps to establish a twisted version of the well known \emph{Baaj-Skandalis duality} \cite{SkandalisUnitaries,VergniouxThesis}.
		\begin{note}
		``Takesaki-Takai duality'' for quantum groups is also referred as ``Baaj-Skandalis duality'' in the literature. We prefer to reserve the latter terminology for such duality at the level of $KK$-theory.
	\end{note}
	
	\begin{theo}[Twisted Baaj-Skandalis duality]\label{theo.TwistedBaajSkandalis}
		Let $(\mathbb{G}, A, \delta, \Omega)$ and $(\mathbb{G}, B, \vartheta, \Omega)$ be two twisted dynamical systems with respect to a given $2$-cocycle $\Omega$ of finite type on $\mathbb{G}$. Then there exists a canonical group isomorphism:
		$$J^{\Omega}_{\mathbb{G}}: KK^{\mathbb{G}_{\Omega}}(A, B)\overset{\sim}{\longrightarrow} KK^{\widehat{\mathbb{G}}^{cop}}(A\underset{r, (\delta,\Omega)}{\rtimes} \mathbb{G}, B\underset{r, (\vartheta,\Omega)}{\rtimes} \mathbb{G})\mbox{,}$$
		which is compatible with the Kasparov product, that is, if $(\mathbb{G}, C, \nu, \Omega)$ is another twisted dynamical system with respect to $\Omega$, then we have:
		$$J^{\Omega}_{\mathbb{G}}(\mathcal{X}\underset{C}{\otimes}\mathcal{Y})=J^{\Omega}_{\mathbb{G}}(\mathcal{X})\underset{C\underset{r, (\nu,\Omega)}{\rtimes} \mathbb{G}}{\otimes} J^{\Omega}_{\mathbb{G}}(\mathcal{Y})\mbox{ and } J^{\Omega}_{\mathbb{G}}(\mathbbold{1}_A)=\mathbbold{1}_{A\underset{r, (\delta,\Omega)}{\rtimes} \mathbb{G}}\mbox{,}$$
		for all $\mathcal{X}\in KK^{\mathbb{G}_{\Omega}}(A, C)$ and $\mathcal{Y}\in KK^{\mathbb{G}_{\Omega}}(C, B)$.
	\end{theo}
	\begin{proof}
		Given the twisted dynamical system $(\mathbb{G}, A, \delta, \Omega)$, put $A_1:=A\underset{r, \delta}{\rtimes} \mathbb{G}_{\Omega}$. Thus $(A_1, \widehat{\delta})$ is an object in $\mathscr{K}\mathscr{K}^{\widehat{\mathbb{G}_{\Omega}}^{cop}}$. It is known that $\widehat{\mathbb{G}_{\Omega}}$ and $\widehat{\mathbb{G}}$ are monoidally co-Morita equivalent in the sense of \cite{KennyComonoidalMorita}, \cite{KennyRegularSemiRegular}. The corresponding co-linking quantum groupoid takes the form $C^*_r(\mathbb{G}_{\Omega}) \oplus C^*_r(\mathbb{G}_{\Omega},\Omega^*) \oplus C^*_r(\mathbb{G}, \Omega) \oplus C^*_r(\mathbb{G})$. Next, following the notations from \cite[Section 2.4]{BaajCrespo}, we consider the \emph{exterior comultiplication} $\widehat{\Delta}_{ext}\equiv ({}_{\Omega}\widehat{\Delta}_{\Omega^*})_{11}^{2}: C^*_r(\mathbb{G}_{\Omega}) \longrightarrow \widetilde{M}(C^*_r(\mathbb{G}_{\Omega},\Omega^*)  \otimes C^*_r(\mathbb{G},\Omega))$. 
%KDC: 1) According to \cite[Proposition 2.27]{BaajCrespo} and Remarks \ref{rems.RegularityInBaajCrespo}, $\widehat{\Delta}_{ext}$ is implemented by $\Sigma (V^{\Omega})^*\Sigma$. Formula not correct, but we don't need it
%2) We did not introduce $C_r^*(\mathbb{G},\Omega^*)$, but it is more natural to take it simply the oppposite C*-algebra of $C_r^*(G,\Omega)$, since we don't need its concrete form anyway
Following \cite[Proposition 4.1]{BaajCrespo}, we consider the C$^*$-algebra:
				\begin{equation*}
				\begin{split}
					A_2:&=\overline{span}\{(id\otimes id\otimes \eta)(id_{A_1}\otimes \widehat{\Delta}_{ext}^{cop})\widehat{\delta}(a')\ |\  a' \in A_1, \eta\in\mathcal{B}(L^2(\mathbb{G}))_*\}\\
					&=\overline{span}\{(id\otimes id\otimes \eta)(\delta(a)\otimes 1)(1\otimes \widehat{\Delta}_{ext}^{cop}(x))\ |\  a'\in A, x\in C^*_r(\mathbb{G}_\Omega), \eta\in\mathcal{B}(L^2(\mathbb{G}))_*\}.
				\end{split}
				\end{equation*}
Since $\{(id\otimes \eta)\widehat{\Delta}_{ext}^{cop}(x)\mid x\in C^*_r(\mathbb{G}_{\Omega}),\eta\in \mathcal{B}(L^2(\mathbb{G}))_*\}$ is norm-dense in $C_r^*(\mathbb{G},\Omega)$, we have by construction that $A_2=A\underset{r, (\delta,\Omega)}{\rtimes} \mathbb{G}$. Now, $A_2$ is an object in $\mathscr{K}\mathscr{K}^{\widehat{\mathbb{G}}^{cop}}$ with $\widehat{\mathbb{G}}^{cop}$-action given by Proposition \ref{pro.TwistedDualAction}, which we still denote by $\widehat{\delta}$.
		
		In other words, the equivalence of triangulated categories $\mathscr{K}\mathscr{K}^{\widehat{\mathbb{G}_{\Omega}}^{cop}}\overset{J_{\widehat{\mathbb{G}_{\Omega}},\ \widehat{\mathbb{G}}^{cop}}}{\cong} \mathscr{K}\mathscr{K}^{\widehat{\mathbb{G}}^{cop}}$ from \cite[Section 4.5]{BaajCrespo} sends $A\underset{r, \delta}{\rtimes} \mathbb{G}_{\Omega}$ to $A\underset{r, (\delta,\Omega)}{\rtimes} \mathbb{G}$. An application of the equivariant Morita equivalence of \cite[Section 4.4]{BaajCrespo} yields that $(A\otimes \mathcal{K}(L^2(\mathbb{G})))\underset{r, \widetilde{\delta}}{\rtimes}\mathbb{G}_{\Omega}\cong A_1\underset{r, \widehat{\delta}}{\rtimes} \widehat{\mathbb{G}_{\Omega}}^{cop}\underset{r, \widehat{\widehat{\delta}}}{\rtimes} \mathbb{G}_{\Omega}$ in $\mathscr{K}\mathscr{K}^{\widehat{\mathbb{G}_{\Omega}}^{cop}}$ (here $\widetilde{\delta}$ is defined by Takesaki-Takai duality for $\mathbb{G}_{\Omega}$) is sent to $(A\otimes \mathcal{K}(L^2(\mathbb{G})))\underset{r, \widetilde{\delta}}{\rtimes} \mathbb{G}\cong A_2\underset{r, \widehat{\delta}}{\rtimes} \widehat{\mathbb{G}}^{cop}\underset{r, \widehat{\widehat{\delta}}}{\rtimes} \mathbb{G}$ in $\mathscr{K}\mathscr{K}^{\widehat{\mathbb{G}}^{cop}}$ (here $\widetilde{\delta}$ is defined through the twisted Takesaki-Takai from Theorem \ref{theo.TwistTakesakiTakai}) through $J_{\widehat{\mathbb{G}_{\Omega}},\ \widehat{\mathbb{G}}^{cop}}$. 
		
		Therefore, the twisted Baaj-Skandalis is obtained as follows:
		\begin{equation*}
		\begin{split}
			KK^{\mathbb{G}_{\Omega}}(A, B)&\cong KK^{\mathbb{G}_{\Omega}}(A\otimes \mathcal{K}(L^2(\mathbb{G}), B\otimes \mathcal{K}(L^2(\mathbb{G}))\\
			&\overset{J_{\mathbb{G}_{\Omega}}}{\cong} KK^{\widehat{\mathbb{G}_{\Omega}}^{cop}}((A\otimes \mathcal{K}(L^2(\mathbb{G}))\underset{r, \widetilde{\delta}}{\rtimes}\mathbb{G}_{\Omega}, (B\otimes \mathcal{K}(L^2(\mathbb{G}))\underset{r, \widetilde{\delta}}{\rtimes}\mathbb{G}_{\Omega})\\
			&\overset{J_{\widehat{\mathbb{G}_{\Omega}},\ \widehat{\mathbb{G}}^{cop}}}{\cong} KK^{\widehat{\mathbb{G}}^{cop}}((A\otimes \mathcal{K}(L^2(\mathbb{G}))\underset{r, \widetilde{\delta}}{\rtimes}\mathbb{G}, (B\otimes \mathcal{K}(L^2(\mathbb{G}))\underset{r, \widetilde{\delta}}{\rtimes}\mathbb{G})\\
			&\cong KK^{\widehat{\mathbb{G}}^{cop}}(A\underset{r, (\delta,\Omega)}{\rtimes} \mathbb{G}, B\underset{r, (\vartheta,\Omega)}{\rtimes} \mathbb{G}).
		\end{split}
		\end{equation*}
	\end{proof}
	
	As a consequence, we obtain a twisted version of the well-known \emph{descent map} for quantum groups \cite{VergniouxThesis}.
	\begin{cor}[Twisted descent map]
		Let $(\mathbb{G}, A, \delta, \Omega)$ and $(\mathbb{G}, B, \vartheta, \Omega)$ be two twisted dynamical systems with respect to a given $2$-cocycle $\Omega$ of finite type on $\mathbb{G}$. Then there exists a canonical group homomorphism:
		$$j^{\Omega}_{\mathbb{G}}: KK^{\mathbb{G}_{\Omega}}(A, B)\longrightarrow KK(A\underset{r, (\delta,\Omega)}{\rtimes} \mathbb{G}, B\underset{r, (\vartheta,\Omega)}{\rtimes} \mathbb{G})$$
		called \emph{twisted descent map (with respect to $\mathbb{G}$)}. Moreover, $j^{\Omega}_{\mathbb{G}}$ is compatible with the Kasparov product, that is, if $(\mathbb{G}, C, \nu, \Omega)$ is another twisted dynamical system with respect to $\Omega$, then we have:
		$$j^{\Omega}_{\mathbb{G}}(\mathcal{X}\underset{C}{\otimes}\mathcal{Y})=j^{\Omega}_{\mathbb{G}}(\mathcal{X})\underset{C\underset{r, (\nu,\Omega)}{\rtimes} \mathbb{G}}{\otimes} j^{\Omega}_{\mathbb{G}}(\mathcal{Y})\mbox{ and } j^{\Omega}_{\mathbb{G}}(\mathbbold{1}_A)=\mathbbold{1}_{A\underset{r, (\delta,\Omega)}{\rtimes} \mathbb{G}}\mbox{,}$$
		for all $\mathcal{X}\in KK^{\mathbb{G}_{\Omega}}(A, C)$ and $\mathcal{Y}\in KK^{\mathbb{G}_{\Omega}}(C, B)$. Moreover, we have $j^{\Omega}_{\mathbb{G}}= \mathcal{O}_{\mathbb{G}}\circ J^{\Omega}_{\mathbb{G}}$, where $\mathcal{O}_{\mathbb{G}}$ is the obvious forgetful functor.
	\end{cor}
	\begin{proof}
		Similarly as in Theorem \ref{theo.TwistedBaajSkandalis} we have $KK^{\mathbb{G}_{\Omega}}(A, B)\cong KK^{\widehat{\mathbb{G}}^{cop}}((A\otimes \mathcal{K}(L^2(\mathbb{G}))\underset{r, \widetilde{\delta}}{\rtimes}\mathbb{G}, (B\otimes \mathcal{K}(L^2(\mathbb{G}))\underset{r, \widetilde{\delta}}{\rtimes}\mathbb{G})$. Applying Baaj-Skandalis duality for $\mathbb{G}$, the latter is isomorphic to $KK^{\mathbb{G}}(A\otimes \mathcal{K}(L^2(\mathbb{G}), B\otimes \mathcal{K}(L^2(\mathbb{G}))$. Applying the ordinary descent map $j_{\mathbb{G}}$, we get $KK((A\otimes \mathcal{K}(L^2(\mathbb{G}))\underset{r, \widetilde{\delta}}{\rtimes}\mathbb{G}, (B\otimes \mathcal{K}(L^2(\mathbb{G}))\underset{r, \widetilde{\delta}}{\rtimes}\mathbb{G})$, which is isomophic to $KK(A\underset{r, (\delta,\Omega)}{\rtimes} \mathbb{G}, B\underset{r, (\vartheta,\Omega)}{\rtimes} \mathbb{G})$ thanks to the untwisting trick from Proposition \ref{pro.QPackerRaeburnUntwistTrick}. The map $j^{\Omega}_{\mathbb{G}}$ of the statement is obtained as result of these compositions.
	\end{proof}

%KDC: What precisely is this an application of? 
\section{\textsc{Application: quantum assembly map for permutation torsion-free discrete quantum groups}}\label{sec.QuantumAssemblyMapProj}
	\subsection{The Baum-Connes property for discrete quantum groups}\label{sec.BCQG}
	
	The framework for the formulation of the Baum-Connes property for discrete quantum groups following the approach of R. Meyer and R. Nest is based on triangulated categories and Bousfield localisation techniques. We refer to \cite{MeyerNest} or \cite{Jorgensen} for a complete presentation of the subject.
	
	Let $\widehat{\mathbb{G}}$ be a discrete quantum group and consider the corresponding equivariant Kasparov category, $\mathscr{K}\mathscr{K}^{\widehat{\mathbb{G}}}$, with canonical suspension functor denoted by $\Sigma$. Then $\mathscr{K}\mathscr{K}^{\widehat{\mathbb{G}}}$ is a triangulated category whose distinguished triangles are given by mapping cone triangles. The word \emph{homomorphism (resp.\ isomorphism)} will mean \emph{homomorphism (resp.\ isomorphism) in the corresponding Kasparov category}; it will be a true homomorphism (resp.\ isomorphism) between $\widehat{\mathbb{G}}$-C$^*$-algebras or any Kasparov triple between $\widehat{\mathbb{G}}$-C$^*$-algebras (resp.\ any equivariant $KK$-equivalence between $\widehat{\mathbb{G}}$-C$^*$-algebras). Analogously, we can consider the equivariant Kasparov category $\mathscr{K}\mathscr{K}^{\mathbb{G}}$.
	
	Assume for the moment that $\widehat{\mathbb{G}}$ is \emph{torsion-free}. In that case, consider the usual complementary pair of localizing subcategories in $\mathscr{K}\mathscr{K}^{\widehat{\mathbb{G}}}$, $(\mathscr{L}_{\widehat{\mathbb{G}}}, \mathscr{N}_{\widehat{\mathbb{G}}})$. Denote by $(L,N)$ the canonical triangulated functors associated to this complementary pair. More precisely we have that $\mathscr{L}_{\widehat{\mathbb{G}}}$ is defined as the \emph{localizing subcategory of $\mathscr{K}\mathscr{K}^{\widehat{\mathbb{G}}}$ generated by the objects of the form $Ind^{\widehat{\mathbb{G}}}_{\mathbb{E}}(C)=C\otimes c_0(\widehat{\mathbb{G}})$ with $C$ any C$^*$-algebra in the Kasparov category $\mathscr{K}\mathscr{K}$} and $\mathscr{N}_{\widehat{\mathbb{G}}}$ is defined as the \emph{localizing subcategory of objects which are isomorphic to $0$ in $\mathscr{K}\mathscr{K}$}: $\mathscr{L}_{\widehat{\mathbb{G}}}:=\langle\{Ind^{\widehat{\mathbb{G}}}_{\mathbb{E}}(C)=C\otimes c_0(\widehat{\mathbb{G}})\ |\ C\in Obj.(\mathscr{K}\mathscr{K})\}\rangle$ and $\mathscr{N}_{\widehat{\mathbb{G}}}=\{A\in Obj.(\mathscr{K}\mathscr{K}^{\widehat{\mathbb{G}}})\ |\ Res^{\widehat{\mathbb{G}}}_{\mathbb{E}}(A)=0\}$.
	%$$\mathscr{L}_{\widehat{\mathbb{G}}}:=\langle\{Ind^{\widehat{\mathbb{G}}}_{\mathbb{E}}(C)=c_0(\widehat{\mathbb{G}})\otimes C\ |\ C\in Obj.(\mathscr{K}\mathscr{K})\}\rangle,$$
	%$$\mathscr{N}_{\widehat{\mathbb{G}}}=\{A\in Obj.(\mathscr{K}\mathscr{K}^{\widehat{\mathbb{G}}})\ |\ Res^{\widehat{\mathbb{G}}}_{\mathbb{E}}(A)=0\}.$$%=\{A\in Obj.(\mathscr{K}\mathscr{K}^{\widehat{\mathbb{G}}})\ |\ L(A)=0\}.$$
	
	If $\widehat{\mathbb{G}}$ is \emph{not} torsion-free, then a technical property lacked in the literature in order to define a suitable complementary pair. The natural candidate used in the related works (see for instance \cite{MeyerNestTorsion} and \cite{VoigtBaumConnesAutomorphisms}) is given by the following localizing subcategories of $\mathscr{K}\mathscr{K}^{\widehat{\mathbb{G}}}$:
	$$\mathscr{L}_{\widehat{\mathbb{G}}}:=\langle\{C\otimes T\underset{r}{\rtimes} \mathbb{G}\ |\  C\in Obj.(\mathscr{K}\mathscr{K})\mbox{, } T\in\text{Tor}(\widehat{\mathbb{G}})\}\rangle,$$
	$$\mathscr{N}_{\widehat{\mathbb{G}}}:=\mathscr{L}^{\dashv}_{\widehat{\mathbb{G}}}=\{A\in Obj(\mathscr{K}\mathscr{K}^{\widehat{\mathbb{G}}})\ |\ KK^{\widehat{\mathbb{G}}}(L, A)=0\mbox{, $\forall$ $L\in Obj(\mathscr{L}_{\widehat{\mathbb{G}}})$}\}.$$
	
	\begin{rem}
			We put $\widehat{\mathscr{L}}_{\widehat{\mathbb{G}}}:=\langle\{T\otimes C\ |\  C\in Obj.(\mathscr{K}\mathscr{K})\mbox{, } T\in\text{Tor}(\widehat{\mathbb{G}})\}\rangle$, so that we have $\widehat{\mathscr{L}}_{\widehat{\mathbb{G}}}\rtimes \mathbb{G}=\mathscr{L}_{\widehat{\mathbb{G}}}$ by definition. Similarly we put $\widehat{\mathscr{N}}_{\widehat{\mathbb{G}}}:=\mathscr{N}_{\widehat{\mathbb{G}}}\rtimes \widehat{\mathbb{G}}$.
	\end{rem}
	
In \cite{YukiBCTorsion}, Y. Arano and A. Skalski have showed that these two subcategories form indeed a complementary pair of localizing subcategories in $\mathscr{K}\mathscr{K}^{\widehat{\mathbb{G}}}$ (note that, by \cite[Proposition 2.9]{MeyerNest} any complementary pair $(\mathscr{L},\mathscr{N})$ must have $\mathscr{N} = \mathscr{L}^{\dashv}$). By Baaj-Skandalis duality, one then also obtains that the subcategories $\widehat{\mathscr{L}}_{\widehat{\mathbb{G}}}$ and $\widehat{\mathscr{N}}_{\widehat{\mathbb{G}}}$ form a complementary pair of localising subcategories in $\mathscr{K}\mathscr{K}^{\mathbb{G}}$. 

The key step is a generalization of the \emph{Green-Julg isomorphism}. In this section, we will provide a different proof of this generalization in the case of projective torsion, see Section \ref{sec.TwistedGreenJulg}.

%As an application of our cleftness property obtained in Theorem \ref{theo.vNCleftProjectiveTorsion}, we are going to prove in  that $(\mathscr{L}_{\widehat{\mathbb{G}}}, \mathscr{N}_{\widehat{\mathbb{G}}})$ is indeed complementary when $\mathscr{L}_{\widehat{\mathbb{G}}}$ is defined only in terms of torsion actions of $\mathbb{G}$ of projective type. It consists in generalizing the

%KDC: commented the next section, as this is now redundant.	
%	\begin{note}\label{note.DiracHom}
%		The following nomenclature is useful. Given $A\in Obj.(\mathscr{K}\mathscr{K}^{\widehat{\mathbb{G}}})$ consider a $(\mathscr{L}_{\widehat{\mathbb{G}}}, \mathscr{N}_{\widehat{\mathbb{G}}})$-triangle associated to $A$, say $\Sigma N(A)\longrightarrow L(A)\overset{D}{\longrightarrow} A\longrightarrow N(A)$. We know from the theory of Bousfield localisations that such triangles are distinguished and unique up to isomorphism. The homomorphism $D:L(A)\longrightarrow A$ is called \emph{Dirac homomorphism for $A$}. In particular, we consider the Dirac homomorphism for $\mathbb{C}$ (as trivial $\widehat{\mathbb{G}}$-C$^*$-algebra), $D_{\mathbb{C}}:L(\mathbb{C})\longrightarrow \mathbb{C}$. We refer to $D_{\mathbb{C}}$ simply as \emph{Dirac homomorphism}.
%	\end{note}

\subsection{Two-sided crossed products}\label{sec.TwoSidedCrossedProducts}
	In this section we give some technical tools necessary for the next section. In order to define an assembly map for discrete quantum groups, we need to define a suitable pair of adjoint functors on $\mathscr{K}\mathscr{K}^{\widehat{\mathbb{G}}}$ taking into account the torsion phenomena of $\widehat{\mathbb{G}}$. For classical discrete groups, the torsion phenomena are completely described in terms of finite subgroups. Hence, the induction and restriction functors will provide such an adjunction. In the quantum case, the torsion is described in terms of $\mathbb{G}$-C$^*$-algebras, so the induction-restriction approach is no longer valid since finite discrete quantum groups do not exhaust the torsion phenomena for $\widehat{\mathbb{G}}$. Moreover, the torsion objects in the Kasparov category are also $\mathbb{G}$-C$^*$-algebras. In this sense, we need a construction encoding both the induction process from the classical setting and the \emph{diagonal} action with respect to $\mathbb{G}$. We call it a \emph{two-sided crossed product} and it is already used in \cite{YukiBCTorsion}, see also \cite[Section 2.6]{NikshychVainerman} for an algebraic precursor.

	\begin{defi}
		Let $\mathbb{G}$ be a compact quantum group. If $(B, \beta)$ is a left $\mathbb{G}$-C$^*$-algebra and $(A, \alpha)$ is a right $\mathbb{G}$-C$^*$-algebra, then the two-sided crossed product of $B$ and $A$ by $\mathbb{G}$, denoted by $B\underset{r, \beta}{\rtimes}\mathbb{G}\underset{r, \alpha}{\ltimes}A$, is the C$^*$-algebra defined by:
		$$B\underset{r, \beta}{\rtimes}\mathbb{G}\underset{r, \alpha}{\ltimes}A:=C^*\langle ((id\otimes \lambda)\beta(B)\otimes 1)(1\otimes \widehat{\lambda}(c_0(\widehat{\mathbb{G}}))\otimes 1)(1\otimes (\rho\otimes id)(\alpha(A)))\rangle\subset \mathcal{L}_{B\otimes A}(B\otimes L^2(\mathbb{G})\otimes A).$$
	\end{defi}
	\begin{rem}
		First, to lighten the notations we will omit the representations $\lambda$, $\widehat{\lambda}$ and $\rho$ in the definition of $B\underset{r, \beta}{\rtimes}\mathbb{G}\underset{r, \alpha}{\ltimes}A$, and note that $\rho(x) = U_{\mathbb{G}}xU_{\mathbb{G}}$ for $x\in C(\mathbb{G})$. We also write $\alpha_U(x) = (U_{\mathbb{G}}\otimes id)\alpha(x)(U_{\mathbb{G}}\otimes id)$ for $x\in A$. Next, it is easy to show that $B\underset{r, \beta}{\rtimes}\mathbb{G}\underset{r, \alpha}{\ltimes}A=\overline{span}\{(\beta(B)\otimes 1)(1\otimes c_0(\widehat{\mathbb{G}})\otimes 1)(1\otimes \alpha_U(A))\}$, cf.\ Lemma \ref{lem.TwistedCrossedProd}. From now on we will use these two descriptions of $B\underset{r, \beta}{\rtimes}\mathbb{G}\underset{r, \alpha}{\ltimes}A$ interchangeably. As a consequence, we see that the maps $B\longrightarrow\mathcal{L}_{B\otimes A}(B\otimes L^2(\mathbb{G})\otimes A)$, $A\longrightarrow\mathcal{L}_{B\otimes A}(B\otimes L^2(\mathbb{G})\otimes A)$ and $c_0(\widehat{\mathbb{G}})\longrightarrow\mathcal{L}_{B\otimes A}(B\otimes L^2(\mathbb{G})\otimes A)$ given by $b\mapsto \beta(b)\otimes 1$, $a\mapsto 1\otimes \alpha_U(a)$ and $x\mapsto 1\otimes \widehat{\lambda}(x)\otimes 1$ respectively, send $B$, $A$ and $c_0(\widehat{\mathbb{G}})$ respectively onto non-degenerate C$^*$-subalgebras of $M(B\underset{r, \beta}{\rtimes}\mathbb{G}\underset{r, \alpha}{\ltimes}A)$.
	\end{rem}
	
\begin{rem}
In \cite{YukiBCTorsion}, a universal version $B\underset{\beta}{\rtimes}\mathbb{G}\underset{\alpha}{\ltimes}A$ of the double crossed product is introduced. It is not hard to see that our definition is compatible with theirs, in the sense that there is a natural surjective $*$-homomorphism: 
\begin{equation}\label{EqUnivProp2cross}
\pi_{r}: B\underset{\beta}{\rtimes}\mathbb{G}\underset{\alpha}{\ltimes}A \rightarrow B\underset{r,\beta}{\rtimes}\mathbb{G}\underset{r,\alpha}{\ltimes}A.
\end{equation}
Indeed, take some univeral representations of $B, A$ on Hilbert spaces $H_B,H_A$. Then one simply needs to observe that the representations $\pi: b \mapsto \beta(b)\otimes 1$ and $\theta:a\mapsto 1\otimes (\rho\otimes \mathrm{id})\alpha(a)$ of resp.\ $B$ and $A$ on $H = H_B \otimes L^2(\mathbb{G})\otimes H_A$, together with the unitary representation $U := V_{24}$ of $\mathbb{G}$,  form a covariant representation in the sense of \cite[Section 3]{YukiBCTorsion}: it follows immediately from the formulas in Theorem \ref{theo.RRegularRepresentation} that: 
\[
U_{12}(\pi(b)\otimes 1)U_{12}^* = (\pi\otimes \mathrm{id})\beta(b),\qquad U_{12}^*(\rho(a)\otimes 1)U_{12} = (\theta\otimes \mathrm{id})\alpha^{\mathrm{op}}(a). 
\]
The universal property of \cite[Proposition 3.2]{YukiBCTorsion} now provides \eqref{EqUnivProp2cross}. 

It is furthermore not hard to see that $\pi_r$ will be an isomorphism \emph{if $A$ is finite-dimensional}: By e.g.\ \cite[Theorem 5.31]{KennyActions}, $B\underset{\beta}{\rtimes} \mathbb{G} = B\underset{r,\beta}{\rtimes} \mathbb{G}$, and then both    $B\underset{\beta}{\rtimes} \mathbb{G}\underset{\alpha}{\ltimes}A$ and $B\underset{r,\beta}{\rtimes} \mathbb{G}\underset{\alpha}{\ltimes} A$ are implemented on the algebraic tensor product vector space $(B\underset{r,\beta}{\rtimes} \mathbb{G})\odot A$, so $\pi_r$ must be an isomorphism.
\end{rem}

%	\begin{pro}
%		Let $\mathbb{G}$ be a compact quantum group. Let $(B, \beta)$ be a right $\mathbb{G}$-C$^*$-algebra and $(A, \alpha)$ a left $\mathbb{G}$-C$^*$-algebra. If $\varphi$ is any $\mathbb{G}$-equivariant state on $A$, then there is a (non-degenerate) conditional expectation $E_\varphi: B\underset{r, \beta}{\rtimes}\mathbb{G}\underset{r, \alpha}{\ltimes}A\longrightarrow B\underset{r, \beta}{\rtimes}\mathbb{G}$ such that $E_\varphi\big((\beta(b)\otimes 1)(1\otimes x\otimes 1)(1\otimes \alpha_U(a))\big)= \varphi(a)\beta(b)(1\otimes x)$, for all $b\in B$, $a\in A$ and $x\in c_0(\widehat{\mathbb{G}})$.
%	\end{pro}
%	\begin{proof}
%		Since $\varphi$ is a state, the formula of the statement defines a  contractive completely positive map. Moreover, a straightforward computation shows that $E_\varphi$ is a $B\underset{r, \beta}{\rtimes}\mathbb{G}$-bimodule map. Finally, let $\{e_n\}_{n\in\mathbb{N}}$ be an approximate unit for $ B\underset{r, \beta}{\rtimes}\mathbb{G}\underset{r, \alpha}{\ltimes}A$. For each $n\in\mathbb{N}$, we write without loss of generality $e_n=(\beta(b_n)\otimes 1)(1\otimes x_n\otimes 1)(1\otimes \alpha_U(a_n))$, with $b_n\in B$, $a_n\in A$ and $x_n\in c_0(\widehat{\mathbb{G}})$. Then it is clear that $E_\varphi(e_n)\overset{n\rightarrow \infty}{\longrightarrow} 1$ strictly. %for all $z\in B\underset{r, \beta}{\rtimes}\mathbb{G}\underset{r, \alpha}{\ltimes}A$.
%	\end{proof}

	As for usual crossed products, we can show that the two-sided crossed product construction is functorial. More precisely, we have the following:
	\begin{pro}
		Let $\mathbb{G}$ be a compact quantum group. Let $(B, \beta)$, $(B', \beta')$ be left $\mathbb{G}$-C$^*$-algebras and $(A, \alpha)$, $(A', \alpha')$ right $\mathbb{G}$-C$^*$-algebras. 
		\begin{enumerate}[i)]
			\item If $\phi: B\longrightarrow M(B')$ is a non-degenerate $\mathbb{G}$-equivariant $*$-homomorphism, then there exists a non-degenerate $*$-homomorphism:
		$$\phi\rtimes id\ltimes id: B\underset{r, \beta}{\rtimes}\mathbb{G}\underset{r, \alpha}{\ltimes}A\longrightarrow M(B'\underset{r, \beta'}{\rtimes}\mathbb{G}\underset{r, \alpha}{\ltimes}A)$$
		such that $\phi\rtimes id\ltimes id\big((\beta(b)\otimes 1)(1\otimes x\otimes 1)(1\otimes\alpha_U(a))\big)=(\beta'(\phi(b))\otimes 1)(1\otimes x\otimes 1)(1\otimes\alpha_U(a))$, for all $b\in B$, $a\in A$ and $x\in c_0(\widehat{\mathbb{G}})$.
			\item If $\psi: A\longrightarrow M(A')$ is a non-degenerate $\mathbb{G}$-equivariant $*$-homomorphism, then there exists a non-degenerate $*$-homomorphism:
		$$id\rtimes id\ltimes \psi: B\underset{r, \beta}{\rtimes}\mathbb{G}\underset{r, \alpha}{\ltimes}A\longrightarrow M(B\underset{r, \beta}{\rtimes}\mathbb{G}\underset{r, \alpha'}{\ltimes}A')$$
		such that $id\rtimes id\ltimes \psi\big((\beta(b)\otimes 1)(1\otimes x\otimes 1)(1\otimes\alpha_U(a))\big)=(\beta(b)\otimes 1)(1\otimes x\otimes 1)(1\otimes\alpha_U'(\psi(a)))$, for all $b\in B$, $a\in A$ and $x\in c_0(\widehat{\mathbb{G}})$.
		\end{enumerate}
	\end{pro}
		
	If a torsion action of projective type of $\mathbb{G}$ is involved in a two-sided crossed product, then we can give an alternative description of the latter, which is useful for our purpose. Recall first from Lemma \ref{lem.DecompVTand} that, given a  torsion action of projective type $\delta$ on $T = \mathcal{B}(H)$ with implementing $\Omega^*$-representation $u$ for some $\Omega$, we can construct the $\Omega$-representation $u^{\circ}$ on $\overline{H}$ with associated coaction $\overline{\delta}$ on $T^{op} \cong \mathcal{B}(\overline{H})$. By Remark \ref{rem.switchOmOmStar} we can view $u^\circ$ as an $\Omega = (\Omega^*)^*$-representation of $\mathbb{G}_\Omega$, and correspondingly $\delta_{u^{\circ}}$ as a left action of $\mathbb{G}_\Omega$ on $T^{op}$. 

%By Proposition \ref{theo.UnitaryImplementation} there exists a unitary representation of $\mathbb{G}_\Omega$, $V_{T^{op}}\in \mathcal{B}(L^2(T^{op}))\otimes C(\mathbb{G}_\Omega)$, implementing $\delta_{u^{\circ}}$. The analogous argument as in Lemma \ref{lem.DecompVTand} yields that 
%\begin{equation}\label{EqVTop}
%V_{T^{op}}=u^{\circ}_{13}u_{23}.
%\end{equation}

%KDC: it's not clear if $V_{T^{op}}$, similarly defined as before, is really $u^{\circ}_{13}u_{23}$. But we don't need this, we only need to know that $u^{\circ}_{13}u_{23}$ is an ordinary $\mathbb{G}_{\Omega}$-representation, which is obvious. 

	\begin{pro}\label{pro.2CrossedProdT}
		Let $\mathbb{G}$ be a compact quantum group. Let $(T, \delta)$ be a torsion action of projective type of $\mathbb{G}$. Let $u$ be an $\Omega^*$-representation of $\mathbb{G}$ implementing $\delta$ for some $2$-cocycle $\Omega$ (necessarily of finite type). Let $(B, \beta)$ be a $\mathbb{G}$-C$^*$-algebra. Then $B\underset{r, \beta}{\rtimes}\mathbb{G}\underset{r, \overline{\delta}}{\ltimes}T^{op}\cong (B\otimes T^{op})\underset{\widetilde{\beta}}{\rtimes} \mathbb{G}_\Omega$, where $\widetilde{\beta}:=Ad_{u^\circ_{23}}\circ \beta_{13}$.
	\end{pro}
We use here that the twisted quantum group $\mathbb{G}_{\Omega}$ is again a compact quantum group by Theorem \ref{theo.TwistCompact}. 
	\begin{proof}
		First of all, it is straightforward to check that $\widetilde{\beta}:=Ad_{u^\circ_{23}}\circ \beta_{13}$ is an action of $\mathbb{G}_\Omega$ on $B\otimes T^{op}$. For given $b\in B$ and $x\in T^{op}$ we have:
	\begin{equation*}
	\begin{split}
		(id\otimes {}_{\Omega}\Delta_{\Omega^*})\widetilde{\beta}(b\otimes x)&=(id\otimes {}_{\Omega}\Delta_{\Omega^*})\big(u^\circ_{23}(b_{(0)}\otimes x \otimes b_{(1)})(u^\circ_{23})^*\big)\\
		&=\Omega_{34}\Delta_3(u^\circ_{23})(b_{(0)}\otimes x\otimes \Delta(b_{(1)}))\Delta_3(u^\circ_{23})^*\Omega^*_{34}\\
		&\overset{(*)}{=}u^{\circ}_{23}u^\circ_{24}(b_{(0)}\otimes x\otimes \Delta(b_{(1)}))(u^\circ_{24})^*(u^{\circ}_{23})^*\\
		&\overset{(**)}{=}u^{\circ}_{23}u^\circ_{24}(b_{(0)(0)}\otimes x\otimes b_{(0)(1)}\otimes b_{(1)})(u^\circ_{24})^*(u^{\circ}_{23})^*\\
		&=(\widetilde{\beta}\otimes id)\big(u^\circ_{23}(b_{(0)}\otimes x \otimes b_{(1)})(u^\circ_{23})^*\big)=(\widetilde{\beta}\otimes id)\widetilde{\beta}(b\otimes x),
	\end{split}
	\end{equation*}
	where in $(*)$ we used the fact that $u^\circ$ is an $\Omega$-representation of $\mathbb{G}$ and in $(**)$ we use the fact that $\beta$ is a $\mathbb{G}$-action on $B$.
	
	Next, recall that $u^\circ$ implements $\overline{\delta}$. Then on the one hand, by Remark \ref{rem.OmegaInnerDynamicalSystem} we have $\mathbb{G}\underset{r, \overline{\delta}}{\ltimes} T^{op}\cong C^*_r(\Omega,\mathbb{G})\otimes T^{op}$ (here the identification is given by conjugating with $u^{\circ}_{21}(U_{\mathbb{G}}\otimes 1)$). On the other hand, by Remark \ref{rem.switchOmOmStar} we view $u^\circ$ as a $\Theta^*:=(\Omega^*)^*$-representation of $\mathbb{G}_\Omega$ so that, by applying again Remark \ref{rem.OmegaInnerDynamicalSystem}, we obtain that $T^{op}\underset{r, Ad_{u^\circ}}{\rtimes} \mathbb{G}_\Omega\cong T^{op}\otimes C^*_r(\mathbb{G}_\Omega, \Theta)=T^{op}\otimes C^*_r(\mathbb{G}_\Omega, \Omega^*)$ (here the identification is given by conjugating with $(u^\circ)^*$). Recall from Remark \ref{rem.TwistedCStarGOmega} that $C^*_r(\mathbb{G}_\Omega, \Omega^*)=JC^*_r(\mathbb{G}, \Omega)J$. Therefore, by Remark \ref{rems.rightdualreg} we see that $T^{op}\underset{r, Ad_{u^\circ}}{\rtimes} \mathbb{G}_\Omega$ is $*$-isomorphic to $\mathbb{G}\underset{r, \overline{\delta}}{\ltimes} T^{op}$ by introducing an extra conjugation with $X_{\Omega}U_{\mathbb{G}}$ (and a flip map). This allows to conclude the isomorphism of the statement. Indeed, since $U_{\mathbb{G}} X_{\Omega}U_{\mathbb{G}} \in L^{\infty}(\mathbb{G})'$, we compute
	\begin{equation*}
	\begin{split}
		(B&\otimes T^{op})\underset{\widetilde{\beta}}{\rtimes} \mathbb{G}_\Omega=\overline{span}\{(id\otimes id\otimes \lambda)\widetilde{\beta}(B\otimes T^{op})(1\otimes 1\otimes c_0(\widehat{\mathbb{G}_{\Omega}}))\}\\
		&=\overline{span}\{(id\otimes id\otimes \lambda)u^\circ_{23}(B_{(0)}\otimes T^{op}\otimes B_{(1)})(u^\circ_{23})^*(1\otimes 1\otimes c_0(\widehat{\mathbb{G}_{\Omega}}))\}\\
		&\overset{Ad_{(u^\circ_{23})^*}}{\cong}\overline{span}\{(id\otimes id\otimes \lambda)(B_{(0)}\otimes T^{op}\otimes B_{(1)})(u^\circ_{23})^*(1\otimes 1\otimes c_0(\widehat{\mathbb{G}_{\Omega}}))u^\circ_{23}\}\\
		&=\overline{span}\{(id\otimes id\otimes \lambda)(B_{(0)}\otimes 1\otimes B_{(1)})(1\otimes \underbracket[0.8pt]{T^{op}\otimes 1})(1\otimes \underbracket[0.8pt]{(u^\circ)^*(1\otimes c_0(\widehat{\mathbb{G}_{\Omega}}))u^\circ})\}\\
		&\underset{(*)}{=} \overline{span}\{(id\otimes id\otimes \lambda)(B_{(0)}\otimes 1\otimes B_{(1)})(1\otimes \underbracket[0.8pt]{T^{op}\otimes C^*_r(\mathbb{G}_\Omega, \Omega^*)})\}\\
		&\cong \overline{span}\{(id\otimes id\otimes \lambda)(B_{(0)}\otimes 1\otimes B_{(1)})(1\otimes T^{op}\otimes JC^*_r(\mathbb{G}, \Omega)J)\}\\
		&\overset{Ad_{\Sigma_{23}(U_{\mathbb{G}}X_{\Omega}U_{\mathbb{G}})_3}}{\cong}\overline{span}\{(id\otimes \lambda\otimes id)(B_{(0)}\otimes B_{(1)}\otimes 1)(1\otimes \underbracket[0.8pt]{U_{\mathbb{G}}C^*_r(\Omega,\mathbb{G})U_{\mathbb{G}}\otimes T^{op}})\}\\
		&\overset{Ad_{u_{21}^{\circ}}}{\underset{(**)}{\cong}}\overline{span}\{(id\otimes \lambda\otimes id)(B_{(0)}\otimes B_{(1)}\otimes 1)   (1\otimes \underbracket[0.8pt]{c_0(\widehat{\mathbb{G}})\otimes 1})(1\otimes \underbracket[0.8pt]{\rho\otimes 1}) (1\otimes \underbracket[0.8pt]{\Sigma (u^{\circ})^*(1\otimes T^{op})u^\circ \Sigma})\}\\
		&=\overline{span}\{((id\otimes \lambda)\beta(B)\otimes 1)(1\otimes c_0(\widehat{\mathbb{G}})\otimes 1)(1\otimes (\rho\otimes 1)\overline{\delta}(T^{op}))\}=B\underset{r, \beta}{\rtimes}\mathbb{G}\underset{r, \overline{\delta}}{\ltimes}T^{op},
	\end{split}
	\end{equation*}
	where in $(*)$ and $(**)$ we have used the identifications $T^{op}\underset{r, Ad_{u^\circ}}{\rtimes} \mathbb{G}_\Omega\cong T^{op}\otimes C^*_r(\mathbb{G}_\Omega, \Omega^*)$ and $\mathbb{G}\underset{r, \overline{\delta}}{\ltimes} T^{op}\cong C^*_r(\Omega,\mathbb{G})\otimes T^{op}$ explained above, respectively. 
	\end{proof}

	The two-sided crossed product construction can also be defined for Hilbert modules in a similar way as we do for usual crossed products.
	\begin{defi}\label{defi.DescentTwoSidedCrossedProd}
		Let $\mathbb{G}$ be a compact quantum group. Let $(B, \beta)$ be a left $\mathbb{G}$-C$^*$-algebra and $(A, \alpha)$ a right $\mathbb{G}$-C$^*$-algebra. If $(E, \delta_E)$ is a $\mathbb{G}$-equivariant Hilbert $B$-module, we define the two-sided crossed product of $E$ and $A$ by $\mathbb{G}$, denoted by $E\underset{r, \delta_E}{\rtimes}\mathbb{G}\underset{r, \alpha}{\ltimes}A$, as the following Hilbert $B\underset{r, \beta}{\rtimes}\mathbb{G}\underset{r, \alpha}{\ltimes}A$-module $E\underset{r, \delta_E}{\rtimes}\mathbb{G}\underset{r, \alpha}{\ltimes}A:=E\underset{B}{\otimes} B\underset{r, \beta}{\rtimes}\mathbb{G}\underset{r, \alpha}{\ltimes}A$.
	\end{defi}
	
	As for the usual crossed products, the embeddings of $E\cong \mathcal{K}_{B}(B, E)$ and $B\cong \mathcal{K}_B(B)$ into $\mathcal{K}_{B}(B\oplus E)$ induce an embedding of $E\underset{r, \delta_E}{\rtimes}\mathbb{G}\underset{r, \alpha}{\ltimes}A$ into $\mathcal{K}_{B}(B\oplus E)\underset{r}{\rtimes}\mathbb{G}\underset{r, \alpha}{\ltimes}A$. In this way we have the following (see for instance \cite[Lemme 5.2]{VergniouxThesis} for a proof):
	\begin{pro}
		Let $\mathbb{G}$ be a compact quantum group. Let $(B, \beta)$ be a left $\mathbb{G}$-C$^*$-algebra and $(A, \alpha)$ a right $\mathbb{G}$-C$^*$-algebra. If $(E, \delta_E)$ is a $\mathbb{G}$-equivariant Hilbert $B$-module, then $\mathcal{K}_{B}(E)\underset{r}{\rtimes}\mathbb{G}\underset{r, \alpha}{\ltimes}A\cong \mathcal{K}_{B\underset{r, \beta}{\rtimes}\mathbb{G}\underset{r, \alpha}{\ltimes}A}\big(E\underset{r, \delta_E}{\rtimes}\mathbb{G}\underset{r, \alpha}{\ltimes}A\big)$.
	\end{pro}
	
	Following similar arguments as for usual crossed products (see for instance \cite[Proposition 5.3]{VergniouxThesis} for more details), it is easy to show that Definition \ref{defi.DescentTwoSidedCrossedProd} above passes also at the level of Kasparov triples. More precisely, we have the following:
	\begin{pro}
		Let $\mathbb{G}$ be a compact quantum group. Let $(B, \beta)$, $(B', \beta')$ be left $\mathbb{G}$-C$^*$-algebras and $(A, \alpha)$ a right $\mathbb{G}$-C$^*$-algebra. If $((E, \delta_E), \pi, F)$ is a $\mathbb{G}$-equivariant Kasparov $(B', B)$-module in $KK^{\mathbb{G}}(B', B)$, then the triple $\big(E\underset{r, \delta_E}{\rtimes}\mathbb{G}\underset{r, \alpha}{\ltimes}A, \pi\rtimes id\ltimes id, F\underset{B}{\otimes} id\big)$ defines a Kasparov $\big(B'\underset{r, \beta'}{\rtimes}\mathbb{G}\underset{r, \alpha}{\ltimes}A, B\underset{r, \beta}{\rtimes}\mathbb{G}\underset{r, \alpha}{\ltimes}A\big)$-module in $KK\big(B'\underset{r, \beta'}{\rtimes}\mathbb{G}\underset{r, \alpha}{\ltimes}A, B\underset{r, \beta}{\rtimes}\mathbb{G}\underset{r, \alpha}{\ltimes}A\big)$.
	\end{pro}
	
	Finally, as for usual crossed products, the two-sided crossed product functor intertwines the suspension of $\mathbb{G}$-C$^*$-algebras and transforms mapping cone triangles into mapping cone triangles. In other words, the functor $(\ \cdot\ )\underset{r}{\rtimes}\mathbb{G}\underset{r, \alpha}{\ltimes}A$ preserves \emph{semi-split extensions} i.e.\ extensions of $\mathbb{G}$-equivariant C$^*$-algebras that split through a $\mathbb{G}$-equivariant completely positive contractive linear section, see for instance \cite{MeyerCatAspects}; and the class of all triangles in $\mathscr{K}\mathscr{K}^{\mathbb{G}}$ isomorphic to mapping cone triangles is the same as the class of all triangles in $\mathscr{K}\mathscr{K}^{\mathbb{G}}$ isomorphic to extension triangles (see for instance \cite[Lemma 1.2.3.7]{RubenThesis}). In conclusion, we have obtained that, for a fixed right $\mathbb{G}$-C$^*$-algebra $(A_0, \alpha_0)$, the association $(B, \beta)\mapsto B\underset{r, \beta}{\rtimes}\mathbb{G}\underset{r, \alpha_0}{\ltimes}A_0$, for all left $\mathbb{G}$-C$^*$-algebra $(B, \beta)$ defines a \emph{triangulated functor} $j_{\mathbb{G}, A_0}:=(\ \cdot\ )\underset{r}{\rtimes}\mathbb{G}\underset{r, \alpha_0}{\ltimes}A_0:\mathscr{K}\mathscr{K}^\mathbb{G}\longrightarrow\mathscr{K}\mathscr{K}$.

\subsection{Twisted Green-Julg isomorphism}\label{sec.TwistedGreenJulg}

	First of all, let us recall briefly the Green-Julg isomorphism for compact quantum groups (see \cite{VergniouxThesis} for more details). If $C$ is a C$^*$-algebra equipped with the trivial action of $\mathbb{G}$, then we have that $\Psi: KK^\mathbb{G}(C, B)\overset{\sim}{\longrightarrow} KK(C, B\underset{r, \beta}{\rtimes} \mathbb{G})$, for all $\mathbb{G}$-C$^*$-algebra $(B,\beta)$. Since $C$ is equipped with the trivial action of $\mathbb{G}$, then $C\underset{r}{\rtimes} \mathbb{G}\cong C\otimes c_0(\widehat{\mathbb{G}})$ and we have a natural $*$-homomorphism:
	$$
		\begin{array}{rcclccl}
			\phi_C:&C& \longrightarrow & C\otimes  c_0(\widehat{\mathbb{G}}), &c& \longmapsto & \phi(c):=c\otimes p_0\mbox{,}
		\end{array}
	$$
	where $p_0:=(id\otimes h_{\mathbb{G}})(V_{\mathbb{G}})\in c_0(\widehat{\mathbb{G}})$ is the canonical projection onto the subspace of invariant vectors of $(V_\mathbb{G}, L^2(\mathbb{G}))$. In this way we obtain a Kasparov triple $[\phi_C]\in KK(C, C\otimes c_0(\widehat{\mathbb{G}}))$. The Green-Julg isomorphism is given precisely by $\Psi(\mathcal{X}):=\phi_C^*(j_{\mathbb{G}}(\mathcal{X}))=[\phi_C]\underset{C\otimes  c_0(\widehat{\mathbb{G}})}{\otimes} j_{\mathbb{G}}(\mathcal{X})$, for all $\mathcal{X}\in KK^\mathbb{G}(C, B)$. It is also possible to give an explicit expression of its inverse. Given any C$^*$-algebra $C$ in $\mathscr{K}\mathscr{K}$, we denote by $\tau(C)$ the same C$^*$-algebra $C$ equipped with the trivial action of $\mathbb{G}$ and so we regard it as an object in $\mathscr{K}\mathscr{K}^\mathbb{G}$. In this way, we define the Kasparov triple $\mathcal{E}_B:=[(B\otimes L^2(\mathbb{G}), \pi_r, 0)]\in KK^\mathbb{G}\Big(\tau\big(B\underset{r, \beta}{\rtimes} \mathbb{G}\big), B\Big)$, where $\pi_r$ denotes the canonical representation of $B\underset{r, \beta}{\rtimes} \mathbb{G}$ in $B\otimes L^2(\mathbb{G})$. The action of $\mathbb{G}$ on $B\otimes L^2(\mathbb{G})$ is defined as the tensor product action of $\beta$ with the action of $\mathbb{G}$ on $L^2(\mathbb{G})$ induced by the unitary $\Sigma \check{V}_{\mathbb{G}}\Sigma=(U_{\mathbb{G}}\otimes 1)V_{\mathbb{G}}(U_{\mathbb{G}}\otimes 1)$ (see \cite{VergniouxThesis} for the precise definitions). Then we have $\Psi^{-1}(\mathcal{Y})=\tau(\mathcal{Y})\underset{\tau\big(B\underset{r, \beta}{\rtimes} \mathbb{G}\big)}{\otimes} \mathcal{E}_B$, for all $\mathcal{Y}\in KK(C, B\underset{r, \beta}{\rtimes} \mathbb{G})$. In other words, the Green-Julg isomorphism can be rephrased by saying that the functors $\mathscr{K}\mathscr{K}^\mathbb{G}\overset{j_{\mathbb{G}}}{\longrightarrow} \mathscr{K}\mathscr{K} \mbox{ and } \mathscr{K}\mathscr{K}\overset{\tau}{\longrightarrow} \mathscr{K}\mathscr{K}^\mathbb{G}$
	are adjoint: $\tau$ is a left adjoint of $j_\mathbb{G}$. Precisely, the unit of the adjunction is given by $\eta_C:=[\phi_C]$ and the counit by $\varepsilon_B:=\mathcal{E}_B$, for all $C\in Obj(\mathscr{K}\mathscr{K})$ and all $B\in Obj(\mathscr{K}\mathscr{K}^\mathbb{G})$. 
	
	The goal of this section is to generalise these constructions when $C$ is replaced by an object of the form $C\otimes T\in\widehat{\mathscr{L}}_{\widehat{\mathbb{G}}}$, where $(T, \delta)$ is a torsion action of $\mathbb{G}$ of projective type. Recall that a torsion action of projective type of $\mathbb{G}$, $(T, \delta)$, means simply that $T=\mathcal{M}_k(\mathbb{C})$ for some $k\in\mathbb{N}$ and that $\delta$ is ergodic such that $T$ is not $\mathbb{G}$-Morita equivalent to $\mathbb{C}$. We fix a state $\varphi_T=\text{Tr}(\varrho\ \cdot)$ on $T$ (recall Section \ref{sec.Conven}). Recall as well that, by virtue of Theorem \ref{theo.vNCleftProjectiveTorsion}, $\delta$ is implemented by an $\Omega^*$-representation of $\mathbb{G}$, say $u$, for some (measurable)  $2$-cocycle $\Omega$ on $\mathbb{G}$. The $2$-cocycle $\Omega$ is necessarily of finite type (recall Definition \ref{def.fintype}). Hence $\mathbb{G}_\Omega$ is again a compact quantum group by Theorem \ref{theo.TwistCompact}. Following Equation \eqref{eq.oppcoact}, we denote by $(T^{op}, \overline{\delta})$ the corresponding opposite twisted dynamical system. In this case, $\overline{\delta}$ is implemented (in the sense of \eqref{eq.deltau}) by an $\Omega$-representation of $\mathbb{G}$ that we denote by $u^\circ$. The  representation of $\mathbb{G}$ on $L^2(T)$ implementing $\delta$ according to Proposition \ref{theo.UnitaryImplementation} is denoted by $V_T$. Given such a projective torsion action of $\mathbb{G}$, we define the following triangulated functors:
	$$
		\begin{array}{rcclccl}
			j_{\mathbb{G}, T}:&\mathscr{K}\mathscr{K}^\mathbb{G}& \longrightarrow & \mathscr{K}\mathscr{K}, &(B, \beta)& \longmapsto & j_{\mathbb{G}, T}(B,\beta):= B\underset{r, \beta}{\rtimes}\mathbb{G}\underset{r, \overline{\delta}}{\ltimes}T^{op},
		\end{array}
	$$
	$$
		\begin{array}{rcclccl}
			\tau_{ T}:&\mathscr{K}\mathscr{K}& \longrightarrow & \mathscr{K}\mathscr{K}^\mathbb{G}, &C& \longmapsto & \tau_{T}(C):= (C\otimes T, id\otimes \delta).
		\end{array}
	$$
	
	We are going to show that $\tau_T$ is a left adjoint of $j_{\mathbb{G}, T}$ for every torsion action of projective type $(T, \delta)$ of $\mathbb{G}$. To do so we start by showing an appropriate equivalence of triangulated categories between $\mathscr{K}\mathscr{K}^\mathbb{G}$ and $\mathscr{K}\mathscr{K}^{\mathbb{G}_\Omega}$. Then the adjunction between $\tau_T$ and $j_{\mathbb{G}, T}$ will result from the usual Green-Julg isomorphism applied to $\mathbb{G}_\Omega$.
	
	Let us consider the following triangulated functors:
	$$
		\begin{array}{rcclccl}
			\Pi_{T^{op}}:&\mathscr{K}\mathscr{K}^\mathbb{G}& \longrightarrow & \mathscr{K}\mathscr{K}^{\mathbb{G}_\Omega}, &(B, \beta)& \longmapsto & \Pi_{T^{op}}(B,\beta):= (B\otimes T^{op}, \widetilde{\beta}:=Ad_{u^\circ_{23}}\circ \beta_{13}),
		\end{array}
	$$
	$$
		\begin{array}{rcclccl}
			\Pi_{T}:&\mathscr{K}\mathscr{K}^{\mathbb{G}_\Omega}& \longrightarrow & \mathscr{K}\mathscr{K}^{\mathbb{G}}, &(C, \gamma)& \longmapsto & \Pi_{T}(C,\gamma):= (C\otimes T, \widetilde{\gamma}:=Ad_{u_{23}}\circ \gamma_{13}).
		\end{array}
	$$
	
	First of all, observe that these functors are well defined. On the one hand, given $(B, \beta)\in Obj(\mathscr{K}\mathscr{K}^\mathbb{G})$,  we have proved in Proposition \ref{pro.2CrossedProdT} that $\widetilde{\beta}:=Ad_{u^\circ_{23}}\circ \beta_{13}$ is an action of $\mathbb{G}_\Omega$ on $B\otimes T^{op}$. So  $\Pi_{T^{op}}(B,\beta)\in Obj(\mathscr{K}\mathscr{K}^{\mathbb{G}_\Omega})$. A similar computation yields that if $(C, \gamma)\in Obj(\mathscr{K}\mathscr{K}^{\mathbb{G}_\Omega})$, then $\widetilde{\gamma}:=Ad_{u_{23}}\circ \gamma_{13}$ is an action of $\mathbb{G}$ on $C\otimes T$. So  $\Pi_{T}(C,\gamma)\in Obj(\mathscr{K}\mathscr{K}^{\mathbb{G}})$. 
	
	On the other hand, given two objects $(B_1, \beta_1), (B_2, \beta_2)\in Obj(\mathscr{K}\mathscr{K}^\mathbb{G})$ and a Kasparov triple $\mathcal{X}\in KK^\mathbb{G}(B_1, B_2)$, then $\Pi_{T^{op}}(\mathcal{X})$ is given by the right exterior tensor product of Kasparov triples with respect to $T^{op}$ i.e. $\Pi_{T^{op}}(\mathcal{X})=\mathcal{X}\otimes T^{op}\in KK^{\mathbb{G}_\Omega}(B_1\otimes T^{op}, B_2\otimes T^{op})$ (if $\mathcal{X}$ is represented by the $\mathbb{G}$-equivariant Hilbert $B_2$-module $E$ with action $\delta_{E}$, then $\Pi_{T^{op}}(\mathcal{X})$ is represented by the Hilbert $B_2\otimes T^{op}$-module $E\otimes T^{op}$ with action of $\mathbb{G}_\Omega$ given by $Ad_{v_{23}}\circ (\delta_E)_{13}$). Similarly, $\Pi_{T}(\mathcal{Y})=\mathcal{Y}\otimes T$, for all $\mathcal{Y}\in KK^{\mathbb{G}_\Omega}(C_1, C_2)$ with $(C_1, \gamma_1), (C_2, \gamma_2)\in Obj(\mathscr{K}\mathscr{K}^{\mathbb{G}_{\Omega}})$. Clearly, both $\Pi_{T^{op}}$ and $\Pi_{T}$ intertwine the suspensions of each category. Moreover, they transform mapping cone triangles into mapping cone triangles. This is true by the following general fact: if $\phi: A\longrightarrow B$ is a homomorphism between C$^*$-algebras and $D$ is any other C$^*$-algebra, then we have that $C_{\phi}\otimes D\cong C_{\phi\otimes id}$ induced by the canonical identification $C_0\big((0,1], B\big)\otimes D\cong C_0\big((0,1], B\otimes D\big)$, $f\otimes d\mapsto \Big(t\mapsto f(t)\otimes d\Big)$. If in addition $\phi$ is a $\mathbb{G}$-equivariant homomorphism between the $\mathbb{G}$-C$^*$-algebras $(A,\alpha)$ and $(B,\beta)$, then $C_\phi$ is a $\mathbb{G}$-C$^*$-algebra with action $\gamma\big((a,h)\big):=(\alpha(a), \beta\circ h)$, for all $(a,h)\in C_\phi$. In this way, it is straightforward to check that given a $\mathbb{G}$-equivariant homomorphism between two $\mathbb{G}$-C$^*$-algebras $(B_1,\beta_1)$ and $(B_2, \beta_2)$, say $\phi: B_1\longrightarrow B_2$, then the isomorphism $C_{\phi}\otimes T^{op}\cong C_{\phi\otimes id}$ is $\mathbb{G}_\Omega$-equivariant. As a consequence, the functor $\Pi_{T^{op}}$ preserves mapping cone triangles; and similarly for $\Pi_T$. In conclusion, both $\Pi_{T^{op}}$ and $\Pi_T$ are well defined triangulated functors.
	
	\begin{lem}\label{lem.EquivalenceGGOmega}
		Following the previous notations, the pair of functors $(\Pi_{T^{op}}, \Pi_T)$ defines an equivalence of triangulated categories between $\mathscr{K}\mathscr{K}^\mathbb{G}$ and $\mathscr{K}\mathscr{K}^{\mathbb{G}_\Omega}$.
	\end{lem}
	\begin{proof}
		It only remains to show that $\Pi_{T}\circ \Pi_{T^{op}}\cong id_{\mathscr{K}\mathscr{K}^\mathbb{G}}$ and $\Pi_{T^{op}}\circ \Pi_{T}\cong id_{\mathscr{K}\mathscr{K}^{\mathbb{G}_\Omega}}$. 
		
		On the one hand, given an object $(B, \beta)\in Obj(\mathscr{K}\mathscr{K}^\mathbb{G})$, we have $\Pi_{T}( \Pi_{T^{op}}(B, \beta))=B\otimes T^{op}\otimes T\overset{id\otimes \Sigma}{\cong} B\otimes T\otimes T^{op}$ equipped with the $\mathbb{G}$-action $\widetilde{\widetilde{\beta}}=Ad_{u_{24}}\circ Ad_{u^\circ_{34}}\circ \beta_{14}$. By identifying $T\otimes T^{op}\cong \mathcal{B}(L^2(T))=\mathcal{K}(L^2(T))$ along with the $\mathbb{G}$-action $Ad_{V_T}$, where $V_T=u_{13}u^\circ_{23}$ as in Lemma \ref{lem.DecompVTand}, we obtain that $\Pi_{T}( \Pi_{T^{op}}(B, \beta))$ is $\mathbb{G}$-equivariantly Morita equivalent to $(B, \beta)$; and this identification is natural. So $\Pi_{T}\circ \Pi_{T^{op}}\cong id_{\mathscr{K}\mathscr{K}^\mathbb{G}}$. On the other hand, given an object $(C, \gamma)\in Obj(\mathscr{K}\mathscr{K}^{\mathbb{G}_\Omega})$, we have $\Pi_{T^{op}}( \Pi_{T}(C, \gamma))=C\otimes T\otimes T^{op}\overset{id\otimes \Sigma}{\cong} B\otimes T^{op}\otimes T$ equipped with the $\mathbb{G}$-action $\widetilde{\widetilde{\gamma}}=Ad_{u^\circ_{24}}\circ Ad_{u_{34}}\circ \gamma_{14}$. By identifying $T^{op}\otimes T=(T\otimes T^{op})^{op}\cong \mathcal{B}(L^2(T))^{op}\cong \mathcal{B}(L^2(T^{op}))=\mathcal{K}(L^2(T^{op}))$ along with the $\mathbb{G}_{\Omega}$-action $Ad_{V_{T^{op}}}$, where we define the (ordinary) $\mathbb{G}_{\Omega}$-representation $V_{T^{op}}=u^\circ_{13}u_{23}$, we obtain that $\Pi_{T^{op}}( \Pi_{T}(C, \gamma))$ is $\mathbb{G}_\Omega$-equivariantly Morita equivalent to $(C, \gamma)$; and this identification is natural. So $\Pi_{T^{op}}\circ \Pi_{T}\cong id_{\mathscr{K}\mathscr{K}^{\mathbb{G}_\Omega}}$.
	\end{proof}
%KDC: made some small changes in the explanation
	\begin{rem}
Using the theory of (bi)Galois objects for compact quantum groups \cite{BichonVaesRijdt}, there was proven in \cite[Section 8]{VoigtBaumConnesFree} an equivalence of triangulated categories $\mathscr{K}\mathscr{K}^{\mathbb{G}_1}\cong \mathscr{K}\mathscr{K}^{\mathbb{G}_2}$ when the compact quantum groups $\mathbb{G}_1$ and $\mathbb{G}_2$ are \emph{monoidally equivalent}. In our setting, with $\Omega$ a $2$-cocycle of finite type for the compact quantum group $\mathbb{G}$, it is rather the discrete quantum group \emph{duals} $\widehat{\mathbb{G}}$ and $\widehat{\mathbb{G}_{\Omega}}$ that are monoidally equivalent (\cite{KennyGaloisObjTwistings}). We then obtain equivalences:
\begin{equation}\label{EqEquiAlt}
\mathscr{K}\mathscr{K}^{\mathbb{G}} \cong \mathscr{K}\mathscr{K}^{\widehat{\mathbb{G}}} \cong \mathscr{K}\mathscr{K}^{\widehat{\mathbb{G}_{\Omega}}} \cong \mathscr{K}\mathscr{K}^{\mathbb{G}_{\Omega}},
\end{equation}
where the outer equivalences are by Baaj-Skandalis duality, and where in the middle we use the extension of the results of \cite[Section 8]{VoigtBaumConnesFree} to the (regular) locally compact quantum group setting \cite[Theorem 4.36]{BaajCrespo}. For lack of space, we refrain from showing that \eqref{EqEquiAlt} agrees with the equivalence we obtain - the argument is based on the observation that, up to matrix amplification, $T\rtimes \mathbb{G}$ gives a Galois object $C^*(\mathbb{G},\Omega)$ for the discrete quantum group $\widehat{\mathbb{G}}^{cop}$.
	\end{rem}

	\begin{theo}[Twisted Green-Julg isomorphism]\label{theo.TwistedGreenJulg}
		Let $\mathbb{G}$ be a compact quantum group. Let $(T, \delta)$ be a torsion action of projective type of $\mathbb{G}$. Let $u$ be an $\Omega^*$-representation of $\mathbb{G}$ implementing $\delta$ for some $2$-cocycle $\Omega$ (necessarily of finite type). Then $\tau_T: \mathscr{K}\mathscr{K}\longrightarrow \mathscr{K}\mathscr{K}^{\mathbb{G}}$ is a left adjoint of $j_{\mathbb{G}, T}: \mathscr{K}\mathscr{K}^{\mathbb{G}}\longrightarrow \mathscr{K}\mathscr{K}$ as triangulated functors. More precisely,
		$$\Psi_T: KK^\mathbb{G}(C\otimes T, B)\overset{\sim}{\longrightarrow} KK\big(C, B\underset{r, \beta}{\rtimes}\mathbb{G}\underset{r, \overline{\delta}}{\ltimes}T^{op}\big),$$
		for all $C\in Obj(\mathscr{K}\mathscr{K})$ and $(B, \beta)\in Obj(\mathscr{K}\mathscr{K}^\mathbb{G})$.
	\end{theo}
	\begin{proof}
		Since the $2$-cocycle $\Omega$ is necessarily of finite type, the twisted quantum group $\mathbb{G}_{\Omega}$ is compact by Theorem \ref{theo.TwistCompact}. Given a C$^*$-algebra $C\in Obj(\mathscr{K}\mathscr{K})$ and a $\mathbb{G}$-C$^*$-algebra $(B, \beta)\in Obj(\mathscr{K}\mathscr{K}^\mathbb{G})$, the previous lemma allows to write the following:
		$$KK^\mathbb{G}(C\otimes T, B)\cong KK^{\mathbb{G}_\Omega}(C\otimes T\otimes T^{op}, B\otimes T^{op})\cong KK^{\mathbb{G}_\Omega}(C, B\otimes T^{op}).$$
		
		Next, by applying the usual Green-Julg isomorphism we obtain that:
		$KK^{\mathbb{G}_\Omega}(C, B\otimes T^{op})\overset{\Psi}{\cong} KK(C, (B\otimes T^{op})\underset{\widetilde{\beta}}{\rtimes} \mathbb{G}_\Omega).$ To conclude, we observe that $ (B\otimes T^{op})\underset{\widetilde{\beta}}{\rtimes} \mathbb{G}_\Omega\cong B\underset{r, \beta}{\rtimes}\mathbb{G}\underset{r, \overline{\delta}}{\ltimes}T^{op}$ by virtue of Proposition \ref{pro.2CrossedProdT}. Therefore $\Psi_T=\Psi\circ\Pi_{T^{op}}$.
	\end{proof}
	
\begin{rem}
By the twisted Green-Julg isomorphism, obtained in Theorem \ref{theo.TwistedGreenJulg} for the specific case of torsion actions of projective type and in \cite[Theorem 4.5]{YukiBCTorsion} for general torsion actions, one easily obtains that $(T, \delta)$ is a compact object in $\mathscr{K}\mathscr{K}^{\mathbb{G}}$, that is, the functor $KK^{\mathbb{G}}(T,\ \cdot\ )$ is compatible with countable direct sums. Indeed, if $\{(B_n, \beta_n)\}_{n\in\mathbb{N}}$ is a countable family of $\mathbb{G}$-C$^*$-algebras, the twisted Green-Julg isomorphism with $C:=\mathbb{C}$ gives: 
		\begin{equation*}
		KK^\mathbb{G}(T, \underset{n\in\mathbb{N}}{\bigoplus}B_n)\cong KK\big(\mathbb{C}, \big(\underset{n\in\mathbb{N}}{\bigoplus}B_n\big)\underset{r, \beta}{\rtimes}\mathbb{G}\underset{r, \overline{\delta}}{\ltimes}T^{op}\big)\cong K_0\big(\big(\underset{n\in\mathbb{N}}{\bigoplus}B_n\big)\underset{r, \beta}{\rtimes}\mathbb{G}\underset{r, \overline{\delta}}{\ltimes}T^{op}\big),
		\end{equation*}
and both the $K_0$ functor and the two-sided crossed product functor $(\cdot)\underset{r}{\rtimes}\mathbb{G}\underset{r, \overline{\delta}}{\ltimes}T^{op}$ are compatible with countable direct sums.

	We believe that this property will be useful to study the equivariant Kasparov category $\mathscr{K}\mathscr{K}^{\mathbb{G}}$ from a geometrical and topological perspective according to works by I. Dell'Ambrogio and his collaborators (see for example \cite{AmbrogioBCSpec}, \cite{AmbrogioMeyerSpec}). For instance, the above compactness result yields that the subcategory $\mathscr{P}_{D(\mathbb{G})}:=\langle\{T\ |\ T\in\text{Tor}(\widehat{D(\mathbb{G})})\}\rangle$ is a compactly generated tensor triangular subcategory of $\mathscr{K}\mathscr{K}^{D(\mathbb{G})}$ when $\mathbb{G}$ is finite (note that we need to consider the Drinfeld double construction to provide a tensor structure on the Kasparov category). In particular, $(\mathscr{P}_{D(\mathbb{G})}, \mathscr{P}_{D(\mathbb{G})}^{\dashv})$ is a complementary pair of localizing subcategories in $\mathscr{K}\mathscr{K}^{D(\mathbb{G})}$ as a consequence of Brown representability theorem. Hence it will be interesting to compute its spectrum, $Spc(\mathscr{P}_{D(\mathbb{G})})$, in the sense of Balmer \cite{BalmerSpc} and to make a connection with the Baum-Connes property for $\widehat{\mathbb{G}}$. 
\end{rem}
\bibliographystyle{acm}
\bibliography{quantumBCI}

\begin{comment}
\vspace{1cm}
\textsc{K. De Commer, Mathematics department, Vrije Universiteit Brussel, Belgium.} 

\textit{E-mail address:} \textbf{\texttt{Kenny.De.Commer@vub.be}}

\vspace{0.5cm}

\textsc{R. Martos, Department of Mathematical Sciences, University of Copenhagen, Denmark.} 

\textit{E-mail address:} \textbf{\texttt{ruben.martos@math.ku.dk}}

\vspace{0.5cm}
\textsc{R. Nest, Department of Mathematical Sciences, University of Copenhagen, Denmark.} 

\textit{E-mail address:} \textbf{\texttt{rnest@math.ku.dk}}
\end{comment}

 \end{document}